\documentclass{article}
\usepackage[utf8]{inputenc}
\usepackage{enumitem}

\newpage

\newcommand{\red}[1]{{\color{red}{#1}}} 

\usepackage{amsmath}
\usepackage{amssymb}
\usepackage{amsfonts}
\usepackage{amsthm}
\usepackage{bbm}
\usepackage[mathscr]{eucal}
\usepackage{a4wide}
\usepackage{authblk}
\usepackage{float}
\usepackage{comment}
\usepackage{enumitem}

\usepackage[normalem]{ulem}

\usepackage{mathtools}
\mathtoolsset{showonlyrefs}

\usepackage{hyperref}
\hypersetup{
	colorlinks = true,
	urlcolor = {magenta},
	citecolor = {red},
	linkcolor = {blue}
}
\usepackage{doi}
\usepackage[square,numbers]{natbib}

\usepackage{xcolor}
\usepackage{todonotes}

\newcommand{\RNum}[1]{\uppercase\expandafter{\romannumeral #1\relax}}

\newcommand{\pr}{\mathbb{P}}								

\newcommand{\Prob}[1]{\pr\left(#1\right)}					

\newcommand{\CProb}[2]{\pr\left(#1 \mid #2\right)}	

\newcommand{\E}{\mathbb{E}}								

\newcommand{\Exp}[1]{\E\left[#1\right]}					

\newcommand{\CExp}[2]{\E\left[\left.#1\right|#2\right]}	

\newcommand{\plim}{\ensuremath{\stackrel{\pr}{\rightarrow}}}	
\newcommand{\dlim}{\ensuremath{\stackrel{d}{\rightarrow}}}		

\newcommand{\1}{\mathbbm{1}}								
\newcommand{\ind}[1]{\1_{\{#1\}}}	
\newcommand{\indE}[1]{\1_{#1}}					

\newcommand\numberthis{\addtocounter{equation}{1}\tag{\theequation}}

\allowdisplaybreaks

\newtheorem{theorem}{Theorem}[section]
\newtheorem{lemma}[theorem]{Lemma}

\newtheorem{proposition}[theorem]{Proposition}
\newtheorem{corollary}[theorem]{Corollary}
\newtheorem{conjecture}[theorem]{Conjecture}

\newtheorem{problem}[theorem]{Problem}

\theoremstyle{definition}
\newtheorem{definition}[theorem]{Definition}
\newtheorem{assumption}[theorem]{Assumption}

\newtheorem{remark}[theorem]{Remark}

\numberwithin{equation}{section}

\newcommand{\bd}{\mathbf d}

\newcommand{\bv}{\mathbf v}

\newcommand{\cA}{\mathcal A}

\newcommand{\cB}{\mathcal B}

\newcommand{\cC}{\mathcal C}
\newcommand{\bc}{\mathbf c}

\newcommand{\sD}{\mathscr D}
\newcommand{\cE}{\mathcal E}
\newcommand{\sF}{\mathscr F}
\newcommand{\cG}{\mathcal G}

\newcommand{\cL}{\mathcal L}
\newcommand{\sL}{\mathscr L}
\newcommand{\cM}{\mathcal M}
\newcommand{\cN}{\mathcal N}

\newcommand{\cT}{\mathcal T}

\newcommand{\cS}{\mathcal S}
\newcommand{\sS}{\mathscr S}

\newcommand{\cY}{\mathcal Y}

\newcommand{\T}{\mathbb T}
\newcommand{\sT}{\mathscr T}

\newcommand{\bt}{\mathbf t}

\newcommand{\N}{\mathbb N}
\newcommand{\R}{\mathbb R}
\newcommand{\G}{\mathbb G}

\newcommand{\sM}{\mathscr{M}}

\newcommand{\sX}{\mathscr{X}}

\newcommand{\bp}{\boldsymbol{p}}

\newcommand{\by}{\boldsymbol{y}}

\newcommand{\bB}{\boldsymbol{B}}
\newcommand{\bD}{\boldsymbol{D}}

\newcommand*{\be}{\begin{equation}}
	\newcommand*{\ee}{\end{equation}}
\newcommand*{\ba}{\begin{aligned}}
	\newcommand*{\ea}{\end{aligned}}

\newcommand{\eps}{\epsilon}

\newcommand{\cF}{\mathcal F}

\newcommand{\BP}{\mathrm{BP}}

\newcommand{\invisible}[1]{}

\usepackage{setspace}
\title{Multipartite random graphs with given degrees: local limit, revisiting the giant, distances}

\begin{document}
	\author{Neeladri Maitra\thanks{e-mail: nmaitra@illinois.edu}}\affil{Department of Mathematics,
University of Illinois at Urbana-Champaign}
    \date{}
	
	\maketitle

\begin{abstract}
We consider multipartite random graphs with given degree sequences, within and across different partitions. Under general assumptions, we prove the local limit of this graph is a multi-type branching process, establish that a giant component exists only when the local limit survives, and deduce that the typical distance is of logarithmic order in probability in the supercritical regime. Our analysis removes two major assumptions from \cite{gamarnik2015giant}, where the giant component problem for this model was first considered. In particular, we do not assume irreducibility of the local limit, and provide a general framework to extract giant components even when the limiting branching process is reducible, which we hope to be useful in other contexts. We also provide a new simpler survival criterion of multi-type branching processes, which we hope to be useful when direct calculation of the spectral radius of the offspring matrix may prove to be difficult.
	\end{abstract}

    \newpage

\tableofcontents

\newpage

\section{Introduction} 

The study of random graphs with a given degree sequence has a rich history. Enumeration results for random graphs with a given degree sequence go back to Bender and Canfield \cite{bender1978asymptotic}, refined later by Bollobás \cite{bollobas1980probabilistic} through probabilistic techniques. The advent of modern network science \cite{barabasi2013network,dorogovtsev2013evolution,newman2018networks,pastor2007evolution} predicting that real-world networks are typically `scale-free', motivated the need to study structural properties of such graphs under general assumptions on the degree sequences. The works that began exploration in this direction are by Molloy and Reed \cite{molloy1995critical,molloy1998size}, who started the study of the emergence of a giant connected component in these graphs. Their results later were generalized via a continuous-time approach by Janson and Luczak \cite{janson2009new}, and under a large-deviation estimate assumption by Bollobás and Riordan \cite{bollobas2015old}. For a comprehensive account of the rigorous random graph literature in this direction, we refer the reader to van der Hofstad's book \cite[Notes, Chapter 4]{van2024random}, or Dhara's thesis \cite{dhara2018critical}.

In many real-world networks, individuals are often separated into different communities, and thus it is crucial to model such networks using random graphs with a vertex set divided into prescribed community structures. Examples where multipartite networks have been used to model network data include plant-pollinator networks \cite{lampo2024structural}, security applications \cite{kent2016cyber}, human trafficking networks \cite{szekely2015building} and biological applications such as drug-target prediction \cite{zong2021drug} and medicine \cite{jafari2020unsupervised}. 

In this paper, we marry the notions of random graphs with a given degree sequence with multipartite random graphs, to study random multipartite graphs with a given degree sequence; i.e., we fix the degree sequences of inter and intra partition connections, and uniformly sample from graphs with such specified degree sequences. The model was first considered by Gamarnik and Misra \cite{gamarnik2015giant}, where under certain restrictive conditions, the authors prove the existence of a unique giant connected component, when a related branching process survives. Recent interest in machine learning problems, and particularly the problem of community detection, inferring the presence of multipartite structures in networks \cite{abbe2018community,arenas2008motif,10.1145/3071178.3071295,lancichinetti2009community,larremore2014efficiently,liu2016community}, motivated us to revisit this related model, with the goal of removing a couple of extra assumptions the authors of \cite{gamarnik2015giant} work under, and proving the result in as much generality as possible. Along the way we also develop a new survival condition for multi-type branching processes, which is interesting on its own. The following informal statement summarizes our main results.

\begin{theorem}[Informal; see Theorems \ref{thm:loc_lim}, \ref{thm:extinct_1dep}, \ref{thm:giants_MCMs} and \ref{thm:dist} in the following sections]
    Consider a sequence of uniformly sampled random multipartite graphs of growing sizes, with given inter and intra partition degree sequences, and where each vertex partition has a positive proportion of vertices. Under some standard regularity assumptions on the degree sequences, the local limit of the graph sequence exists as a branching process tree. Without any assumption on the irreducibility of the local limit, there is a simple (new) criterion that checks whether the limiting tree survives, and when it does survive, is tailored to extract linear sized connected components in the pre-limit via exploration processes. Finally, in the case of survival, the typical distance is of logarithmic order in probability.
\end{theorem}


\paragraph{Graph theoretic conventions.} To state our results formally, we begin with fixing a few standard graph theoretic conventions. A multigraph is a tuple $G=(V(G),E(G))$, where $V(G)$ is a countable set, and $E(G)$ is a multiset of size two multi-subsets of $V(G)$. Thus, the elements of $E(G)$ are of the form $\{u,v\}$ with $u,v \in V(G)$, possibly equal. We will say $V(G)$ is the set of \emph{vertices} of $G$ and $E(G)$ is the multiset of \emph{edges} of $G$. For a countable set $S$, throughout the paper we denote by $|S|$ its cardinality. 

For a graph $G$, we will further consider some special forms of the edge multiset $E(G)$. For example, the elements of $E(G)$ may come endowed with an \emph{orientation}, which we denote formally using an ordered pair notation $(u,v)$ where $u,v\in V(G)$, and say that the edge $(u,v)$ is \emph{directed} from $u$ to $v$, and denote this by $u \rightsquigarrow v$. Thus, if $E(G)$ is a multiset of \emph{ordered $2$-tuples} of elements from $V(G)$, we say that $G$ is a \emph{directed multigraph}, otherwise it is \emph{undirected}. In the undirected case, for an edge $\{u,v\}\in E(G)$ we also write $u \sim v$. An element of the form $\{x,x\}\in E(G)$ (or $(x,x)$ if directed) is called a \emph{self-loop} at the vertex $x$. In either the directed or the undirected case, if an edge appears more than once in the multiset $E(G)$, we say that there are \emph{multiple edges} between $u$ and $v$, where $u,v\in V(G)$ are the vertices in the edge.

A \emph{simple graph} is a multigraph $G=(V(G),E(G))$ with no multiple edges between any two vertices, the edges are undirected and no vertex has a self-loop at it. 
For a multigraph $G=(V(G),E(G))$ and $v\in V(G)$, the number
\begin{align*}
    \sum_{e\in E(H)}\ind{v\in e}\numberthis \label{eq:def_degree}
\end{align*}
is called the \emph{degree} of $v$ in $H$ and is denoted by $d_H(v)$. Here in the above sum, the summand corresponding to $e$ is appears $M_e \times M_{v,e}$ times, where $M_e$ is the multiplicity of $e$ in the multiset $E(G)$, while $M_{v,e}$ is the multiplicity of $v$ in the multiset $e$. 
In particular, for an undirected multigraph $G$, the contribution of self-loops at some vertex $v$, to its degree, is \emph{twice} their multiplicity in the edge multiset. For an undirected multigraph $G$ and $A\subset V(G)$, we additionally denote by $d_A(v)$ the degree of $v$ \emph{into} $A$, i.e.,
\begin{align*}
    d_A(v)=\sum_{u \in A}\ind{u \sim v},
\end{align*}
the summand corresponding to $\{u,v\}$ appearing its multiplicity many times.

These definitions can be extended naturally to the directed graphs, and we will talk about the \emph{indegree} of a vertex $v$, corresponding to the number of elements of the form $(u,v)$ in $E(G)$, or the \emph{outdegree} of $v$, corresponding to the number of elements of the form $(v,u)$ in $E(G)$, or the degree of $v$ \emph{to} $A$, corresponding to the number of edges of the form $(v,u)$ with $u \in A$, the degree of $v$ \emph{from} $A$, etc.; the countings, as usual, are done respecting multiplicities appropriately.

\subsection{Multipartite simple graphs with given degrees}
In this section we present our main results on uniform multipartite graphs with a given degree sequence. 

Fix $k\geq 1$ an integer. We will consider uniformly sampling random graphs, whose vertex set is split into $k$ disjoint parts, and degree statistics of within the same part, and across different parts are given. 

More specifically, we consider a vertex set partitioned as $V=V_1\cup \dots \cup V_k$. For each $i,j\in [k]$, let $\bd_i$ denote the degree sequence within $V_i$ and $\bd_{(i,j)}$ denote the degree sequence \emph{from $V_i$ to $V_j$}. This simply means:
\begin{itemize}
    \item For each $i \in [k]$, $\bd_i=(d^{(i)}_1,\dots,d^{(i)}_{|V_i|})$ is a sequence of non-negative integers with $\sum_{x=1}^{|V_i|}d^{(i)}_x$ even and $\max_{1\leq x \leq |V_i|}d^{(i)}_x\leq |V_i|$, such that there exists at least one simple graph $G$ on the vertex set $[|V_i|]$\footnote{Here and everywhere for any positive integer $N\geq 1$ by $[N]$ we denote the set $\{1,\dots,N\}$.} with the degree of vertex $x$ being $d^{(i)}_x$.
    \item For each $i,j\in [k]$ $\bd_{(i,j)}=(d^{(i,j)}_1,\dots,d^{(i,j)}_{|V_i|})$ is a sequence of non-negative integers with $\max_{1\leq x \leq |V_i|}d^{(i,j)}_{|V_i|}\leq |V_j|$, and such that
    \begin{align*}
        \sum_{x =1}^{|V_i|}d^{(i,j)}_x=\sum_{y=1}^{|V_j|}d^{(j,i)}_y,
    \end{align*}
    and such that there exists at least one simple bipartite graph $G$ on the bipartitioned vertex set $\{1,\dots,|V_i|\}\cup \{1,\dots,|V_j|\}$ with the degree of a vertex $x$ in $G$ being $d^{(i,j)}_x$ if $x\in \{1,\dots,|V_i|\}$ and it being $d^{(j,i)}_x$ if $x\in \{1,\dots,|V_j|\}$. 
\end{itemize}
\begin{remark}[Feasible degree sequences for existence of simple graphs]\label{rem:feasible} At this point it is important to remark about the exact conditions under which a non-negative integer sequence (respectively bi-degree sequence) can be realized as the degree sequence of a simple graph (respectively a simple bipartite graph), i.e., the vertices can be labelled in a way such that the degree of the $i$-th vertex corresponds to the $i$-th entry of the sequence.
    \begin{itemize}
        \item The Erd\H{o}s-Gallai theorem \cite{erdos1960grafok} states that for any sequence $(d_1,\dots,d_n)$ of non-negative integers with total sum even and such that $d_1\geq\dots \geq d_n$, it is realizable as the degree sequence of a simple graph on $n$ vertices if and only if for every $k=1,2,\dots,n$ $$\sum_{i=1}^k d_i\leq k(k-1)+\sum_{i=k+1}^n\min\{d_i,k\}.$$ Such a sequence $(d_1,\dots,d_n)$ is also called \emph{graphic}.
        \item The Gale-Ryser theorem (published independently by Gale \cite{gale1957theorem} and Ryser \cite{Ryser_1957}) states that a tuple of sequences $(a_1,\dots,a_n)$ and $(b_1,\dots,b_m)$ with $a_1\geq \dots \geq a_n$ is realizable as the bi-degree sequence of a simple bipartite graph on the vertex set $A\cup B$ with $|A|=n$ and $|B|=m$ and with $(a_1,\dots,a_n)$ (respectively $(b_1,\dots,b_m)$) being the degree sequence of the vertices of $A$ into $B$ (respectively $B$ into $A$), if and only if $\sum_{i=1}^n a_i=\sum_{j=1}^m b_j$, and for any $k=1,2,\dots,n$
        $$\sum_{i=1}^k a_i \leq \sum_{j=1}^m \min\{b_i,k\}.$$ Such a pair $((a_1,\dots,a_n),(b_1,\dots,b_m))$ is also called \emph{bigraphic}.
    \end{itemize}
\end{remark}

We will consider simple graphs $G=(V(G),E(G))$ on the vertex set $V(G)=V=V_1\cup\dots \cup V_k$ such that for any $i,j \in [k]$, the vertices in $V_i$ can be labeled by (i.e., can be bijectively mapped to) the set $\{1,\dots,|V_i|\}$, such that the degree of $x\in \{1,\dots,|V_i|\}$ in $V_j$ is $d^{(i,j)}_x$ if $j\neq i$ and it is $d^{(i)}_x$ if $j=i$. In other words, $(\bd_{(i,j)},\bd_{(j,i)})$ is the degree sequence of the bipartite graph on $V_i \cup V_j$ formed by removing all edges within either $V_i$ or $V_j$, while $\bd_{i}$ is the degree sequence of the subgraph of $G$ spanned by $V_i$. In particular, we will sample uniformly from the set of such graphs $G$, assuming always that it is non-empty. That is, we implicitly assume throughout the paper that for each $i,j\in [k]$ the degree sequence $\bd_i$ is graphic, and the bi-degree sequence $(\bd_{(i,j)},\bd_{(j,i)})$ is bigraphic. We need some further assumptions on our graphs and degrees, which we formulate in the next few sections, along with also presenting our main results. 

\subsection{Results on simple multipartite graphs with given degrees} 

Recall the sequences $\bd_i$ and $\bd_{(i,j)}$ for $i,j\in [k]$ from the last section. We will in general consider sequences $(\bd_i(n))_{n \geq 1}$ and $(\bd_{(i,j)}(n))_{n \geq 1}$ of these degree sequences indexed by $n \geq 1$. Here the $n$-th term of these sequences will correspond to degrees in a graph $G_n=(V(G_n),E(G_n))$ with $n$ vertices, i.e., $|V(G_n)|=n$, and our assumptions will be of asymptotic nature, as $n \to \infty$. Let us write $V_n=V(G_n)$ and $E_n=E(G_n)$. As before, for a fixed $k \geq  1$, the vertex set $V_n$ is partitioned as $V_n=\cN^{(1)}\cup \dots \cup \cN^{(k)}$, where the degree sequence within $\cN^{(i)}=\cN^{(i)}(n)$ is given by $\bd_i(n)$ and the degree sequence of vertices from $\cN^{(i)}$ into $\cN^{(j)}$ is given by $\bd_{(i,j)}(n)$. Our first assumption is that all the parts are \emph{comparable} in size asymptotically: 
\begin{assumption}\label{assump:part_i_size}
    Defining $|\cN^{(i)}|=N^{(i)}=N^{(i)}(n)$, we have 
\begin{align*}
    \lim_{n \to \infty}\frac{N^{(i)}}{n}\to \gamma^{(i)}, \numberthis \label{eq:part_i_size}
\end{align*}
for some $\gamma^{(i)} \in (0,1)$, with $\sum_{i=1}^k \gamma^{(i)}=1$.
\end{assumption}

It is useful to define the random variable $T_{n,\varnothing}$ with support $[k]$, where
\begin{align*}
    \Prob{T_{n,\varnothing}=i}=\frac{N^{(i)}}{n}. \numberthis \label{eq:def_T_nphi}
\end{align*}
The interpretation is that if we sample a vertex uniformly at random, then it falls in the partition $\cN^{(T_{n,\varnothing})}$. 

Next, we list our assumptions on the degrees. First, we collect all the information encoded by $(\bd_i(n))$ and $\bd_{(i,j)}(n)$ as a matrix $\bD(n)$, whose $(i,j)$-th entry is $\bD_{(i,j)}(n)=\bd_{(i,j)}(n)$ and whose diagonal entries are $\bD_{(i,i)}=\bd_i$. Note that each entry of $\bD(n)$ is a sequence of reals. In general we allow for the sequences $\bd_i(n)$ and $\bd_{(i,j)}(n)$ to be \emph{random}.

Following van der Hofstad \cite[Chapter 9]{van2024random}, we assume that the empirical distribution function of the different types of degrees from a given partition has a limit, with convergence of the expectation and the second moment of the associated marginals:

\begin{assumption}\label{assump:degree_regularity}
    For fixed $x_1,\dots,x_k \in \R_+$, for any $i \in [k]$, 
\begin{align*}
    \frac{1}{N^{(i)}}\sum_{x=1 }^{N^{(i)}}\ind{d^{(i,1)}_x(n)\leq x_1,\dots,d^{(i,k)}_x(n)\leq x_k}\to F^{(i)}(x_1,\dots,x_k), \numberthis \label{eq:empirical_limit}
\end{align*}
for some $k$-variate joint distribution function $F^{(i)}:\R^k \to [0,1]$. We further assume for any $i,j \in [k]$,
\begin{align*}
    \frac{1}{N^{(i)}}\sum_{x=1}^{N^{(i)}}\left(d^{(i,j)}_x(n)\right)^r \to \Exp{D_{(i, j)}^r}<\infty,\;\;\text{for}\;\;r=1,2,\numberthis \label{assump:regularity_1_2}
\end{align*}
where $D_{(i ,j)}$ is the random variable having the $j$-th marginal distribution of $F^{(i)}$,
\begin{align*}
    \Prob{D_{(i , j)}\leq x}=F_{(i , j)}(x):=\lim_{x_1,\dots,x_{j-1},x_{j+1},x_k \to \infty}F^{(i)}(x_1,\dots,x_{j-1},x,x_{j+1},\dots,x_k). \numberthis \label{eq:marginal_dist_degs}
\end{align*}
\end{assumption}Defining the integer valued random variable $D_{n,(i,j)}$ as, for any $m=0,1,2,\dots$
\begin{align*}
    \Prob{D_{n,(i,j)}=m}=\frac{1}{N^{(i)}}\sum_{x=1}^{N^{(i)}}\ind{d^{(i,j)}_x(n)=m},\numberthis \label{eq:the_RV_D_nij}
\end{align*}
the previous assumption says that the mean and the second moment of $D_{n,(i,j)}$ converges to those of $D_{(i,j)}$.

The model becomes sort of degenerate for our purposes, if we do not assume that there is at least one partition with linear number of edges either going inside or leaving towards another partition, and further that most vertices don't have degree $1$ or $2$, so we assume this:
\begin{assumption}\label{assump:one_good_deg_greater_2}
    There exists $i,j\in [k]$ such that $\Exp{D_{(i,j)}}>0$. Further, for any such pair $(i,j)$, we have $\Prob{D_{(i,j)}>2}>0$.
\end{assumption}

\begin{definition}[GOOD and BAD pairs]
    A pair $(i,j)\in [k]\times [k]$ is called GOOD if $\Exp{D_{(i,j)}}>0$, otherwise it is BAD. Note that $(i,j)$ is GOOD if and only if $(j,i)$ is GOOD.
\end{definition}

Recall the matrix $\bD(n)$ with entries $\bd_{(i,j)}(n)$. Let us denote by $\cG(\bD(n))$ the set of all simple graphs on the vertex set $V=\cN^{(1)}\cup \dots \cup \cN^{(k)}$ such that the degree sequence into $\cN^{(j)}$ of vertices in $\cN^{(i)}$ is given by $\bd_{(i,j)}(n)$. Throughout the paper, we will denote by $G_n=(V_n, E_n)$ a uniformly sampled graph from this set. 

\begin{remark}[Inducing uniform graphs within and across partitions]\label{rem:uniform_induce}
    Since conditional measures of uniforms are uniform, note that:
\begin{itemize}
    \item The subgraph $G_n(i)$ of $G_n$ spanned by only the vertices in $\cN^{(i)}$ is distributed as a uniform graph on $\cN^{(i)}$ with degree sequence $\bd_i(n)$.
    \item The bipartite subgraph on the vertex subset $\cN^{(i)}\cup \cN^{(j)}$ spanned by the edges going in between two vertex partitions $\cN^{(i)}$ and $\cN^{(j)}$ is distributed as a uniform bipartite random graph on this subset with a given bi-degree sequence $(\bd_{(i,j)}(n),\bd_{(j,i)}(n))$. 
\end{itemize}
\end{remark}

\begin{remark}[Enumeration]
Define for $i,j \in k$ such that $(i,j), (i,i)$ is GOOD,
\begin{align*}
    &\phi_i:=\frac{\Exp{D_{(i,i)}(D_{(i,i)}-1)}}{\Exp{D_{(i,i)}}}\textrm{ and }\\&\phi_{i,j}:=\frac{1}{4}\frac{\Exp{D_{(i,j)}(D_{(i,j)}-1)\ind{D_{(i,j)}\geq 2}}}{\Exp{D_{(i,j)}}}\cdot \frac{\Exp{D_{(j,i)}(D_{(j,i)}-1)\ind{D_{(j,i)}\geq 2}}}{\Exp{D_{(j,i)}}}. \numberthis \label{eq:def_nu_ij}
\end{align*}
Let us also define for convenience,
\begin{align*}
    l_i(n):=\sum_{x\in \cN^{(i)}}d^{(i,i)}_x \textrm{ and } l_{i,j}(n):=\sum_{x\in \cN^{(i)}}d^{(i,j)}_x=\sum_{x\in \cN^{(j)}}d^{(j,i)}_x.
\end{align*}

    As a consequence of Remark \ref{rem:uniform_induce} and by existing enumeration results, for example, by \cite[Theorem 1.10]{angel2019limit} and \cite[Corollary 7.16]{van2016random}, under the further assumption that all pairs $(i,j)$ are GOOD, we obtain the following result on the size of $\cG(\bD(n))$. As $n \to \infty$,
    \begin{align*}
        |\cG(\bD(n))|=(1+o(1))\left(\prod_{i\in [k]}\left(\frac{e^{(-\phi_i/2-\phi_i^2/4)}(l_i(n))!!}{\prod_{x \in \cN^{(i)}}(d^{(i,i)}_x)!} \right) \prod_{i<j}\left(\frac{e^{-\phi_{i,j}}(l_{i,j}(n))!}{\prod_{x \in \cN^{(i)}}(d^{(i,j)}_x)!}\right) \right).
    \end{align*}
\end{remark}

\subsubsection{Local limit}
Let us state our result on the local structure of the graph sequence $(G_n)_{n \geq 1}$. We will need one further assumption on the degrees to have a well-defined local limit. Towards this, fix $i,j \in [k]$. For any $i,j\in [k]$ such that there is at least one non-zero entry in $\bd_{(j,i)}(n)$, let $W_n(i,j)$ be the random index in $\{1,\dots,N^{(j)}\}$ distributed as,
\begin{align*}
    \Prob{W_n(i, j)=I}=\frac{d^{(j,i)}_I}{\sum_{x =1}^{N^{(j)}}d^{(j,i)}_x},\;\;\text{for}\;\;I \in \{1,\dots,N^{(j)}\}.
\end{align*}
Thus, $W_n(i, j)$ is a random index in $\{1,\dots,N^{(j)}\}$ with mass proportional to the corresponding entry in the degree sequence $\bd_{(j,i)}$. 

Next, for $l \in [k]$, we define 
an integer valued random variable $D_{n,(i,j,l)}$ as,
\begin{align*}
    \Prob{D_{n,(i,j,l)}=m}=\begin{cases}
        &\Prob{d^{(j,l)}_{W_n(i,j)}=m},\textrm{ if }l\neq i\\& \Prob{d^{(j,i)}_{W_n(i,j)}-1=m},\textrm{ if } l=i,
    \end{cases}\;\;\textrm{ for }m=0,1,2,\dots . \numberthis \label{eq:law_Dn_ijr}
\end{align*}
Thus, $D_{n,(i,j,l)}$ is the degree into $\cN^{(l)}$ of a random vertex with law $W_n(i,j)$ if $l \neq i$ and it is one less than the degree if $l=i$.

\begin{assumption}\label{assump:offspring_degree_reg}
We assume the existence of limiting random variables $D_{(i,j,l)}$, such that
\begin{align*}
    D_{n,(i,j,l)} \dlim D_{(i,j,l)},\;\;\text{with}\;\; \Exp{D_{n,(i,j,l)}}\to \Exp{D_{(i,j,l)}}.\numberthis \label{eq:assump_regularity_tensor}  
\end{align*}

\end{assumption}

\begin{remark}\label{rem:assump}
    Note that the second assertion of Assumption \ref{assump:offspring_degree_reg} means as $n \to \infty$
\begin{align*}
    \frac{\sum_{x =1}^{N^{(j)}}d^{(j,i)}_xd^{(j,l)}_x}{\sum_{x=1}^{N^{(j)}}d^{(j,i)}_x}\to \Exp{D_{(i,j,l)}} \textrm{ if }l\neq i,\textrm{ and }\frac{\sum_{x =1}^{N^{(j)}}d^{(j,i)}_x(d^{(j,i)}_x-1)}{\sum_{x=1}^{N^{(j)}}d^{(j,i)}_x}\to \Exp{D_{(i,j,i)}}\textrm{ if }l=i.\numberthis \label{eq:deg_tensor_ijr_entry_desc}
\end{align*}
In particular, by Assumption \ref{assump:degree_regularity}, for $(i,j)$ a GOOD pair we have $\Exp{D_{(i,j,i)}}=\frac{\Exp{D_{(j,i)}(D_{(j,i)}-1)}}{\Exp{D_{(j,i)}}}$.
\end{remark}

\paragraph{Description of the local limit.} Let us next describe a multi-type branching process $\cT_\infty$ that arises as the local weak limit of the graph sequence $(G_n)_{n \geq 1}$. First, for any integer valued random variable $X$, we define $\underline{X}$ to be its \emph{shifted size-biased} copy, i.e., 
\begin{align*}
    \Prob{\underline{X}=k}=\frac{(k+1)\Prob{X=k+1}}{\Exp{X}}\;\;\text{for}\;\;k=0,1,2,\dots \numberthis \label{eq:shifted_sb}
\end{align*}
The set of types in $\cT_\infty$ is the set $\cS=[k]$. 

\begin{itemize}
    \item The root $\varnothing$ has a random type $T_\varnothing$, where $\Prob{T_\varnothing=i}=\gamma_i$ for $i \in [k]$, where recall $\gamma_i$ from Assumption \ref{assump:part_i_size}.
    \item Conditionally on $T_\varnothing=i$, for each $j \in [k]$, the root has a random number $D_\varnothing(j)$ of type-$j$ offspring that is equal in distribution to $D_{(i, j)}$, where recall these variables from Assumption \ref{assump:degree_regularity}. Furthermore, for $j\neq l$, $D_\varnothing(j)$ and $D_\varnothing(l)$ are independent. Thus, the total number of offspring $D_\varnothing=\sum_{i=1}^k D_\varnothing(i)$ satisfies the distributional identity,
    \begin{align*}
        D_\varnothing \stackrel{d}{=} \sum_{j=1}^k D_{(T_\varnothing , j)},
    \end{align*}
    where given $T_\varnothing$, the RHS above is an independent sum.
    \item For any vertex $v \in \cT_\infty$ such that $v \neq \varnothing$, conditionally on it having type $T_v$ and its parent having type $T_{p(v)}$, it produces $D_v(j)$ offspring of type $j$, and these numbers are independent across different $j$ and different $v$. Further, if $j=T_{p(v)}$, $D_v(T_{p(v)})$ has the same law as $\underline{D_{(T_v,T_{p(v)})}}$, while for any $j \neq T_{p(v)}$, $D_v(j)$ has the same law as $D_{(T_{p(v)},T_v,j)}$, where recall the variables $D_{(i,j,l)}$ from Assumption \ref{assump:offspring_degree_reg}.
\end{itemize}
The obtained branching process tree \emph{together with the vertex types} is denoted by $\cT_\infty$.

\paragraph{Rooted isomorphisms with types.} Recall for any graph $G$ by $V(G)$ and $E(G)$ we respectively denote its vertex set and edge set. A graph $G$ together with a distinguished vertex $o \in V(G)$ is called a \emph{rooted graph} and denoted as the tuple $(G,o)$. We consider rooted graphs where each vertex has a type from some finite set $S$ of types. Throughout this paper we have $S=[k]$, and we will think of the type of $v\in V_n$ as $i$ if $v \in \cN^{(i)}$. 

For two rooted graphs $(G_1,o_1)$ and $(G_2,o_2)$ with vertex types from some finite set $S$, we say they are \emph{rooted-isomorphic} if there is a bijection $\sigma$ from $V(G_1)$ to $V(G_2)$ with $\sigma(o_1)=o_2$ and such that $T_{\sigma(v)}=T_v$ for every $v \in V(G_1)$, where for any vertex $v$ either in $G_1$ or $G_2$ by $T_v\in S$ we denote its type. In the case of existence of such an isomorphism, we write 
\begin{align*}
    (G_1,o_1)\simeq_S (G_2,o_2)
\end{align*}
to indicate this. For any rooted graph $(G,o)$, we define the \emph{ball of radius $r$} about $o$ as
\begin{align*}
    B_{G}(o,r):=\{w\in V(G):d_{G}(w,o)<r\},
\end{align*}
$d_G(\cdot,\cdot)$ denoting the \emph{graph distance}\footnote{Recall, the graph distance between two vertices in any graph is the length of the shortest length path between them, interpreted as equal to infinity if there is no such path (equivalently, if the two vertices lie in different connected components).} in $G$.
Our main result on the local limit of $(G_n)_{n \geq 1}$ is as follows.
\begin{theorem}[Local limit of simple uniform multipartite graphs]\label{thm:loc_lim}
    Consider the graph $G_n=(V_n,E_n)$ as a graph with types from $S=[k]$, with a vertex $v$ having type $i$ if $v \in \cN^{(i)}$, and let $o_n$ be a randomly sampled vertex from $V_n$, independent of $G_n$. Then under Assumptions \ref{assump:part_i_size}, \ref{assump:degree_regularity}, \ref{assump:offspring_degree_reg} and \ref{assump:one_good_deg_greater_2}, the rooted graph with types $(G_n,o_n)$ converges \emph{locally in probability} to the infinite rooted tree with types $(\cT_\infty,\varnothing)$, i.e., the following convergence in probability holds: for any finite rooted tree $(\bt,o)$ with vertex types from $[k]$ and for any $r\geq 1$,
    \begin{align*}
    \frac{1}{n}\sum_{v \in V}\ind{B_{G_n}(v,r)\simeq_S (\bt,o)}\plim \Prob{B_{\cT_\infty}(\varnothing,r)\simeq_S (\bt,o)}. \numberthis \label{eq:loc_conv_claim}
\end{align*}
\end{theorem}
The local limit was also discovered albeit under certain extra assumptions by the authors of \cite{gamarnik2015giant}, although the authors do not phrase it that way.

\begin{remark}[Metrizing the convergence]
    In fact the last convergence result can be metrized by an appropriate \emph{local} distance metric, for details on this, we refer the reader to \cite[Chapter 2]{van2024random}. Let us comment that the mentioned reference works with graphs without types, but the basic principles work through for graphs with types in a straightforward manner, at least when one has finitely many types, as in our setting.
\end{remark}

\subsection{Survival of the local limit and giant components}
In this section we phrase our main result on the survival analysis of the local limit, and a giant component phase transition result for $(G_n)_{n \geq 1}$. As mentioned before, the analysis of the size of the giant component of this model was first undertaken in the paper \cite{gamarnik2015giant}, where the authors prove a phase transition under certain assumptions. A more detailed comparison between the results of \cite{gamarnik2015giant} and our results is discussed after the statements of our results in this direction.

\subsubsection{Survival of the local limit}
Consider a general branching process, of which the local limit of Theorem \ref{thm:loc_lim} is a particular example, defined as follows. Let $S$ be a finite type space. Let $\bp:=(p_s)_{s \in S}$ be a probability vector. For each $i,j \in S$, let $D_{(i,j)}$ be an integer valued random variable, and we denote by $\bD$ the random matrix whose $(i,j)$-th entry is $D_{(i,j)}$. For each $i,j, k \in S$, let ${D}_{(i,j,k)}$ be another collection of integer valued random variables, and we denote by $\Tilde{\bD}$ the array (or tensor) $(D_{(i,j,k)})_{i,j,k \in S}$.

\begin{definition}[Multi-type branching process corresponding to $\bp$, $\bD$ and $\tilde \bD$]\label{def:BP}
The multi-type branching process corresponding to $\bp$, $\bD$ and $\Tilde{\bD}$ is denoted as $\mathrm{BP}(\bp, \bD, \Tilde{\bD})$, and is defined as the following branching process tree:
\begin{itemize}
    \item The root $\varnothing$ has a random type $s$ with probability $p_s$.
    \item Conditionally on its type being $s$, the root $\varnothing$ produces $D_{(s,i)}$ many offspring of type $i$ for each $i \in S$, and these numbers across different $i$ are (conditionally) independent.
    \item For any other vertex $v \neq \varnothing$, given its type being $T_v$ and the type of its parent being $T_{p(v)}$, it produces $D_{(T_{p(v)},T_v,i)}$ many vertices of type $i$, and these numbers across different $i$ are (conditionally) independent.
\end{itemize}
\end{definition}
To state the main result on the survival probability of the above process, we need to introduce a few concepts. Recall a digraph is a graph with directed edges.

\begin{definition}[Type digraph]\label{def:type graph}
    Consider a digraph $G_S$, with vertex set $V(G_S)$ consisting of all \emph{type triplets}, $V(G_S):=\{(u,v,w):u,v,w \in S\}$ and directed edges go from vertices of the form $(\cdot,v,w)$ to $(v,w,\cdot)$. Denote the edge set of this graph by $E(G_S)$. Thus $E(G_S)$ consists of ordered pairs $(u_0,u_1)$ for $u_0,u_1 \in V(G_S)$. We call $G_S=(V(G_S),E(G_S))$ the \textit{type digraph} corresponding to the type set $S$. 
\end{definition}

\begin{remark}
    Note that $|V(G_S)|=|S|^3$.
\end{remark}

Consider the above branching process $\BP(\bp,\bD,\Tilde{\bD})$ as defined in Definition \ref{def:BP} with vertices having types from the set $S$. For any vertex $u=(s_0,s_1,s_2) \in G_S$, we (naturally) denote by $D_u$ the random variable $D_{(s_0,s_1,s_2)}$, that is the $(s_0,s_1,s_2)$-th entry of the tensor $\Tilde{\bD}$. For any set of types $S$ and corresponding type digraph $G_S$, for any edge $e=(u,v)\in E(G_S)$ we define the \emph{weight} of $e$ to be 
\begin{align*}
    w(e):=\ind{\Exp{D_u}\Exp{D_v}>0}.
\end{align*}
Thus, note that if both $(u,v)$ and $(v,u)$ are directed edges in $G_S$, then the weights of these edges are the same. Denote by $D(G_S)$ the subgraph of $G_S$ on the same vertex set $V(G_S)$ but with edge set $\{e\in E(G_S):w(e)=1\}$, and for any $v \in G_S$, denote by $\cC_D(v)$ the set of vertices that can be reached using a directed path from $v$ in $D(G_S)$.

For fixed $s_0,s_1 \in S$, define the vertex subset $$V(s_0,s_1):=\{v\in V(G_S):v=(s_0,s_1,s)\textrm{ for some }s \in S\}\subseteq V(G_S).$$ Finally, for any vertex $v \in V(G_S)$, denote by $\cL(v)$ the collection of all \textit{simple cycles} in $D(G_S)$ that starts and ends at $v$. Recall a simple cycle that starts and ends at $v$ in $D(G_S)$ is a directed path in $D(G_S)$ of the form $(v_0,v_1,\dots, v_k)$, where $v_0=v_k=v$, with $v_i\neq v$ for all $1\leq i \leq k-1$, and where all of $v_1,\dots, v_{k-1}$ are distinct. For any $c=(v_0,v_1,\dots,v_k) \in \cL(v)$ with $v_0=v_k=v$, define the number
\begin{align*}
    \sM(c):=\prod_{i=0}^{k-1}\Exp{D_{v_i}}. \numberthis \label{eq:loop_product}
\end{align*}
Let us denote for any $i,j,k \in S$ and for $m =0,1,2,\dots$
\begin{align*}
    \Prob{D_{(i,j)}=m}=p^{(i,j)}_m\;\;\text{and}\;\;\Prob{D_{(i,j,k)}=m}=p^{(i,j,k)}_m,
\end{align*}
and recall that the type of the root has the probability mass vector $(p_s)_{s\in S}$ on $S$.

\begin{theorem}[Extinction of $\BP(\bp,\bD,\Tilde{\bD})$]\label{thm:extinct_1dep}
    Consider the multi-type branching process $\BP(\bp,\bD,\Tilde{\bD})$ from Definition \ref{def:BP}, and denote its extinction probability by $\eta$. With the convention that the maximum over an empty set is $0$,  we have $\eta=1$ if and only if
    \begin{align*}
        \max_{\{s\in S:p_s>0\}}\max_{\{s'\in S:p^{(s,s')}_0<1\}} \max_{v \in V(s,s')}\max_{u \in \cC_D(v)}\max_{c \in \cL(u)}\sM(c)\leq 1. \numberthis \label{eq:ext_BP_condition}
    \end{align*}
\end{theorem}

%
\paragraph{Relation with standard multi-type branching processes.}
The process $\BP(\bp,\bD,\Tilde{\bD})$ can be seen as a \emph{standard} multi-type branching process on the type space $S^2:=\{(i,j):i,j\in S\}$ by letting a vertex $v$ have type $(T_{p(v)},T_v)$, and letting it produce $D_{(T_{p(v)},T_v,j)}$ many children of type $(T_v,j)$ independently for each $j$. Thus, to analyze its extinction, one can apply classical spectral techniques. Indeed, under the two extra assumptions of Theorem \ref{thm:gamarnik_giant} below, this was the perspective adopted in \cite{gamarnik2015giant}. However note that, crucially, we don't assume anything related to \emph{irreducibility} in Theorem \ref{thm:extinct_1dep}. In particular, the associated mean offspring matrix may be reducible, in which case the classical results that come via Perron-Frobenius theory, see for example \cite[Chapter V]{athreya2012branching}, cannot be \emph{directly} applied, and one has to peer deep into the irreducible components of the matrix. 

In the end, we want to couple graph exploration with the branching process local limit, and having reducibility, together with the multipartite nature of the problem, makes the development of this coupling quite complicated if one wants to approach via spectral analysis as was done in \cite{gamarnik2015giant}. These complications led us to develop Theorem \ref{thm:extinct_1dep}, which, as we shall see, gives us a natural hint about where might the largest components of the graph $G_n$ be located, and thus aids us in graph exploration, together with revealing information regarding the \emph{typical distance} between two random vertices.

Of independent interest, Theorem \ref{thm:extinct_1dep} gives a natural simpler condition for the extinction of standard multi-type of branching processes; we discuss this interesting result in Section \ref{sec:disc_ext}.

\subsubsection{Giant components}

Let us now state our giant component result. Recall the random graph sequence $(G_n)_{n \geq 1}$, and its local limit $(\cT_\infty,\varnothing)$ from Theorem \ref{thm:loc_lim}. Denote the component of any vertex $v \in V_n$ in the graph $G_n$ by $\cC(v)$ and the $i$-th largest such component by $\cC^{(i)}_n$, $i \geq 1$. For the next theorem, we naturally represent our local limit $(\cT_\infty,\varnothing)$ in the notation $\BP(\bp,\bD,\Tilde{\bD})$ with obvious meanings for $\bp,\bD,\Tilde{\bD}$:
\begin{itemize}
    \item $\bp=(p_s)_{s\in S}$ with $p_s=\gamma_s$ for all $s\in S=[k]$.
    \item The entries of $\bD$ are $D_{(i,j)}$, the variables from Assumption \ref{assump:degree_regularity}.
    \item When $l=i$, the $(i,j,l)$-th entry of $\Tilde{\bD}$ is $\underline{D_{(j,i)}}$, where recall the shifted size-biased variable $\underline{X}$ from \eqref{eq:shifted_sb}.
    \item When $l\neq i$, the $(i,j,l)$-th entry of $\Tilde{\bD}$ is $D_{(i,j,l)}$, the variable from Assumption \ref{assump:offspring_degree_reg}. 
\end{itemize}

\begin{theorem}[Giants in uniform multipartite graphs with given degrees] \label{thm:giants_MCMs}
   Under the assumptions of Theorem \ref{thm:loc_lim}, consider the branching process $(\cT_\infty,\varnothing)$ that arises as the local limit of the graph sequence $(G_n)_{n \geq 1}$, and let $\eta$ denote its extinction probability.     
    \begin{itemize}
    \item [1.]\textbf{Supercritical regime.} Suppose condition \eqref{eq:ext_BP_condition} does not hold, so that $\eta<1$. Then
        \begin{align*}
        \liminf_{n \to \infty}\frac{\Exp{|\cC^{(1)}_n|}}{n}>0.
            \numberthis \label{eq:supcrit_conclusion}
        \end{align*}
    
        \item [2.]\textbf{Subcritical regime.} If condition \eqref{eq:ext_BP_condition} holds, so that $\eta=1$, then as $n \to \infty$ we have $\frac{|\cC^{(1)}_n|}{n}\plim 0.$
        \end{itemize}
\end{theorem}
For any graph sequence $\mathbf{G}_n=(V(\mathbf{G}_n),E(\mathbf{G}_n))$ let us call a connected component $\cC$ of $\mathbf{G}_n$ a \emph{giant component} if it satisfies $\liminf_{n \to \infty}\frac{\Exp{|\cC|}}{|V(\mathbf{G}_n)|}>0$\footnote{In the literature, typically the term \emph{giant component} refers to the largest component. However, here we are using this terminology for any component that has linear (in the number of vertices) size.}. Thus, the last result says that when the local limit survives, there is at least one giant component in the sequence $(G_n)_{n \geq 1}$, while when it dies out, there are no such.

\paragraph{Relation and comparison with results of \cite{gamarnik2015giant}.}
As mentioned before, the problem of understanding the giant component in the sequence $(G_n)_{n \geq 1}$ was first undertaken in \cite{gamarnik2015giant}. The main results of this work translated into our language are summarized in the following theorem. 
\begin{theorem}[Main results of the work \cite{gamarnik2015giant} by Gamarnik and Misra]\label{thm:gamarnik_giant}
    Under the assumptions of Theorem \ref{thm:giants_MCMs}, assume further:
    \begin{itemize}
        \item If a pair $(i,j)$ is BAD, then $\bd_{(i,j)}(n)$ and $\bd_{(j,i)}(n)$ are null sequences for all $n \geq 1$, i.e., they only have zero entries.
        \item  Considering the natural offspring matrix $M$ on $S^2 \times S^2$ where the $((i_1,j_1),(i_2,j_2))$-th entry is the number $\Exp{D_{(i_1,j_1,j_2)}}\ind{j_1=i_2}$, there exists a $p\geq 1$ such that all entries of $M^p$ are strictly positive.
    \end{itemize}
    Then if $\eta<1$, $\frac{|\cC^{(1)}_n|}{n}\plim 1-\eta$, while if $\eta=1$, then $\frac{|\cC^{(1)}_n|}{n}\plim 0$. Additionally, if $\eta<1$, the second largest component $\cC^{(2)}_n$ satisfies $|\cC^{(2)}_n|=O(\log n)$.
\end{theorem}
Thus, our Theorem \ref{thm:giants_MCMs} removes the two additional assumptions made in Theorem \ref{thm:gamarnik_giant}. The first of these assumptions simply rules out the existence of sublinear number of edges either within a partition or across two partitions. This we feel is quite non-realistic, as in many real-world settings there could be potentially edges connecting two communities, even though perhaps very few. The second assumption is precisely the \emph{irreducibility} that we mentioned in the paragraph after stating Theorem \ref{thm:loc_lim}. Roughly speaking, in terms of the local limit, what this means is that any type can give birth to any other type by generation $p$.

In particular, removing these two assumptions make the model truly deviate from the unipartite case, and thus makes the analysis quite a technical challenge, something we have managed to overcome in the current paper. Under the general assumptions that we work under, our supercritical result corresponding to $\eta<1$ in Theorem \ref{thm:giants_MCMs} is in fact sharp, as we do not expect $\frac{|\cC^{(1)}|}{n}$ to converge in probability to a non-zero constant, but rather to fluctuate at scale $n$. Further, there can be more than one giant. We discuss these phenomena in more detail after introducing the \emph{configuration model} in Section \ref{sec:CMs_sharpness}. Additionally, our techniques reveal information about the typical distance between two random vertices, see Theorem \ref{thm:dist} below.

Further, although the general philosophy that survival of the local limit implies the existence of a giant in the pre-limiting graph sequence under some general conditions is well understood by now (e.g., see \cite[Theorem 3.1]{bollobas2007phase} or \cite[Theorem 2.28]{van2024random}), as far as we are aware, Theorem \ref{thm:giants_MCMs} is the first result that proves this when the local limit is allowed to be reducible\footnote{For a comparison, the necessary and and sufficient condition of \cite[Theorem 2.28]{van2024random} can be understood as a condition ruling out reducibility of the local limit.}, and we believe the technique can be used to prove existence of giants in other graph sequences as well, where the local limit is not necessarily irreducible. We comment more on this in Section \ref{sec:disc_giant}.


\subsection{Typical distances}
Our final result is on \emph{typical distances} between two uniform vertices in $G_n$. Let $o_1$ and $o_2$ be uniformly sampled (without replacement) random vertices of $G_n$. Recall for any two vertices in $G_n$, by $d_{G_n}(u,v)$ we denote the {graph distance} between $u$ and $v$ in $G_n$. Recall $\eta$ denotes the extinction probability of the local limit $\cT_\infty$ of $G_n$.
\begin{theorem}\label{thm:dist}
Assume the conditions of Theorem \ref{thm:loc_lim}.
\begin{itemize}
    \item[1.] If $\eta=1$, 
    \begin{align*}
        \limsup_{n\to \infty}\Prob{d_{G_n}(o_1,o_2)<\infty}=0.
    \end{align*}
    \item[2.] If $\eta<1$, then there is $j \in [k]$ and a constant $K^*>0$ such that
    \begin{align*}
        \liminf_{n \to \infty}\CProb{d_{G_n}(o_1,o_2)<K^* \log n}{o_1,o_2\in \cN^{(j)}}>0.
    \end{align*}
    In particular, $\liminf_{n \to \infty}\Prob{d_{G_n}(o_1,o_2)<K^* \log n}>0$.
    \item[3.] If $\eta<1$, and additionally, $(i,j)$ is $\rm GOOD$ for all $i,j \in [k]$, then there is $j\in [k]$ and a constant $K_*$ such that
    \begin{align*}
        \liminf_{n \to \infty}\CProb{d_{G_n}(o_1,o_2)>K_* \log n}{o_1,o_2\in \cN^{(j)}}>0.
    \end{align*}
    In particular, $\liminf_{n \to \infty}\Prob{d_{G_n}(o_1,o_2)>K_* \log n}>0$.
\end{itemize}

\end{theorem}

For part $(3.)$ above, the additional assumption on all pairs $(i,j)$ being GOOD should not be necessary, however we have not managed to prove the lower bound without this assumption. In particular, if a \emph{monotonicity result} (see Problem \ref{prob:dist}) is true, then the lower bound will be true without this assumption, which we expect indeed to be the case; this extra assumption is an artifact of our proof. We discuss more on this in Section \ref{sec:disc_TD}.

\paragraph{Organization of the rest of the article.} In Section \ref{sec:conf_mod}, we discuss a useful technique, namely sampling using configuration models. In Section \ref{sec:proof_loclim} we provide the proof of our local limit result Theorem \ref{thm:loc_lim}. In Section \ref{sec:survival} we prove the survival condition of Theorem \ref{thm:extinct_1dep}. In Section \ref{sec:giant_proof} we prove Theorem \ref{thm:giants_MCMs} on the existence of giant components. In Section \ref{sec:dist} we prove Theorem \ref{thm:dist} on typical distances. Finally, we close off with a discussion in Section \ref{sec:disc} where we discuss a few applications.

\section{Sampling using configuration models} \label{sec:conf_mod}

\subsection{Multipartite configuration models and sharpness of Theorem \ref{thm:giants_MCMs}} \label{sec:CMs_sharpness}
In this section we discuss the \emph{configuration model} (see e.g. Bollobás \cite{bollobas2011random}, van der Hofstad \cite[Chapter 4]{van2024random}), that constructs either a unipartite or a bipartite multigraph with a given degree sequence via random matching of half-edges. Recall the degree sequences $\bd_{(i,j)}(n)$ and $\bd_i(n)$ for each $n \geq 1$. For every $i \in [k]$ and every $x \in [N^{(i)}]$, associate sets $H_{x,i\rightarrow j}$ and $H_{x,i\rightarrow i}$ of \emph{half-edges}, where the sizes of these sets satisfy $|H_{x,i\rightarrow j}|=d^{(i,j)}_x(n)$ and $|H_{x,i\rightarrow i}|=d^{(i)}_x(n)$. We also define 
\begin{align*}
    H_{i \rightarrow j}=\cup_{x \in [N^{(i)}]}H_{x,i \rightarrow j}\textrm{ when }i \neq j,H_{i \rightarrow i}=\cup_{x \in [N^{(i)}]}H_{x,i\rightarrow i},\textrm{ and }H=\cup_{i,j\in [k]}H_{i \rightarrow j}.\numberthis \label{eq:half_edge_sets}
\end{align*}
Any half-edge $h\in H_{i \rightarrow j}$ is said to be \emph{compatible} with any half-edge in $H_{j \rightarrow i}$ and \emph{not compatible} with any half-edge in $H_{p\rightarrow q }$ where either $p \neq j$ or $q \neq i$. For a half-edge $h \in \cup_{j \in [k]}H_{x,i\rightarrow j}$, we say that $h$ is \emph{incident to} $x$.

Then, we construct a multigraph $CM_n$ on the vertex set $V_n$ as follows:
\begin{itemize}
    \item We begin the process by taking a half-edge $h$ arbitrarily from $H$, sampling a uniformly random compatible half-edge $h'$, and matching $h$ with $h'$. If $h \in H_{x,i\rightarrow j}$ and $h' \in H_{y,j \rightarrow i}$ for $i,j \in [k]$, we declare $\{x,y\}$ to be an edge in $CM_n$.
    \item In a general step, we do the same. Take an arbitrary yet unmatched half-edge from $H$, and sample uniformly at random a compatible yet unmatched half-edge, and pair them. Declare $\{x,y\}$ to be an edge in $CM_n$ if $x$ and $y$ are the vertices these half-edges are incident to.
\end{itemize}
We run this process till we run out of unmatched half-edges. The condition $\sum_{x\in [N^{(i)}]}d^{(i,j)}_x(n)=\sum_{y\in [N^{(j)}]}d^{(j,i)}_y(n)$ and $\sum_{x \in [N^{(i)}]}d^{(i)}_x(n)$ being even ensures that no unmatched half-edges will be left at the end. The associated multigraph $CM_n$ thus constructed is called the \emph{multipartite configuration model} with degree sequence matrix $\bD(n)$. Let us further denote by $CM^{i,i}_n$ the subgraph of $CM_n$ induced by the vertices in $\cN^{(i)}$, and for $i\neq j$ by $CM^{i,j}_n$ the bipartite graph induced by the cross-edges going in between $\cN^{(i)}$ and $\cN^{(j)}$, on the vertex set $\cN^{(i)}\cup \cN^{(j)}$. Then:
\begin{itemize}
    \item $CM^{i,i}_n$ is distributed as a (unipartite) configuration model on $\cN^{(i)}$ with degree sequence $\bd_{i}(n)$;
    \item For $i \neq j$, $CM^{i,j}_n$ is the bipartite configuration model on $\cN^{(i)}\cup \cN^{(j)}$ with bi-degree sequence $(\bd_{(i,j)}(n),\bd_{(j,i)}(n))$.
    \item Let us also note that the graphs $CM^{i,j}_n$ can be constructed on their own, without constructing the entire graph $CM_n$, by only pairing half-edges from $H_{i \rightarrow j}$ with those in $H_{j \rightarrow i}$, for $i,j\in [k]$.
\end{itemize}

Before proceeding any further, let us discuss the sharpness of our giant component result Theorem \ref{thm:giants_MCMs}.

\paragraph{Sharpness of Theorem \ref{thm:giants_MCMs}.} As mentioned after the statement of Theorem \ref{thm:gamarnik_giant}, in general, under Assumptions \ref{assump:part_i_size}, \ref{assump:degree_regularity}, \ref{assump:offspring_degree_reg} and \ref{assump:one_good_deg_greater_2}, one cannot conclude decisively about uniqueness of the giant component as in Theorem \ref{thm:gamarnik_giant}. For example, take $k=2$, and let both $(1,1)$ and $(2,2)$ be GOOD, and let $\bd_{(1,1)}(n)$ and $\bd_{(2,2)}(n)$ satisfy Assumptions \ref{assump:offspring_degree_reg} and be such that $$\frac{\Exp{D_{(1,1)}(D_{(1,1)}-1)}}{\Exp{D_{(1,1)}}},\frac{\Exp{D_{(2,2)}(D_{(2,2)}-1)}}{\Exp{D_{(2,2)}}}>1. $$ It then follows from existing results (e.g., by Molloy and Reed \cite{molloy1995critical,molloy1998size}) that the subgraphs $G_n(1)$ and $G_n(2)$ themselves have their own giant components of linear size with logarithmic second largest components, under Assumption \ref{assump:part_i_size}. Now, say $(1,2)$ is BAD. If $\bd_{(1,2)}(n)$ is a null sequence for each $n \geq 1$, then clearly there are no edges going between $\cN^{(1)}$ and $\cN^{(2)}$, so that in this situation we have two giants in $G_n$. 

On the other hand, consider the following random process. Take a uniformly random bijection $\psi$ from $\cN^{(1)}$ to $\cN^{(2)}$ (assume they have equal size). For each $x \in [N^{(1)}]$ consider independent ${\rm Ber}(m/n)$ random variables $B_x$ where $m=n/\log n$ and conditionally on $\psi$, assign for each $x$ such that $B_x=1$, $d^{(1,2)}_x=d^{(2,1)}_{\psi(x)}=3$, and assign for all other $x$ in either $\cN^{(1)}$ or $\cN^{(2)}$, $d^{(1,2)}_x=d^{(2,1)}_x=0$. First note that $\sum_{x=1}^{N^{(1)}}d^{(1,2)}_x=\sum_{x=1}^{N^{(2)}}d^{(2,1)}_x$ is concentrated around $3m=o(n)$, so that $(1,2),(2,1)$ are BAD pairs. Now, on the (positive probability) event that the largest components $\cC_1$ and $\cC_2$ in respectively $G_n(1)$ and $G_n(2)$ are at least $\delta n$, a standard Bernoulli concentration argument shows that $\Theta(m)$ many vertices in all of $\cC_1, \cC_1^c, \cC_2$ and $\cC_2^c$ have degree $3$ going into the other partition. By a result of Janson \cite{janson2009probability}, since $\sum_{x \in [N^{(i)}]:B_x=1}(d^{(i,j)}_x)^2=\Theta(\sum_{x \in [N^{(i)}]:B_x=1}d^{(i,j)}_x)$, the \emph{cross-edges} going in between $\cN^{(1)}$ and $\cN^{(2)}$ can be sampled as a bipartite configuration model, conditionally on its asymptotically positive probability event of being simple, with bi-degree sequence $(\bd_{(1,2)}(n),\bd_{(2,1)}(n))$. 

Now, given the good event of \emph{simplicity}, and that \emph{all of $\cC_1, \cC_1^c, \cC_2$ and $\cC_2^c$ have $\Theta(m)$ many vertices with degree $3$ going into the other partition}, it follows that with probability at most $c^m$ for some $c<1$, all of the half edges from $\cC_1$ get paired to half-edges from $\cC_2^c$. Thus, the event that there  is a direct cross-edge connecting $\cC_1$ and $\cC_2$ has probability at least $1-c^m=1-o(1)$. Thus, we have a unique giant component in this case with (conditional) high probability (given the good event), even though the average number of cross-edges going in between $\cN^{(1)}$ and $\cN^{(2)}$ is sublinear in $n$. In particular, this example illustrates that, under our general assumptions of Theorem \ref{thm:giants_MCMs}, the giant can be both unique or non-unique each with positive probability, and its size can fluctuate on the linear scale.

\subsection{Sampling $G_n$}

In this section we discuss how to properly use the multipartite configuration model to sample $G_n$. 

\begin{definition}[Amalgamation]
    For every $i,j\in [k]$, let $G_{i,i}=(V_{i,i},E_{i,i})$ be a multigraph on the vertex set $\cN^{(i)}$ and $G_{i,j}=(V_{i,j},E_{i,j})$ be a bipartite multigraph on the bipartitioned vertex set $\cN^{(i)}\cup \cN^{(j)}$. Given graphs $G_{i,j}$ for every $i,j \in [k]$, we define the multigraph $\cG(G_{i,j}:i,j\in [k])$ to be the simply the graph on $V_n=\cN^{(1)}\cup \dots \cup \cN^{(k)}$ whose edge set is the union $\cup_{i,j\in [k]}E_{i,j}$. We call $\cG(G_{i,j}:i,j\in [k])$ the \emph{amalgamation} of the graphs $G_{i,j}$ for $i,j\in [k]$.
\end{definition}
By Remark \ref{rem:uniform_induce}, we note that $G_n$ has the same law as $\cG=\cG(G_{i,j}:i,j\in [k])$ when $G_{i,i}$ is a uniform simple graph on $\cN^{(i)}$ with degree sequence $\bd_i(n)$ and $G_{i,j}$ for $i \neq j$ is a uniform simple bipartite graph on $\cN^{(i)}\cup \cN^{(j)}$ with given bi-degree sequence $(\bd_{(i,j)}(n),\bd_{(j,i)}(n))$, for all $i,j\in [k]$. To sample these individual simple uniform graphs, we use configuration models. Let us begin with a result on the simplicity probability of configuration models. The following result is a summary of \cite[Theorem 7.12, Proposition 7.15]{van2016random} and \cite[Corollary 7.16]{angel2019limit}. Recall $\phi_i$ and $\phi_{i,j}$ from \eqref{eq:def_nu_ij}.

\begin{theorem}[Simplicity probability of configuration models]\label{thm:simplicity_CM}
    Assume the pair $(i,j)$ is $\rm GOOD$, $i,j\in [k]$.
    \begin{itemize}
        \item If $i=j$, $\liminf_{n \to \infty}\Prob{CM_n^{i,i}\textrm{ is simple}}=e^{-\phi_i/2-(\phi_i)^2/4}>0$.
        \item  If $i \neq j$, $\liminf_{n \to \infty}\Prob{CM_n^{i,j}\textrm{ is simple}}=e^{-\phi_{i,j}}>0$.
        \item The law of $CM_n^{i,i}$ conditionally on the event $\{CM_n^{(i,i)}\textrm{ is simple}\}$ is uniform on the set of all simple graphs on $\cN^{(i)}$ with degree sequence $\bd_i(n)$.
        \item The law of $CM_n^{i,j}$ conditionally on the event $\{CM_n^{(i,j)}\textrm{ is simple}\}$ is uniform on the set of all simple bipartite graphs on $\cN^{(i)}\cup \cN^{(j)}$ with bi-degree sequence $(\bd_{(i,j)}(n),\bd_{(j,i)}(n))$.
    \end{itemize}
\end{theorem}

As a direct consequence of the last result and Remark \ref{rem:uniform_induce}, we have the following corollary on sampling $G_n$ whose proof we omit, and which we use throughout the paper.

\begin{corollary}[Sampling $G_n$]\label{cor:sample_Gn}
    Consider graphs $G_{i,j}$ defined as follows. If $(i,j)$ is $\rm BAD$, we simply let $G_{i,j}$ be a uniform random simple graph with degree sequence $\bd_i(n)$ if $i=j$ and with degree sequence $\bd_{(i,j)}(n)$ if $i \neq j$. If $(i,j)$ is $\rm GOOD$, we let $G_{i,j}$ be $CM_n^{i,j}$. Recall $G_n$ is a uniform multipartite random graph with given degree sequences, satisfying Assumptions \ref{assump:part_i_size}, \ref{assump:degree_regularity}, \ref{assump:one_good_deg_greater_2} and \ref{assump:offspring_degree_reg}. The following statements are true:
    \begin{itemize}
        \item The random graph $\G_n:=\cG(G_{i,j}:i,j\in [k])$ conditionally on the event $\{\G_n \textrm{ is simple}\}$ has the same law as $G_n$.
        \item Moreover,
        \begin{align*}
            \liminf_{n \to \infty}\Prob{\G_n \textrm{ is simple}}=\left(\prod_{\substack{i\in [k]:\\(i,i)\textrm{ is GOOD}}}e^{-\phi_i/2-(\phi_i)^2/4}\right)\left(\prod_{\substack{i\neq j \in [k]:\\(i,j)\textrm{ is GOOD }}}e^{-\phi_{i,j}} \right)>0.
        \end{align*}
        Thus, for any subset $A_n$ of simple graphs, $\Prob{G_n\in A_n}=o(1)$ if and only if $\Prob{\G_n \in A_n}=o(1)$.  
    \end{itemize}
\end{corollary}
As a consequence of Corollary \ref{cor:sample_Gn}, it suffices to prove Theorems \ref{thm:loc_lim}, \ref{thm:giants_MCMs} and \ref{thm:dist} for the graph $\G_n$ instead, which is our strategy.

\section{Local limit: proof of Theorem \ref{thm:loc_lim}}\label{sec:proof_loclim}
In this section, we provide the proof of Theorem \ref{thm:loc_lim}. The usual well trodden path of coupling the graph exploration with the growth of the branching process arising as the local limit works in our case, albeit with some technicalities arising because of the multipartite nature of the model. This style of argument has appeared in the literature before (see e.g., the local convergence proofs in \cite{van2024random}). We ask the reader to permit us an argument capturing the key ideas and emphasizing readability, taking care to avoid repetition as much as possible from earlier literature. 

Recall our aim is to show \eqref{eq:loc_conv_claim}. Thanks to Corollary \ref{cor:sample_Gn}, it is enough to show \eqref{eq:loc_conv_claim} with $G_n$ replaced by $\G_n$. Before beginning the proof, let us make one more reduction. First note that the LHS of \eqref{eq:loc_conv_claim} can be written as $\CProb{B_{\G_n}(o_n,r)\simeq_S (\bt,o)}{\G_n}$, where $o_n$ is a randomly sampled vertex of $\G_n$. Thus, \eqref{eq:loc_conv_claim} would follow as a consequence of Chebyshev's inequality (together with the following statement for $l=1,2$), if we can show the general statement 
\begin{align*}
    &\Prob{B_{\G_n}(o_{n,1},r_1)\simeq_S (\bt_1,o_1), B_{\G_n}(o_{n,2},r_2)\simeq_S (\bt_2,o_2),\dots,B_{\G_n}(o_{n,m},r_m)\simeq_S (\bt_m,o_m)}\\&\to \prod_{i=1}^m \Prob{B_{\cT_\infty}(\varnothing,r_i)\simeq_S (\bt_i,o_i)},\numberthis \label{eq:loc_conv_general_claim}
\end{align*}
for any fixed integer $m\geq 1$, where $o_{n,1},\dots,o_{n,m}$ are independently sampled with replacement uniform vertices (in some order) from $\G_n$ (note that with high probability these are distinct), and where $r_1,\dots,r_m\geq 1$ are integers, and $(\bt_1,o_1),\dots,(\bt_m,o_m)$ are some fixed rooted trees with types in $S=[k]$. 
We will prove Theorem \ref{thm:loc_lim} by giving an inductive proof of \eqref{eq:loc_conv_general_claim}. 

 Recall an ordered pair $(i,j)\in [k]\times [k]$ is ${\rm BAD}$ if $\Exp{D_{(i,j)}}=0$. 
 Recall from \eqref{eq:half_edge_sets} by $H_{i \rightarrow j}$ we denote the set of half-edges incident to vertices in $\cN^{(i)}$ that correspond to their degrees in $\cN^{(j)}$. If $(i,j)$ is $\rm BAD$, we also call the half-edges in both $H_{i \rightarrow j}$ and $H_{j \rightarrow i}$ $\rm BAD$. For a vertex $v \in \cN^{(j)}$, let us also define by ${\rm BAD}_v$ the set of all $\rm BAD$ half-edges incident to $v$, i.e. 
\begin{align*}
    {\rm BAD}_v:=\cup_{(j,l){\rm BAD}}\{h \in H_{j \rightarrow l}: h\textrm{ is incident to }v\}.\numberthis \label{eq:def_bad_v}
\end{align*}

Define the sequence
\begin{align*}
    \sS_n:=\sum_{i,j\in [k]:(i,j){\rm BAD}}\frac{|H_{i \rightarrow j}|}{N^{(i)}}+\sum_{i \in [k]}\sum_{\substack{j \in [k]:\\(i,j){\rm GOOD}}}\sum_{v \in \cN^{(j)}}\frac{d^{(i)}_v}{|H_{j \rightarrow i}|}\sum_{l \in [k]:(j,l){\rm BAD}}d^{(l)}_v.\numberthis \label{eq:def_sn_goodbadseq}
\end{align*}

\begin{lemma}\label{lem:sn_small}
As $n \to \infty$, $\sS_n=o(1)$.
\end{lemma}
\begin{proof}
    Note that for any $\rm BAD$ pair $(i,j)$, $|H_{i \to j}|=\sum_{v \in \cN^{(i)}}d^{(j)}_v$, so that $|H_{i \to j}|=o(N^{(i)})$ by the definition of being $\rm BAD$. Thus the first term on the RHS of \eqref{eq:def_sn_goodbadseq} is $o(1)$.

    On the other hand, fix a $\rm GOOD$ pair $(i,j)$ and a $\rm BAD$ pair $(j,l)$, and consider the sum
    \begin{align*}
        \sum_{v \in \cN^{(j)}}\frac{d^{(i)}_vd^{(l)}_v}{|H_{j \rightarrow i}|}.
    \end{align*}
    Since $(i,j)$ is $\rm GOOD$, recall from \eqref{assump:regularity_1_2} and \eqref{eq:part_i_size} that we have $|H_{j \rightarrow i}|=\sum_{v \in \cN^{(j)}}d^{(i)}_v=\Theta(\cN^{(j)})=\Theta(n)$, and thus, the last display equals up to a non-zero multiplicative constant to $\Exp{d^{(i)}_od^{(l)}_o}$, where $o$ is a randomly sampled vertex from $\cN^{(j)}$. By Cauchy-Schwarz, we have 
    \begin{align*}
        \Exp{d^{(i)}_od^{(l)}_o}\leq \sqrt{\Exp{(d^{(i)}_o)^2}\Exp{(d^{(l)}_o)^2}},
    \end{align*}
    and note that $\Exp{(d^{(l)}_o)^2}=\frac{1}{N^{(j)}}\sum_{v \in \cN^{(j)}}(d^{(l)}_v)^2\to \Exp{D_{(j,l)}^2}=0$, where the limit follows from \eqref{assump:regularity_1_2} and the last equality holds since $\Exp{D_{(j,l)}}=0$ and since the variable $D_{(j,l)}$ is non-negative. On the other hand, arguing similarly, $\Exp{(d^{(i)}_o)^2} \to \Exp{D_{(j,i)}^2}<\infty$ due to \eqref{assump:regularity_1_2}. Thus, the second term on the RHS of \eqref{eq:def_sn_goodbadseq} is also $o(1)$.
\end{proof}

\begin{proof}[Proof of Theorem \ref{thm:loc_lim} by proving \eqref{eq:loc_conv_general_claim}]
    \textbf{Step I: Base case $m=1$.} In this case, we have to show that for a fixed rooted tree $(\bt,o)$ with vertex types in $S=[k]$ and for any $r\geq 1$,
    \begin{align*}
        \Prob{B_{\G_n}(o_n,r)\simeq_S (\bt,o)}\to \Prob{B_{\cT_\infty}(\varnothing,r)\simeq_S (\bt,o)}. \numberthis \label{eq:loc_to_show_basecase}
    \end{align*}
As mentioned before, we couple the breadth-first exploration started from $o_n$ in $\G_n$ with the growth of a branching process. As we shall see, by a simple consequence of Lemma \ref{lem:sn_small}, while trying to explore the graph from a uniform vertex $o_n$, for finitely many exploration steps, we will simply \emph{not encounter} a vertex with any BAD half-edge incident to it. Thus, for the sake of graph exploration, the edges in the GOOD configuration models $CM^{i,j}_n$ (where $(i,j)$ is a GOOD pair) are all that matter. 

To sample these edges, we use a standard technique that allows construction of the configuration model for pairing the half-edges \emph{as we are exploring} the graph $\G_n$. That is to say, instead of revealing all the configuration model edges in the graph $\G_n$ at once, we can start from $o_n$, and first pair the half-edges incident to it, then pair the half-edges incident to \emph{discovered} vertices at distance $1$ from $o_n$, etc., in a breadth-first manner, and if we stop this process after say discovering $M$ vertices (for some $M\geq 1$), the graph constructed so far is identically distributed to breadth-first exploration up to $M$ vertices starting from $o_n$ in $\G_n$. This \emph{pairing as we are exploring} idea goes back to at least Janson and Luczak \cite{janson2009new}.

Thus, we need to couple the process outlined above with a branching process growth, specifically, the branching process arising as the local limit in the theorem statement. We do this as follows. Start with $o_n$, and let the root $\varnothing$ of a branching process $\BP_n$ have type $i$ if $o_n \in \cN^{(i)}$. Naturally, on this event, the type of $o_n$ is also $i$. Denoting by $\G_n(o_n,t)$ and $BP_n(t)$ the graph exploration process and the coupled branching process after $t$ steps, we thus have $\G_n(o_n,0)=\{o_n\}$ and $BP_n(0)=\{\varnothing\}$.

For now, let us assume that the following holds \textbf{whp}:
\begin{align*}
    &\textrm{No BAD half-edges are incident to $o_n$, and there is no vertex in any finite radius}\\&\textrm{neighborhood about $o_n$ with a BAD half-edge incident to it.} \numberthis \label{eq:assump_BAD_HE}
\end{align*}
We will describe the coupling under the assumption \eqref{eq:assump_BAD_HE}, and later provide the proof that this is indeed satisfied.

Select an arbitrary half-edge $h$ incident to $o_n$. Say $h \in H_{i \rightarrow j}$, with $(i,j)$ being GOOD. Thus, to pair $h$, we need to sample uniformly a half-edge $h'$ from $H_{j \rightarrow i}$ and pair it with $h$; we do this, and discover an edge in $\G_n$ connecting $o_n$ and a vertex $v$ of type $j$, where $v \in \cN^{(j)}$ is the vertex to which $h'$ is incident. Note that by assumption \eqref{eq:assump_BAD_HE}, there are no BAD half-edges incident to $v$. For the coupling with $\BP_n$, we simply let $\varnothing$ have a child $v$ of type $j$. Thus $\G_n(o_n,1)$ consists of a graph with a single edge $\{o_n,v\}$, while $\BP_n(1)$ has also a single edge $\{\varnothing,v\}$. Let us also declare $H(1):=\{h,h'\}$ to be the set of \textit{used up} half-edges after the first step. In general, $H(t)$ denotes the set of used up half-edges after $t$ steps. 

Say we have constructed $\G_n(o_n,t)$, $BP_n(t)$ and $H(t)$ for some fixed $t \geq 1$. By assumption \eqref{eq:assump_BAD_HE} there are no half-edges incident to any vertex in $\G_n(o_n,t)$. We choose the next half-edge to be paired to construct $\G_n(o_n,t+1)$ in a \emph{breadth-first} manner. Precisely, take a vertex in $\G_n(o_n,t)$ with minimal distance from $o_n$ that has at least one unpaired half-edge incident to it, and choose that half-edge, which we denote by $h$, incident to some vertex $u\in \G_n(o_n,t)$. Now, to pair $h$, we need to choose from the remaining pool $H\setminus H(t)$ of half-edges that are compatible with $h$, i.e., from those that have not already been paired in the first $t$ steps.

The crucial idea of the coupling with $\BP_n(t)$ is the fact that we make already paired half-edges \emph{available} to be paired to $h$ for $\BP_n(t)$. That is to say, to pair $h$ in $\BP_n(t)$, we simply choose uniformly from the set of \emph{all} half-edges compatible to $h$. Let us sample such an half-edge $h'$, and say it is incident to $w\in \cN^{(l)}$. If $h'\in H(t)$ or $w\in \G_n(o_n,t)$, a \emph{conflict} arises, and the coupling breaks down immediately at this step, and we have only managed to couple $\G_n(o_n,k)$ and $\BP_n(k)$ for $k=0,\dots, t$; otherwise, we pair $h'$ with $h$ in $\G_n(o_n,t)$ to give rise to $\G_n(o_n,t+1)$, and let the vertex $u$ in $\BP_n(t)$ have a child $w$ of type $l$, to give rise to $\BP_n(t+1)$. 

Denoting 
\begin{align*}
    \tau_n:=\inf\{t\geq 0:\textrm{ there is a conflict at time }t\}, \numberthis \label{eq:stop_time_tau_n}
\end{align*}
we thus conclude that $\G_n(o_n,t)\simeq_S \BP_n(t)$ for all $t \leq \tau_n$. Now, consider a vertex $w \in \BP_n(t)$ with parent $u$, for some $t<\tau_n$. Let the types of $u$ and $w$ respectively be $j$ and $l$. What is the distribution of the total number of type $i$ children that $w$ can possibly have in $\BP_n(\tau_n)$?

Well, the fact that $u$ has type $j$ tells us that at some previous step, we sampled uniformly a half edge of type $H_{l \rightarrow j}$ (which happened to be incident to $w \in \cN^{(l)}$), and paired it with a half-edge from $H_{j \rightarrow l}$ incident to $u$. Thus, we need to find the degree in $\cN^{(j)}$ of a random vertex in $\cN^{(l)}$ sampled according to probability mass proportional to its degree in $\cN^{(i)}$. Recalling \eqref{eq:law_Dn_ijr} and Remark \ref{rem:assump}, this integer valued random variable has the same distribution as $D_{n,(i,l,j)}$ if $i\neq j$, and $\underline{D_{n,(i,l,i)}}$ if $i=j$. 

Thus, $w$ can have in law $D_{n,(i,l,j)}$ many children in $\BP_n(\tau_n)$ ($\underline{D_{n,(i,l,i)}}$ many if $i=j$), if, of course, we see no conflicts by the time we finish constructing the offspring set of $w$. Analogous to this observation, let us also note that the type of $\varnothing$ in $\BP_n(\tau_n)$ has distribution $T_{n,\varnothing}$, where we recall this variable from \eqref{eq:def_T_nphi}, and given $T_{n,\varnothing}$, the number of type $i$ children of $\varnothing$ in $\BP_n(\tau_n)$ (given that we do not see conflicts in constructing the offspring set of $\varnothing$) has distribution $D_{n,(T_\varnothing , i)}$, where we recall the variable $D_{n,(i,j)}$ from \eqref{eq:the_RV_D_nij}. 

These observations enable us to conclude that $\BP_n(\tau_n)$ has the same distribution as the breadth-first exploration up to $\tau_n$ steps of a 1-dependent branching process $\BP_n:=\BP(\bp_n,\bD_n,\Tilde{\bD}_n)$, where $\bp_n$ is the probability mass vector corresponding to the mass function of $T_{n,\varnothing}$, the $(i,j)$-th entry of the matrix $\bD_n$ is $D_{n(i, \rightarrow j)}$, and the $(i,j,k)$-th entry of the tensor $\Tilde{\bD}_n$ is $D_{n,(i,j,k)}$ for $k \neq i$ and $\underline{D_{n,(i,j,i)}}$ if $k=i$.

Now, if the branching process $\BP_n$ is such that the expected population size up to any finite generation from $\varnothing$, say up to the $r$-th generation, stays uniformly bounded in $n$, a breadth-first exploration of this process up to $\tau_n$ many steps completely reveals $B_{\BP_n}(\varnothing,r)$, as soon as $\tau_n \geq m_n\gg  1$, where $m_n$ is some diverging sequence (no matter how slowly), \textbf{whp}, as $n \to \infty$. In such a situation, we can safely write by the definition of $\tau_n$ and our coupling
\begin{align*}
    \Prob{B_{\G_n}(o_n,r)\simeq_S (\bt,o)}&=\Exp{\CProb{B_{\G_n}(o_n,r)\simeq_S (\bt,o)}{\tau_n\gg m_n}}+o(1)\\&=\Prob{B_{\BP_n}(\varnothing,r)\simeq_S (\bt,o)}+o(1)=\Prob{B_{\cT_\infty}(\varnothing,r)\simeq_S (\bt,o)}+o(1),
\end{align*}
where the final step follows by noting $r$ is a constant, by dominated convergence, using the assumptions \eqref{assump:regularity_1_2} and \eqref{eq:assump_regularity_tensor}, and the definition of $\BP_n$.

To finish the proof of the base case, it remains to check that $\tau_n\geq m_n\gg 1$ for some $m_n$, and that the expected total population size up to any finite generation of $\BP_n$ stays uniformly bounded in $n$. In proving the former, we in fact also verify \eqref{eq:assump_BAD_HE}. Let us continue with this task.

\noindent\paragraph{Divergence of $\tau_n$.}
For any vertex $v$, denote by ${\rm BAD}_v$ the set of all BAD half-edges incident to it. That is, if we assume without loss of generality $v\in \cN^{(j)}$,
\begin{align*}
    {\rm BAD}_v:=\cup_{(j,l) {\rm BAD}}H_{v,j\rightarrow l}.
\end{align*}
Define $\cE_0$ to be the event that $|{\rm BAD}_{o_n}|>0$. Further, for any $t\geq 1$, let us denote by $\cE_t$ the event that at step $t$ of the breadth-first exploration, we paired a half-edge $h$ incident to a vertex in $\G_n(o,t-1)$, to a uniformly sampled half-edge $h'$ compatible with $h$ that is incident to a vertex $v$ satisfying $|{\rm BAD}_v|>0$. At this point note that \eqref{eq:assump_BAD_HE} will be implied if we can show $\Prob{\cup_{0\leq s \leq t}\cE_s}=o(1)$ for any fixed $t\geq 0$. In fact we will prove a stronger statement where we can allow $t=t_n$ to diverge.  

To begin, note that the probability of $\cE_0$ is equivalently seen as the probability that $o_n$ is incident to a vertex $v$ such that $|{\rm BAD}_v|>0$. By definition, since $o_n$ is sampled uniformly, the probability this happens is at most
\begin{align*}
    \frac{1}{n}\sum_{i \in [k]}\sum_{v \in \cN^{(i)}}\sum_{j \in [k]}\ind{d^{(j)}_v>0}\ind{\Exp{D_{i\rightarrow j}}=0}&\leq \frac{1}{n}\sum_{i \in [k]}\sum_{v \in \cN^{(i)}}\sum_{j \in [k]}d^{(j)}_v\ind{\Exp{D_{i\rightarrow j}}=0}\\&\leq  \sum_{i,j\in [k]:(i,j){\rm BAD}}\frac{|H_{i \rightarrow j}|}{N^{(i)}}\leq \sS_n,
\end{align*}
where recall $\sS_n$ from \eqref{eq:def_sn_goodbadseq}. Thus, $\cE_0^c$ occurs with probability at least $(1-\sS_n)$.

Next, conditionally on the event $\cap_{s=0}^{t-1}\cE_s^c$, consider the probability of the event $\cE_t$. Let $h \in H_{i \rightarrow j}$ be the half-edge we are trying to pair at step $t$. First note that since we have conditioned on $\cap_{s=0}^{t-1}\cE_s^c$, the pair $(i,j)$ and thus also $(j,i)$ must be $\rm GOOD$. The probability that at this step we sample uniformly a half-edge $h' \in H_{j \rightarrow i}$ to pair to $h$ such that $h'$ is incident to a vertex $v\in \cN^{(j)}$ with $|{\rm BAD}_v|>0$ is at most (recall $\sS_n$ from \eqref{eq:def_sn_goodbadseq})
\begin{align*}
    \frac{1}{|H_{j\rightarrow i}|}\sum_{l \in [k]:(j,l) {\rm BAD}}\sum_{v \in \cN^{(j)}}d^{(i)}_v\ind{d^{(l)}_v>0}\leq \frac{1}{|H_{j\rightarrow i}|}\sum_{l \in [k]:(j,l) {\rm BAD}}\sum_{v \in \cN^{(j)}}d^{(i)}_vd^{(l)}_v\leq \sS_n,  
\end{align*}
since the pair $(i,j)$ is $\rm GOOD$. We conclude that 
\begin{align*}
    \Prob{\cap_{0\leq s\leq t}\cE_s^c}\geq (1-\sS_n)^t=1-o(1)
\end{align*}
for any $t=t_n$ satisfying $1\ll t_n\ll \sS_n^{-1}$, where note that $t_n$ can be chosen to be diverging (i.e., $t_n \gg 1$) due to Lemma \ref{lem:sn_small}. This checks \eqref{eq:assump_BAD_HE}.

In particular, defining the stopping time
\begin{align*}
    \xi_n:=\inf\{t\geq 0:\cE_t\textrm{ occurs }\},
\end{align*}
we have $\Prob{\xi_n>m_n}\to 1$ for any $1\ll m_n \ll \sS_n^{-1}$. 

Recall our aim is to show $\tau_n$ diverges. Armed with the last observation, to show $\tau_n$ diverges with high probability, it is enough to show that $\tau_n\wedge \xi_n$ diverges. For any $1\ll m_n \ll \sS_n^{-1}$ note that
\begin{align*}
    \Prob{\tau_n\wedge \xi_n<m_n}\leq (1+o(1))\CProb{\tau_n<m_n}{\xi_n>m_n}.\numberthis \label{eq:cond_UB_taun_xin}
\end{align*}

Note that for any $m_n>0$, $\CProb{\tau_n<m_n}{\xi_n>m_n}\leq \CExp{{\rm Conf}(m_n)}{\xi_n>m_n}$, where ${\rm Conf}(m_n)$ is the number of conflicts up to step $m_n$ in the coupling (recall the meaning of a conflict arising from before \eqref{eq:stop_time_tau_n}). We write
\begin{align*}
    \CExp{{\rm Conf}(m_n)}{\xi_n>m_n}=\sum_{s=1}^{m_n}\CProb{\{\textrm{Step }s\textrm{ sees a conflict}\}}{\xi_n>m_n}.
\end{align*}
Conditionally on $\xi_n>m_n$, for some $s\leq m_n$, consider a half-edge $h$ incident to $\G_n(o,s-1)$ that we must pair to obtain $\G_n(o,s)$. Let $h \in H_{i \rightarrow j}$. Since $\xi_n>m_n\geq s$, we observe that $(i,j)$ is $\rm GOOD$. Thus, since at each step of the coupling (before the coupling breaks down) we pair exactly one half-edge,
\begin{align*}
  &\CProb{\{\textrm{Step }s\textrm{ sees a conflict}\}}{\xi_n>m_n}\\&  \leq \frac{s-1}{|H_{j\rightarrow i}|}+\frac{(s-1)\Delta_n(i,j)}{|H_{j \rightarrow i}|}\\&\leq \frac{s-1}{\min_{{\rm GOOD}(i,j)}|H_{j\rightarrow i}|}+\frac{(s-1)\max_{i,j\in[k]}\Delta_n(i,j)}{\min_{{\rm GOOD}(i,j)}|H_{j \rightarrow i}|}, \numberthis \label{eq:expected_conflicts}
\end{align*}
where $\Delta_n(i,j)$ is the maximum degree a vertex in $\cN^{(i)}$ can have in $\cN^{(j)}$. To justify the last display, note that since at step $s$ of the coupling, to construct $BP(s)$ from $BP(s-1)$, we sample a half-edge uniformly at random from the set of all half-edges compatible to $h$, the first term above is an upper bound on the conditional probability, given $\xi_n>m_n$, that we sample a half-edge $h'$ that has already been paired in a previous step, while the second term is an upper bound on the conditional probability that we sample a half-edge that is incident to a vertex in $\cN^{(j)}$ we have already seen in a previous step. All in all, recalling by the definition of $\rm GOOD$, we have $\min_{{\rm GOOD} (i,j)}|H_{j \rightarrow i}|=\Omega(n)$, thanks to the $r=1$ assumption in \eqref{assump:regularity_1_2}, we obtain the overall upper bound,
\begin{align*}
    \CExp{{\rm Conf}(m_n)}{\xi_n>m_n}\leq \sum_{s=1}^{m_n}\left(\frac{s-1}{\Omega(n)}+\frac{(s-1)\Delta_n}{\Omega(n)}\right),
\end{align*}
where the constants hidden by the $\Omega$ terms in the above display are uniform over $s$, and $\Delta_n$ is the maximum degree of $G_n$ (thus also of $\G_n$).

The RHS of the last display is $o(1)$ as soon as $m_n=o(\sqrt{n/\Delta_n})$. Further, even with this restriction on $m_n$, let us agree that we can in fact choose $m_n\gg 1$, since $\Delta_n=o(n)$, following by the $r=2$ assumption in \eqref{assump:regularity_1_2}. Recalling \eqref{eq:cond_UB_taun_xin}, this shows that $\Prob{\tau_n\wedge \xi_n < m_n}=o(1)$ whenever $m_n\gg 1$ is such that $m_n=o\left(\sqrt{n/\Delta_n}\wedge \sS_n^{-1} \right)$. We conclude that the sequence $\tau_n$ diverges with high probability.

\noindent\paragraph{Bounded expected population up to finite generations.} Finally, to conclude the base case, it remains to check that the expected population size up to any finite generation of $\BP_n$ -- say till the $r$-th generation -- stays uniformly bounded. Note that for $M\geq 2$, given the $M$-th generation of $\BP_n$ has in expectation $a(j,i)$ many type $i$ children with a parent of type $j$, for $i,j \in S=[k]$, the expected size of generation $(M+1)$, by an application of Wald's identity, is 
\begin{align*}
    \sum_{i,j\in [k]}a(j,i)\sum_{l\in [k]}\Exp{D_{n,(j,i,l)}},
\end{align*}
where by $D_{n,(j,i,l)}$ let us denote the $(j,i,l)$-th entry of the tensor $\Tilde{\bD}_n$. Thus, given that the expected size of the $M$-th generation stays uniformly bounded in $n$, the above relation shows that the expected size of the $(M+1)$-th generation size stays uniformly bounded in $n$, thanks to assumptions \eqref{assump:regularity_1_2} and \eqref{eq:assump_regularity_tensor}. The check follows by induction together with observing that the expected size of the first generation is
\begin{align*}
    \sum_{i \in k}\frac{N^{(i)}}{n}\sum_{j \in [k]}\Exp{D_n(i\rightarrow j)},
\end{align*}
which is uniformly bounded in $n$, thanks to \eqref{eq:part_i_size} and \eqref{assump:regularity_1_2}.
\medskip

\noindent \textbf{Step II: Induction step.} 
After settling the base case, we now prove the induction step. Recall $o_{n,1},\dots,o_{n,m}$ are uniformly sampled vertices of $\G_n$. Let us write $o_{n,i}=o_i$ for $1\leq i \leq m$ for simplicity.

Let us in fact prove the following stronger statement which clearly implies the result: if we explore the graph $\G_n$ from $o_1,o_2,\dots,o_m$ while pairing half-edges (as discussed in the earlier part of the proof) in a breadth-first manner, up to $m_n\ll\sqrt{n/\Delta_n}\wedge \sS_n^{-1}$ many steps, in some order on these vertices -- say the exploration from $o_1$ first, then from $o_2$, etc. -- then these explorations, call them respectively $\G_n(o_1,m_n),\dots,\G_n(o_m,m_n)$, do not see the pairing of any half-edge incident to a vertex with $\rm BAD$ half-edges incident to it, and can be respectively coupled to be identical with breadth-first exploration up to $m_n$ steps on $\BP^{(1)}_n,\dots,\BP^{(k)}_n$, where the latter are \textit{independent} branching processes, all of whose offspring laws converge to that of $\cT_{\infty}$ as $n \to \infty$.

Indeed, in the first part of the proof, we in fact proved the base case of $m=1$ for this stronger statement. To prove the general statement, assume the result is true for $m-1$ vertices, and we want to prove it for $m$ vertices. Let $H(m_n)$ and $V(m_n)$ respectively be the set of half-edges already paired, and the set of vertices with at least one unpaired half-edge incident to it, after having constructed $\G_n(o_1,m_n),\dots,\G_n(o_{m-1},m_n)$. First note that the probability $o_m$ is not in $V(m_n)$ is at most $(m-1)m_n/n=o(1)$. Thus, we can condition on the complement of this event, and assume $o_m$ is in fact uniformly distributed on $V(m_n)$. For each vertex $v\in V(m_n)\cap \cN^{(j)}$ of type $j$, let $d^{(i),m_n}_v$ denote the number of unpaired half-edges from $H_{j \rightarrow i}$ incident to it, after having constructed $\G_n(o_1,m_n),\dots \G_n(o_{m-1},m_n)$.

The crucial observation is that by the sampling of $\G_n$ (Corollary \ref{cor:sample_Gn}), conditionally on $H(m_n)$ and $V(m_n)$, the breadth-first exploration from a random vertex in $V(m_n)$ has the \textit{same} law as the breadth-first exploration from a random vertex in a smaller graph $\G_n'$, constructed as follows:
\begin{itemize}
    \item For any BAD pair $(i,j)$, we let $G'_{i,j}$ to be a uniform simple random graph with degree sequence $\bd_i(n)$ if $i=j$ and with degree sequence $\bd_{(i,j)}(n)$ if $i \neq j$, as usual.
    \item If $(i,j)$ is good, we let $G'_{i,j}$ to be a \emph{smaller} configuration model ${\tilde CM}^{i,j}_n$ on $\left(\cN^{(i)}\cap V(m_n)\right)\cup \left(\cN^{(j)}\cap V(m_n)\right)$ with bi-degree sequence $(\bd_{i \rightarrow j}'(n),\bd_{j \rightarrow i}'(n))$ if $i \neq j$ and with degree sequence $\bd'_i(n)$ if $i=j$, where:
    \begin{itemize}
        \item for any $i\neq j$, the entry corresponding to $v\in \cN^{(j)}\cap V(m_n)$ in $\bd_{j \rightarrow i}'(n)$ equals $d_v^{(i),m_n}$.
        \item for any $i$, the entry corresponding to $v\in \cN^{(i)}\cap V(m_n)$ in $\bd'_{i}(n)$ equals $d_v^{(i),m_n}$.
    \end{itemize}
    \item Declare $\G'_n:=\cG(G'_{i,j}:i,j\in [k])$.
\end{itemize}

Since $V(m_n)$ differs from $V$ in at most $O(m_n)=o(n)$ many vertices, and taking the difference $H_{i \rightarrow j}\setminus H(m_n)$ removes at most $O(m_n)=o(n)$ many half-edges from $H_{i \rightarrow j}$ for each $i,j \in [k]$, it is straightforward to check that the degree sequences $\bd'_{i \rightarrow j}(n)$ and $\bd'_i(n)$, for all $i\neq j$, $i,j\in [k]$, satisfies the assumptions \eqref{eq:part_i_size}, \eqref{assump:regularity_1_2} and \eqref{eq:assump_regularity_tensor}. Additionally, since we have not touched a $\rm BAD$ half-edge in constructing $\G_n(o_1,m_n),\dots, \G_n(o_{m-1},m_n)$, if we denote by $\sS'_n$ the analogue of the sequence $\sS_n$ (recall it from \eqref{eq:def_sn_goodbadseq}) for the graph $\G_n'$, by the same arguments as above we have $\sS'_n=(1+o(1))\sS_n$. 

In particular, by the first part of the proof, breadth-first exploration from a random vertex $o_m$ in $\G_n'$ up to $m_n$ steps can be coupled to be identical with breadth-first exploration up to these many steps in a branching process $\BP^{(m)}_n$, whose offspring laws converge to that of $\cT_\infty$, provided 
\begin{align*}
    m_n=o\left(\sqrt{|V(m_n)|/\Delta'_n}\wedge (\sS'_n)^{-1}\right)=o\left(\sqrt{| V(m_n)|/\Delta'_n}\wedge \sS_n^{-1}\right),
\end{align*}
where $\Delta'_n$ is the maximum degree (over $i,j\in [k]$, $i \neq j$) from the degree sequences $\bd'_{i \rightarrow j}(n), \bd'_i(n)$. Note that since $\Delta'_n\leq \Delta_n$ and $| V(m_n)|=n(1+o(1))$, $m_n$ clearly satisfies the last condition by recalling $m_n=o(\sqrt{n/\Delta_n})$. Finally, the fact that $\BP^{(m)}_n$ is independent from $\BP^{(1)}_n,\dots,\BP^{(m-1)}_n$ follows by simply noticing that the former processes have to deal with half-edges in $H(m_n)$ and vertices in $V\setminus V(m_n)$, that play no role in constructing $\BP^{(m)}_n$. This concludes the proof of Theorem \ref{thm:loc_lim}. 
\end{proof}
\section{Survival of local limit: proof of Theorem \ref{thm:extinct_1dep}}\label{sec:survival}
Recall the branching process $\BP(\bp,\bD,\tilde{\bD})$ from Definition \ref{def:BP}. We develop the necessary concepts to prove Theorem \ref{thm:extinct_1dep} and give its proof in this section. This section is technical, and we ask the reader to keep the following informal proof idea in mind while reading.

\paragraph{Proof idea.} If the tree $\BP(\bp,\bD,\Tilde{\bD})$ survives with positive probability, this means there must exist an infinite sequence of types $\mathbf{s}=(s_0,s_1,s_2,\dots)$ such that the root has type $s_0$, and has a child of type $s_1$, which has a child of type $s_2$, etc. Because the type space is finite, by a pigeonhole principle argument, there must exist some `loop' of the form $(s_i,s_{i+1},\dots,s_{i+l}=s_i)$ with no repeated types, that recurs infinitely often in $\mathbf{s}$. We can cut out the intermediate parts of the sequence $\mathbf{s}$ to obtain a new sequence $\mathbf{s}'$ such that:
\begin{itemize}
    \item It has an initial part that agrees with $\mathbf{s}$, namely, the part of $\mathbf{s}$ starting from $s_0$ up until the first occurrence of the loop $(s_i,s_{i+1},\dots,s_{i+l}=s_i)$.
    \item After that, the loop simply repeats itself infinitely often.
\end{itemize}
As described in Section \ref{sec:modifying_subtrees}, via this `cutting' operation to obtain $\mathbf{s}'$ from $\mathbf{s}$, we construct a random tree out of the original tree $\BP(\bp,\bD,\Tilde{\bD})$, such that the latter is infinite, and such that all vertices in generation $i$ of it have type given by the $i$-th element of $\mathbf{s}'$. The latter is essentially a Bienaymé-Galton-Watson (BGW) tree, whose survival probability is classical. In particular, via the mentioned construction, we show that if $\BP(\bp,\bD,\Tilde{\bD})$ survives, then there exists a related BGW tree that must also survive, and this immediately gives rise to the negation of \eqref{eq:ext_BP_condition}. This proves the `only if' direction in Theorem \ref{thm:extinct_1dep}.

The `if' direction is much more straightforward as we shall see; as soon as there is a loop $c$ with $\sM(c)>1$, we can sequentially follow the types in this loop and extract a BGW subtree out of $\BP(\bp,\bD,\Tilde{\bD})$ with mean offspring number $\sM(c)>1$, so that this subtree survives and thus so does $\BP(\bp,\bD,\Tilde{\bD})$.

\medskip

To formalize the above idea, we begin by setting up some useful notation. Consider any rooted infinite tree $\T$ with root $\varnothing$, where each vertex $v \in \T$ has a type $T_v \in S$. We further label each vertex in $\T$ by their type. Even though we develop the following notations for a general such tree, we will apply them to the case $\T=\BP(\bp,\bD,\Tilde{\bD})$, an equality we implicitly carry in mind. We think of the edges in $\T$ as \textit{oriented} from parent to child, of the form $e=(u,v)$, where $u$ is the parent of $v$ in $\T$. Naturally, if respectively $T_u,T_v\in S$ be the types of $u$ and $v$, the \emph{type} of such an edge $e=(u,v)$ is the tuple $T(e)=(T_{u},T_v)\in S \times S$. 

\begin{definition}[Type forest, type rays and paths]
    Consider the \emph{Ulam-Harris forest} $\mathscr{F}_S$ on the set of types $S$. That is to say, $\sF_S$ is an infinite forest, consisting of $|S|$ many infinite tree components $\{\sT(s):s\in S\}$. For any fixed $s_0 \in S$, the root of the tree $\sT(s_0)$ is $s_0$, and for any $v\in \sT(s_0)$, it has $|S|$ many offspring $\{vs:s\in S\}$. Observe that any $v\in \sT(s)$ has the form of a word $v=s_0s_1\dots s_k$. We call the forest $\sF_S$ the \textit{type forest} corresponding to the set of types $S$. 
    
    Furthermore any directed infinite path $\pi=(s_0\rightarrow s_0s_1 \rightarrow s_0s_1s_2\rightarrow \dots)$ starting at the root $s_0$ is called a \emph{type ray} in $\sT(s_0)$, and a finite directed path of the form $(s_0\rightarrow s_0s_1 \rightarrow \dots \rightarrow s_0s_1s_2\dots s_k)$ is called a \emph{type path} started at $s_0$. By $|\pi|$ we denote the length of $\pi$. Thus, if $\pi$ is a type ray, $|\pi|=\infty$, while for $\pi=(s_0\rightarrow s_0s_1 \rightarrow \dots \rightarrow s_0s_1\dots s_k)$, $|\pi|=k$.
\end{definition}

Next, let us define two useful ways to cut a type ray/path.

\begin{definition}[Cutting type rays and paths]\label{def:cuts_rays_paths}
    For any type path or ray $\pi=(s_0\rightarrow s_0s_1 \rightarrow s_0s_1s_2 \rightarrow \dots)$ with length $|\pi|$, and for any $1\leq i \leq |\pi|-1$, we define
    \begin{align*}
        &\pi^{(i)}:=(s_i\rightarrow s_is_{i+1}\rightarrow \dots)\textrm{ if }\pi\textrm{ is an infinite type ray; }  \\&\pi^{(i)}:=(s_i\rightarrow s_is_{i+1}\rightarrow \dots \rightarrow s_is_{i+1}\dots s_k)\textrm{ if }\pi\textrm{ is finite and }|\pi|=k.\numberthis \label{eq:def_pi_uppercut}
    \end{align*}
Similarly, for any $0\leq i\leq |\pi|$, we define the path $\pi_{(i)}$ as
\begin{align*}
    \pi_{(i)}:=(s_0\rightarrow s_0s_1 \rightarrow \dots \rightarrow s_0s_1\dots s_{i}).\numberthis \label{eq:def_pi_lowercut}
\end{align*}
    
\end{definition}

Let us next define special sets of descendants of vertices in $\T$ that will be useful. For any $v \in \T$ and $s\in S$, denote by $\bc_\T(v;s)$ the set of all type $s$ children of $v$, where $\bc_\T(v;s)=\emptyset$ if no such child exists.

\begin{definition}[Type path descendants]
    Consider any directed edge $e=(p(v),v)$ in $\T$ where $v\neq \varnothing$, with type $T(e)=(T_{p(v)},T_v)$. Let $s_0=T_{p(v)}$ and $s_1=T_v$, and for a type path $\pi=(s_0\rightarrow s_0s_1 \rightarrow \dots \rightarrow s_0s_1\dots s_k)$ in $\sT(T_{p(v)})$, we inductively define a set $C(e;\pi)$ of descendants of $v$ as follows. If $k=2$ then $C(e;\pi)=\bc_\T(v;s_2)$, while if $k>2$, then $$C(e;\pi)=\cup_{u\in \bc_\T(v;s_2)}C((v,u);\pi^{(1)}),$$ where recall $\pi^{(1)}$ from \eqref{eq:def_pi_uppercut}.
\end{definition}

In words, $C(e;\pi)$ is the set of all type $s_k$ children, of the set of all type $s_{k-1}$ children, \dots , of the set of all type $s_2$ children, of $v$. In the above definition, we use the convention that the union over an empty set is empty.

Certain special classes of type paths will be useful for us, which we define next. 
\begin{definition}[Type circuits]\label{def:type_circuits}
    Given a directed edge $e=(u,v)$ in $\T$ with type $T(e)=(T_u,T_v)=(s_0,s_1)$, a type path $\pi=(s_0\rightarrow s_0s_1 \rightarrow \dots \rightarrow s_0s_1\dots s_k)$ is called a \emph{type circuit} for $e$, if $s_{k-1}=s_0$ and $s_k=s_1$. Define $\sX(e)$ to be the set of all type circuits for the edge $e$.
\end{definition}

Now, consider the process $\BP(\bp,\bD,\Tilde{\bD})$. For any $s_0\in S$ and a path of the form $\pi=(s_0\rightarrow s_0s_1 \rightarrow \dots \rightarrow s_0s_1\dots s_k)$ in $\sT(s_0)$, where $k\geq 2$, given the tensor $\Tilde{\bD}$, we define an integer valued random variable $\sD(\pi)$ inductively as follows. Recall $\pi_{(i)}$ from \eqref{eq:def_pi_lowercut}. For $|\pi|=k=2$ we define $\sD(\pi):=D_{(s_0,s_1,s_2)}$, and for any $k=|\pi|>2$ we inductively define
\begin{align*}
\sD(\pi):=\sum_{i=1}^{\sD(\pi_{(|\pi|-1)})}D^{(i)}_{(s_{k-2},s_{k-1},s_k)}, \numberthis \label{eq:def_offspring_circuit}
\end{align*}
where $\{D^{(i)}_{(s_{k-2},s_{k-1},s_k)}:i\geq 1\}$ is an infinite sequence of i.i.d. copies of the variable $D_{(s_{k-2},s_{k-1},s_k)}$, and the random number $\sD(\pi_{(|\pi|-1)})$ of summands in the above sum is independent of the summands themselves. Observe that for any $u,v\in \BP(\bp,\bD,\Tilde{\bD})$ with $u=p(v)$ and $T(e)=T(u,v)=(s_0,s_1)$, and with $\pi$ as above,
\begin{align*}
  |C(e;\pi)|\stackrel{d}{=} \sD(\pi),\textrm{ and }  \Exp{|C(e;\pi)|}=\Exp{\sD(\pi)}=\prod_{i=0}^{k-2}\Exp{D_{(s_i,s_{i+1},s_{i+2})}}, \numberthis \label{eq:expected_offspring_endofpath}
\end{align*}
where for the second equality above we have repeatedly applied Wald's identity. The last quantity is to be contrasted with $\sM(c)$ as defined in \eqref{eq:loop_product}. Let us continue by stating the main result of this section.

\begin{proposition}[BGW subtrees in $\BP(\bp,\bD,\Tilde{\bD})$]\label{prop:GW_subtrees}
    Consider the branching process $\BP(\bp,\bD,\Tilde{\bD})$. Let $v\neq \varnothing$ in $\BP(\bp,\bD,\Tilde{\bD})$, and let $e:=(p(v),v)$ with $T(e)=(s_0,s_1)$ where $s_0,s_1 \in S$. Consider any type circuit $\pi=(s_0\rightarrow s_0s_1 \rightarrow \dots \rightarrow s_0\dots s_k)$. Define $Z_0:=\{v\}$, and inductively, given $Z_m$, define $$Z_{m+1}:=\cup_{u \in Z_m}C((p(u),u);\pi).$$ Then the $Z_m$ form the set of individuals in the $m$-th generation of a BGW tree which we denote by $\BP_{v,\pi}$. Furthermore, the number of offspring of any vertex in $\BP_{v,\pi}$ is given by $\sD(\pi)$ as defined in \eqref{eq:def_offspring_circuit}.
\end{proposition}
\begin{proof}
Since $\pi$ is a type circuit, for any vertex $a \in \BP_{v,\pi}$ we have $T(p(a),a)=(s_0,s_1)$. Hence, by the observation \eqref{eq:expected_offspring_endofpath}, it is clear that for any two vertices $a_1,a_2 \in \BP_{v,\pi}$, their number of offspring in $\BP_{v,\pi}$ are identically distributed and has the distribution $\sD(\pi)$. Finally, note that from the definition of the process $\BP(\bp,\bD,\Tilde{\bD})$, for any two vertices $u,v \neq \varnothing$ in this process satisfying $T(p(u),u)=T(p(v),v)=(s_0,s_1)$, the numbers $|C((p(u),u);\pi)|$ and $|C((p(v),v);\pi)|$ are independent as long as 
\begin{align*}
    \left(\cup_{i=2}^{|\pi|}C((p(u),u);\pi_{(i)}) \right)\cap \left(\cup_{i=2}^{|\pi|}C((p(v),v);\pi_{(i)}) \right) =\emptyset.
\end{align*}
From the definition of the generations $(Z_m)_{m \geq 0}$, the only way that for any $u,v\in \cup_{m \geq 0}Z_m=\BP_{v,\pi}$ the above intersection is non-empty, is $u=v$. This proves that the number of offspring of distinct individuals in $\BP_{v,\pi}$ are independent. 
\end{proof}

The set of individuals $\left(\cup_{i=2}^{|\pi|}C((p(u),u);\pi_{(i)}) \right)$ that we encountered in the previous argument is special, and we gather it in a definition.

\begin{definition}[Type path/ray induced subtrees]\label{def:type_path_induced}
   Fix $s_0,s_1 \in S$. Consider any rooted tree $\T$ with vertex types in $S$. For any vertex $v \in \T$ with $T_{p(v)}=s_0, T_v=s_1$, and a type path/ray $\pi=(s_0\rightarrow s_0s_1 \rightarrow s_0s_1s_2\rightarrow \dots)$, we define the subtree $\T(v;\pi)$ of $\T$ to be spanned by the vertices
    \begin{align*}
        \{v\}\cup \left(\cup_{i=2}^{|\pi|}C(e;\pi_{(i)})  \right),
    \end{align*}
    where $e$ is the directed edge $(p(v),v)$. We naturally think of $\T(v;\pi)$ to be a rooted tree, rooted at $v$. We call $\T(v;\pi)$ the subtree of $v$ in $\T$ induced by $\pi$.
\end{definition}

\begin{remark}
Observe that when $\pi$ is a type ray, $\T(v;\pi)$ can be infinite, if for example $\T$ is so.
\end{remark}

\subsection{Modifying subtrees induced by type rays to obtain BGW trees }\label{sec:modifying_subtrees}

In this section we develop techniques to modify certain subtrees of $\BP(\bp,\bD,\Tilde{\bD})$ to obtain exact BGW trees from them. We need to develop some definitions first. Consider an infinite type ray $\pi=s_0\rightarrow s_0s_1 \rightarrow \dots$, any rooted tree $\T$ (with vertex types from $S$) and for some $v \in \T$ with $p(v)=u$, $T_{u}=s_0$ and $T_v=s_1$. Recall the subtree $\T(v;\pi)$ induced by $\pi$. Consider further collections of positive integers $\{M_i: i \in I\}$ and $\{N_i:i\in I\}$ where $I$ is some countable index set, satisfying 
\begin{align*}
&M_i<N_i \textrm{ for all }i\in I,[M_i,N_i]\cap[M_j,N_j]=\emptyset\textrm{ for all }i,j\in I, i\neq j,\\&\textrm{ with }s_{M_i+j}=s_{N_i+j}\textrm{ for all }i\geq 1\textrm{ and }j\in\{-1,0,1\},\numberthis \label{eq:integer interval condition}
\end{align*}
where for positive integers $M<N$ we define $[M,N]:=\{M,M+1,\dots,N\}$. For any fixed $i \in I$, corresponding to the pair $(M_i,N_i)$, we define a transformation of $\pi$ via a map $\tau_i$ as,
\begin{align*}
    \tau_i(\pi):=s_0\rightarrow s_0s_1\rightarrow \dots \rightarrow s_0\dots s_{M_i-1}\rightarrow s_0\dots s_{M_i-1}s_{N_i}\rightarrow s_0\dots s_{M_i-1}s_{N_i}s_{N_i+1}\rightarrow \dots. \numberthis \label{eq:tau_i_map}
\end{align*}
In words, to construct $\tau_i(\pi)$, we take $\pi_{(M_i-1)}=s_0\rightarrow s_0s_1\rightarrow \dots \rightarrow s_0\dots s_{M_i-1}$, and attach (or concatenate) to the end vertex $s_0\dots s_{M_i-1}$ of this finite path the infinite ray $\pi^{(N_i)}=s_{N_i} \rightarrow s_{N_i}s_{N_i+1}\rightarrow \cdots$.

Fix some enumeration of the set $I$ as $I=\{i_1,i_2,\dots\}$, and let us also define the ray $\tau_\infty(\pi)$ by
\begin{align*}
  \tau_\infty(\pi):= \lim_{N \to \infty}\tau_{i_N}(\tau_{i_{N-1}}(\dots\tau_{i_2}(\tau_{i_1}(\pi))\dots))= \cdots \tau_{i_3}(\tau_{i_2}(\tau_{i_1}(\pi))) \cdots. \numberthis \label{eq:operator_tau_infty}
\end{align*}
In the last display, it is intuitively clear what the limit means -- we are simply composing functions sequentially. It is possible to define a topology and make rigorous sense of the limit above, but we choose to avoid these technicalities as they are not very illuminating for our purposes.

Since the sets $[M_i,N_i]$ and $[M_j,N_j]$ do not intersect for different values of $i$ and $j$, it should be clear that the definition of $\tau_\infty(\pi)$ does not depend on the sequence of compositions, or equivalently, on the specific enumeration of $I$ that we choose. That is to say, for any bijection $\sigma:\N \to \N$ we have the equality
\begin{align*}
    \tau_\infty(\pi)= \lim_{N \to \infty}\tau_{\sigma(i_N)}(\tau_{\sigma(i_{N-1})}(\dots\tau_{\sigma(i_2)}(\tau_{\sigma(i_1)}(\pi))\dots)). \numberthis \label{eq:ray_perm_inv}
\end{align*}

When $I$ is finite, say $I=\{i_1,\dots,i_l\}$, our (natural) convention is $\tau_{i_N}(\tau_{i_{N-1}}(\dots\tau_{i_2}(\tau_{i_1}(\pi))\dots))=\tau_{i_l}(\tau_{i_{l-1}}(\dots\tau_{i_2}(\tau_{i_1}(\pi))\dots))$ for all $N\geq l$, so that the above limit is essentially over a constant sequence.

Before proceeding, given a type ray or path $\pi=(s_0\rightarrow s_0s_2\rightarrow \dots)$, let us denote by $\BP_\pi$ a branching process grown as follows: the root of this process has type $s_1$, and for $m\geq 1$ every individual in the $m$-th generation gives birth to $D_{(s_{m-1},s_m,s_{m+1})}$ many offspring of type $s_{m+1}$. We label the individuals by their types in $\BP_\pi$. Thus, note that for a vertex $v \in \BP(\bp,\bD,\Tilde{\bD})$ with $T(u,v)=(s_0,s_1)$ where $p(v)=u$, taking $\T=\BP(\bp,\bD,\Tilde{\bD})$, $\BP_\pi$ has the same distribution as $\T(v;\pi)$.

\paragraph{Constructing $\psi_i(\T(v;\pi))$.} Consider any (random or deterministic) rooted tree $\T$ with vertex types in $S$, and a vertex $v \in \T$ with $T(u,v)=(s_0,s_1)$ where $p(v)=u$. Recall the sets $[M_i,N_i]$ satisfying \eqref{eq:integer interval condition}.

\begin{definition}[Definition of $\psi_i(\T(v;\pi))$]\label{def:psi}
For any fixed $i \in I$, corresponding to $M_i,N_i$, and corresponding to the tree $\T(v;\pi)$ induced by the ray $\pi$, we want to construct a (random) map $\psi_i$ such that $\psi_i(\T(v;\pi))$ is another (random) tree, defined as follows:
    \begin{itemize}
    \item If $C(e;\pi_{(M_i+1)})=C(e;s_0\rightarrow s_0s_1 \rightarrow \dots \rightarrow s_0s_1\dots s_{M_i+1})=\emptyset$, declare $\psi_i(\T(v;\pi))=\T(v;\pi)$.
    \item Otherwise, if $C(e;\pi_{(M_i+1)})\neq \emptyset$, there is an $x \in C(e;\pi_{(M_i+1)})$. Let $p(x)$ be its parent. Note that $x$ is a leaf at height $M_i$ in the tree $\T(v;\pi_{(M_i+1)})$. Consider the following set of descendants of $x$ in $\T(v;\pi)$,
    \begin{align*}
    &C((p(x),x);(\pi_{(N_i+1)})^{(M_i)})\\&=C((p(x),x);s_{M_i}\rightarrow s_{M_i}s_{M_i+1}\rightarrow \dots \rightarrow s_{M_i}s_{M_i+1}\dots s_{N_i}s_{N_i+1}).    
    \end{align*}
    If $C((p(x),x);(\pi_{(N_i+1)})^{(M_i)})=\emptyset$ for all such $x$ in $\T(v;\pi_{(M_i+1)})$, then drop i.i.d. copies of the tree $\BP_{\pi^{(N_i)}}$, independent from $\T$, from all such leaves $x$ at distance $M_i$ in the tree $\T(v;\pi_{M_i+1})$. This just means, for each such leaf $x$, we take a copy of the tree $\BP_{\pi^{(N_i)}}$, and identify the root of this tree with $x$ (note that the type $T_x$ of any such leaf $x$ is $s_{M_i+1}=s_{N_i+1}$, so that we are identifying two individuals of the same type, even though they live in independent processes). Declare the resulting tree to be $\psi_i(\T(v;\pi))$.  
    \item Otherwise, if for some $x \in C(e;\pi_{(M_i+1)})$, $C((p(x),x);(\pi_{(N_i+1)})^{(M_i)})\neq \emptyset$, choose arbitrarily some $y \in C((p(x),x);(\pi_{(N_i+1)})^{(M_i)})$, and let $p(y)$ be the parent of $y$. Note that $T_{p(x)}=s_{N_i}=s_{M_i}=T_{p(y)}$ and $T_{x}=s_{N_i+1}=s_{M_i+1}=T_{y}$. Consider the tree $\T(y;\pi^{(N_i)})=\T(y;\pi_{N_i}\rightarrow \pi_{N_i+1}\rightarrow \cdots)$. Take this tree and drop it from $x$ in the tree $\T(v;\pi_{(M_i+1)})$, i.e., \emph{graft} the root $y$ of $\T(y;\pi^{(N_i)})$ at $x$ in $\T(v;\pi_{(M_i+1)})$. Again observe that we are grafting a vertex on another with identical type. Do this procedure for every $x \in C(e;\pi_{(M_i+1)})$, for which $C((p(x),x);(\pi_{(N_i+1)})^{(M_i)})\neq \emptyset$, each time choosing an $y \in C((p(x),x);(\pi_{(N_i+1)})^{(M_i)})$ in some arbitrary fashion. For those $x \in C(e;\pi_{(M_i+1)})$ for which $C((p(x),x);(\pi_{(N_i+1)})^{(M_i)})=\emptyset$, simply drop i.i.d. copies of the tree $\BP_{\pi^{(N_i)}}$ from such $x$ in $\T(v;\pi_{(M_i+1)})$. Call the resulting tree $\psi_i(\T(v;\pi))$. 
\end{itemize}
\end{definition}

The key result with Definition \ref{def:psi} is the following distributional relation.

\begin{proposition}\label{prop:dist_iden_looperase}
    Consider any 1-dependent branching process $\BP(\bp,\bD,\Tilde{\bD})$, and a vertex $v$ in it with $p(v)=u$ and such that $T_{p(v)}=s_0, T_v=s_1$. Furthermore, consider any infinite ray of the form $\pi=s_0 \rightarrow s_0s_1 \rightarrow \cdots$, and collections of positive integers $\{M_i:i\in I\}$, $\{N_i:i\in I\}$ such that conditions \eqref{eq:integer interval condition} hold, where $I$ is a countable index set. Then for any subset $\{i_1,\dots,i_N\}\subseteq I$, taking $\T=\BP(\bp,\bD,\Tilde{\bD})$, viewed as trees labeled by vertex types, there is the distributional identity
    \begin{align*}
       \psi_{i_N}(\dots \psi_{i_2}(\psi_{i_1}(\T))\dots)\stackrel{d}{=}\T(v;\tau_{i_N}(\dots\tau_{i_2}(\tau_{i_1}(\pi))\dots)). \numberthis \label{eq:loop_erase_identity_finite}
    \end{align*}
    In particular, letting $N \to \infty$ above, for the case when $I$ is infinite and $I=\{i_1,i_2,\dots\}$ is some enumeration of it, 
    \begin{align*}
        \psi_\infty(\T):=\dots \psi_{i_3}(\psi_{i_2}(\psi_{i_1}(\T)))\dots\stackrel{d}{=}\T(v;\tau_\infty(\pi)). \numberthis \label{eq:loop_erase_identity_infinite}
    \end{align*}
\end{proposition}

\begin{remark}
    By \eqref{eq:ray_perm_inv} and \eqref{eq:loop_erase_identity_infinite}, note that the distribution of $\psi_{\infty}(\T)$ is invariant under any bijection $\sigma: \N \to \N$, i.e.,
    \begin{align*}
    \lim_{N \to \infty}  \psi_{i_N}(\dots \psi_{i_2}(\psi_{i_1}(\T))\dots)\stackrel{d}{=} \lim_{N \to \infty}  \psi_{\sigma(i_N)}(\dots \psi_{\sigma(i_2)}(\psi_{\sigma(i_1)}(\T))).
    \end{align*}
\end{remark}

\begin{proof}[Proof of Proposition \ref{prop:dist_iden_looperase}]
    Let us first prove the result when $|I|=1$. Thus we have a pair of integers $M_i<N_i$ satisfying the second line of the conditions \eqref{eq:integer interval condition}. Recall the construction of the tree $\psi_{i}(\T(v;\pi))$. Following the three possibilities of the construction, we do a case-by-case analysis. 
    
    First, obviously if $C(e;\pi_{(M_i+1)})=\emptyset$, there is nothing to prove as $\psi_i(\T(v;\pi))=\T(v;\pi)=\T(v;\tau_i(\pi))$, simply because $\T=\BP(\bp,\bD,\Tilde{\bD})$ does not even survive up to generation $M_i$. 
    
    Second, consider the case that $C(e;\pi_{(M_i+1)})\neq \emptyset$ but for all $x\in C(e;\pi_{(M_i+1)})$ we have $C((p(x),x);(\pi_{(N_i+1)})^{(M_i)})=\emptyset$. Note that in this case, $\psi_i(\T)(v;\pi_{(M_i+1)})=\T(v;(\tau_i(\pi))_{(M_i+1)})$. Thus, by the branching property, it is enough to check that for each $x \in C(e;\pi_{(M_i+1)})$, the subtrees $\psi_i(\T)(x;\pi^{(M_i)})$ agree in distribution with the corresponding subtrees $\T(v;\tau_1(\pi))(x;\pi^{(M_i)})$ of $\T(v;\tau_1(\pi))$. Note that since $\T=\BP(\bp,\bD,\Tilde{\bD})$, by the definition of the map $\tau_1$, we have $\T(v;\tau_1(\pi))(x;\pi^{(M_i)})=\T(x;\pi^{(N_i)})\stackrel{d}{=}\BP_{\pi^{(N_i)}}$. Further these subtrees are independent across different $x\in C(e;\pi_{(M_i+1)})$. On the other hand, by the construction of the tree $\psi_i(\T)$, $\psi_i(\T)(x;\pi^{(M_i)})\stackrel{d}{=}\BP_{\pi^{(N_i)}}$. The result now follows.  
    
    Finally, consider the case that there are $x \in C(e;\pi_{(M_i+1)})$ with $y\in C((p(x),x);(\pi_{(N_i+1)})^{(M_i)})$.  As before, by the branching property, it is enough to check that for each $x \in C(e;\pi_{(M_i+1)})$,
    \begin{align*}
    \psi_i(\T)(x;\pi^{(M_i)})\stackrel{d}{=}\T(v;\tau_1(\pi))(x;\pi^{(M_i)}). \numberthis \label{eq:subtree_dist_id}    
    \end{align*}
     As in the last case, for those $x\in C(e;\pi_{(M_i+1)})$ with $C((p(x),x);(\pi_{(N_i+1)})^{(M_i)})=\emptyset$, we can similarly argue \eqref{eq:subtree_dist_id}. Thus, we need only check \eqref{eq:subtree_dist_id} for $x \in C(e;\pi_{(M_i+1)})$ such that there is some $y\in  C((p(x),x);(\pi_{(N_i+1)})^{(M_i)})$. To see this, by the definition of the map $\psi_i$, we note that $\psi_i(\T)(x;\pi^{(M_i)})=\T(y;\pi^{(N_i)})\stackrel{d}{=}\BP_{N^{(i)}}$. And as before, we note that $\T(v;\tau_1(\pi))(x;\pi^{M_i})=\T(x;N^{(i)})\stackrel{d}{=}\BP_{\pi^{(N_i)}}$. This concludes \eqref{eq:subtree_dist_id}.

    Now the general result \eqref{eq:loop_erase_identity_finite} follow because of an inductive argument. To establish \eqref{eq:loop_erase_identity_finite} given the statement is true when $i_{N}$ is replaced by $i_{N-1}$, we can couple the two trees $\cT_1=\psi_{i_{N-1}}(\dots \psi_{i_2}(\psi_{i_1}(\T))\dots)$ and $\cT_2=\T(v;\tau_{i_{N-1}}(\dots\tau_{i_2}(\tau_{i_1}(\pi))\dots))$ in some probability space where they are equal. Working in this space, writing $\cT_2=\T(v;\pi')$ where $\pi'=\tau_{i_{N-1}}(\dots\tau_{i_2}(\tau_{i_1}(\pi))\dots)$ we have
    \begin{align*}
        \cT_1=\psi_{i_{N}}(\dots \psi_{i_2}(\psi_{i_1}(\T))\dots)=\psi_{i_N}(\cT_1)&=\psi_{i_N}(\cT_2)\\&\stackrel{d}{=}\cT_2(v;\tau_{i_N}(\pi'))=\T(v;\tau_{i_N}(\pi'))=\tau_{i_{N}}(\dots\tau_{i_2}(\tau_{i_1}(\pi))\dots),
    \end{align*}
    establishing \eqref{eq:loop_erase_identity_finite}, where for the distributional identity in the above display we have employed the result for $|I|=1$.

    Finally, note that \eqref{eq:loop_erase_identity_infinite} is a consequence of \eqref{eq:loop_erase_identity_finite}, by taking limits.\footnote{Ignoring the topological technicalities as usual, but note that, as discussed before, this is again just composing functions in a sequential order.}
\end{proof}

To approach the proof of Theorem \ref{thm:extinct_1dep}, we need to marry the notions of type rays and paths with the type graph $G_S$ (recall Definition \ref{def:type graph}). Let us do this next. First note that given any type path (respectively ray) $\pi=s_0\rightarrow s_0s_1 \rightarrow s_0s_1s_2 \rightarrow \dots $, we can naturally associate to it a finite (respectively infinite) path $\phi(\pi)$ via a canonical mapping $\phi$ in the graph $G_S$ as
\begin{align*}
    \phi(\pi)=((s_0,s_1,s_2), (s_1,s_2,s_3), \dots).
\end{align*}
Clearly, the map $\phi$ is invertible and let us denote its inverse by $\phi^{-1}$.  Consider any infinite ray $\pi=s_0 \rightarrow s_0s_1 \rightarrow \cdots$, such that there exist integers $1\leq M_i< N_i$ satisfying conditions \eqref{eq:integer interval condition}. Note that the segment $((s_{M_i-1},s_{M_i},s_{M_i+1}),(s_{M_i},s_{M_i+1},s_{M_i+2}),\dots , (s_{N_i-1},s_{N_i},s_{N_i+1}) )$ of the path $\phi(\pi)$ is then a loop at the vertex $(s_{M_i-1},s_{M_i},s_{M_i+1})\in V(G_S)$. This allows us to translate the identities \eqref{eq:loop_erase_identity_finite} and \eqref{eq:loop_erase_identity_infinite} in terms of \emph{loop erasures} on infinite paths in $G_S$. Let us do this next.

For any infinite path $\xi=(v_0,v_1,\dots)$ in $G_S$, and a loop in this path of the form $(v_M,v_{M+1},\dots,v_{M})$ where $v_M=v_N$, we denote the loop erasure of $\xi$ after erasing this loop as the new path
\begin{align*}
    \xi\setminus[M,N]:=(v_0,v_1,\dots,v_{M-1},v_N,v_{N+1},\dots).\numberthis \label{eq:one_loop_erase}
\end{align*}

Next, consider a countable collection of loops $\{(v_{M_i},v_{M_i+1},\dots,v_{N_i}):i \in I\}$. Also assume the integers $M_i$ and $N_i$ satisfy conditions \eqref{eq:integer interval condition}. In particular, note that since the sets $[M_i,N_i]$ and $[M_j,N_j]$ do not intersect for $i \neq j$, it does not matter in which order we erase these loops in the path $\xi$. Assuming $I=\{i_1,i_2,\dots\}$ where $N_{i_j}<M_{i_{j+1}}$ for all $j \geq 1$, this lets us define the path obtained by erasing all these loops (in some order)
\begin{align*}
    \xi\setminus\left(\cup_{i \in I}[M_{i},N_i] \right):=(v_0,\dots,v_{M_{i_1-1}},v_{N_{i_1}},\dots,v_{M_{i_2-1}},v_{N_{i_2}},\dots,v_{M_{i_3-1}},v_{N_{i_3}},\dots). \numberthis \label{eq:multiple_loop_erase}
\end{align*}

In this light, we can translate identity \eqref{eq:loop_erase_identity_infinite} in the following manner. Assume $\xi=(v_0,v_1,\dots)$ and $\pi=(s_0\rightarrow s_0s_1 \dots )$ are respectively an infinite path in $G_S$ and an infinite type ray such that $\phi(\pi)=\xi$, or, equivalently, $v_i=(s_i,s_{i+1},s_{i+2})$ for all $i \geq 0$. Assume further the existence of $\{M_i:i\in I\}$ and $\{N_i:i\in I\}$ satisfying \eqref{eq:integer interval condition}. Recalling the definition of the map $\tau_\infty$ from the statement of Proposition \ref{prop:dist_iden_looperase}, observe that
\begin{align*}
    \phi^{-1}(\xi\setminus\left(\cup_{i \in I}[M_i,N_i] \right))=\tau_\infty(\pi).
\end{align*} 
Thus, we can write \eqref{eq:loop_erase_identity_infinite} as
\begin{align*}
    \psi_\infty(\T)\stackrel{d}{=} \T(v;\phi^{-1}(\xi\setminus\left(\cup_{i \in I}[M_i,N_i] \right))).
\end{align*}
An analogous identity for \eqref{eq:loop_erase_identity_finite} follows similarly.

We can now turn to the proof of Theorem \ref{thm:extinct_1dep}.

\begin{proof}[Proof of Theorem \ref{thm:extinct_1dep}]
   \textbf{Extinction implies \eqref{eq:ext_BP_condition}:}  Let us first assume $\eta =1$, thus almost surely, $\T=\BP(\bp,\bD,\Tilde{\bD})$ goes extinct. We aim to show \eqref{eq:ext_BP_condition}. If not, there is a term in the maximum defining the LHS of \eqref{eq:ext_BP_condition} that is strictly greater than $1$. That is, there is a $s\in S$ with $p_s>0$, such that there is a $s'\in S$ with $p^{(s,s')}_0<1$, such that there is a $v \in V(s,s')$, such that there is a $u \in \cC_D(v)$, such that there is a $c \in \cL(u)$ with $\sM(c)>1$. Note that $v$ has the form $(s_0,s_1,s_2)$ with $s_0=s$ and $s_1=s'$. Since $u \in \cC_D(v)$ there is a path $(v_0,v_1,\dots,v_l)$ connecting $v$ and $u$ in $D(G_S)$, where $v_0=v$, $v_l=u$ and $\Exp{D_{v_i}}\Exp{D_{v_{i+1}}}>0$ for each $0\leq i \leq l-1$. Construct the infinite path $\xi$ by taking the initial segment $(v_0,\dots, v_l)$ and then concatenating the loop $c \in \cL(u)$ infinitely many times to it. That is, if $c=(w_0,\dots, w_{m})$ with $w_0=w_m=u$, 
    \begin{align*}
        \xi=(v_0,v_1,v_2,\dots),\textrm{ with }v_{l+i}=v_{l+i+Km}=w_i\textrm{ for all }0\leq i \leq m \textrm{ and }  K \geq 1. \numberthis \label{eq:def_path_inf_concat_loop}
    \end{align*}
Observe that since $p_s>0$ and $p^{(s,s')}_0<1$, with positive probability the root has type $T_\varnothing=s$ and the root has a child of type $s'$. Let us call this event $\zeta(s,s')$, i.e.,
\begin{align*}
    \zeta(s,s'):=\{T_\varnothing=s,\textrm{ and there exists a child }\by\textrm{ of }\varnothing\textrm{ with }T_{\by} =s'\} \numberthis \label{eq:def_zeta_event},
\end{align*}
and condition on it. We claim that conditionally on $\zeta(s,s')$, the tree $\T=\BP(\bp,\bD,\Tilde{\bD})$ survives forever, which is a contradiction to our assumption, and this will prove \eqref{eq:ext_BP_condition} by contradiction.


To see our claim, given the event $\zeta(s,s')$, let $\mathbf{v}$ be a child of $\varnothing$ such that with $e=(\varnothing,\mathbf{v})$ we have $T(e)=(s,s')$. It is enough to show that given the event $\zeta(s,s')$ (which has positive probability), the conditional probability that the tree $\T(\mathbf{v};\pi)$ is infinite is positive, where $\pi=\phi^{-1}(\xi)$ with $\xi$ defined as in \eqref{eq:def_path_inf_concat_loop}. As before, let us denote $\pi=s_0 \rightarrow s_0s_1 \rightarrow \dots$, so that $v_i=(s_i,s_{i+1},s_{i+2})$. Let us denote the initial segment $(v_0,\dots,v_l)$ of the path $\xi$ by $\xi'$, and consider the type path $\phi^{-1}(\xi')$. Observe that since $u=v_l \in \cC_D(v)$, the quantity $\Exp{D_{v_0}}\Exp{D_{v_1}}\dots \Exp{D_{v_l}}$ is strictly positive. On the other hand, by  repeatedly applying Wald's identity,
\begin{align*}
\CExp{C(e;\phi^{-1}(\xi'))}{\zeta(s,s')}=\Exp{D_{v_0}}\Exp{D_{v_1}}\dots \Exp{D_{v_l}}>0.    
\end{align*}
Thus, given $\zeta(s,s')$, with positive conditional probability there is a vertex $\by \in C(e;\phi^{-1}(\xi'))$. Note that $u=v_l=(s_l,s_{l+1},s_{l+2})$, so that $T_{\by}=s_{l+2}$ and $T_{p(\by)}=s_{l+1}$. It is enough to show that with positive conditional probability, the subtree of this vertex in $\T(\bv;\pi)$ is infinite, i.e., the tree $\T(\by;(\phi^{-1}(\xi))^{(l+1)})=\cT(\by;s_{l+1}\rightarrow s_{l+1}s_{l+2}\rightarrow \cdots)$  is infinite.

Note that the infinite path $\phi\left( (\phi^{-1}(\xi))^{(l+1)}\right)=(v_{l+1},v_{l+2},\dots)$ in $G_S$ is an infinite concatenation of the loop $(w_1,w_2,\dots,w_{m-1}, w_m,w_1)$. Recall Proposition \ref{prop:GW_subtrees}. Observe that letting $e'=(p(\by),\by)$, and $\pi'=\phi^{-1}(w_1,w_2,\dots,w_{m})=(s_{l+1}\rightarrow s_{l+1}s_{l+2}\rightarrow \dots \rightarrow s_{l+1}\dots s_{l+m+2})$, by Proposition \ref{prop:GW_subtrees}, the tree $\BP_{\by,\pi'}$ is a BGW tree with mean offspring number $\sM(c')$, where $c'$ is the loop $(w_1,w_2,\dots,w_{m-1},w_m, w_1)$. However $\sM(c')=\sM(c)>1$ for $c\in \cL(u)$ with $c=(w_0,\dots,w_m)$, where the last inequality holding by our assumption that the LHS of \eqref{eq:ext_BP_condition} is strictly larger than $1$. Thus, $\BP_{\by,\pi'}$ survives forever. Since each generation of $\BP_{\by,\pi'}$ is a sub-generation of $\T(\by;(\phi^{-1}(\xi))^{(l+1)})$ (namely, the $n$-th generation of $\BP_{\by,\pi'}$ is exactly the $(m+1)n$-th generation of $\T(\by;(\phi^{-1}(\xi))^{(l+1)})$), so does the latter tree, thus, so does $\T=\BP(\bp,\bD,\Tilde{\bD})$, and this is a contradiction to $\eta=1$. By contradiction the claimed implication is proved.

\medskip

\textbf{\eqref{eq:ext_BP_condition} implies extinction:} We now turn to the other direction, i.e., assuming the condition \eqref{eq:ext_BP_condition} we show that $\T=\BP(\bp,\bD,\Tilde{\bD})$ dies out almost surely. Recall the event $\zeta(s,s')$ from \eqref{eq:def_zeta_event}. If there is no pair $s,s'\in S$ with $\Prob{\zeta(s,s')}>0$, then we have extinction, and there is nothing to prove; thus, let us assume there exists $s,s'\in S$ with $\Prob{\zeta(s,s')}>0$. Conditionally on the event $\zeta(s,s')$, for any $\bv$ that is a child of $\varnothing$ with $T_\bv=s'$, it is sufficient to show that for any type ray $\pi=(s_0\rightarrow s_0s_1\rightarrow \dots)$ where $s_0=s, s_1=s'$, the tree $\T(\bv;\pi)$ dies out. Let us prove this by contradiction.

For a contradiction, assume the existence of such a type ray $\pi$ such that $\T(\bv;\pi)$ survives with positive conditional probability given $\zeta(s,s')$. Consider the infinite path $\phi(\pi)=\xi=(v_0,v_1,\dots)$ in $G_S$, with $v_i=(s_i,s_{i+1},s_{i+2})$ for all $i \geq 0$. Since $G_S$ is a finite graph, by the pigeonhole principle, there has to be some $u\in G_S$ and and some $c \in \cL(u)$ such that $c$ repeats infinitely often in the path $\phi(\pi)$. Let the indices of occurrences of $c$ be in order $i_1<i_2<\dots $, i.e., if $c=(w_0,\dots,w_m)$ with $w_0=w_m=u$, then $u=w_0=v_{i_1}=v_{i_2}=\dots , w_1=v_{i_1+1}=v_{i_2+1}=\dots,\dots, w_m=v_{i_1+m}=v_{i_2+m}=\dots=u$ . Observe that it is possible to have ${i_j+m}={i_{j+1}}$, if, for example, the $j$-th and $(j+1)$-th appearance of the loop $c$ are consecutive. 

Consider integers $1\leq M_i <N_i$ satisfying the condition \eqref{eq:integer interval condition}. Recall the definition of the map $\psi_i$ from Definition \ref{def:psi}. For any vertex $v \in \T=\BP(\bp,\bD,\Tilde{\bD})$ with $T_{p(v)}=s_0$ and $T_v=s_1$, it is easily checked by the construction of the tree $\psi_i(\T(v;\pi))$ that if the tree $\T(v;\pi)$ survives forever, then so does the tree $\psi(\T(v;\pi))$ (e.g., since $\T(v;\pi)$ survives, note that only the third point in the construction is relevant, and in that case we can always find a $y$ as specified in the construction). Further this tree has the same distribution as $\cT(v;\tau_i(\pi))$ due to Proposition \ref{prop:dist_iden_looperase}. In fact, there is nothing special about a single interval $[M_i,N_i]$, and it is easily checked that for any collection of integers $\{M_i:i\in I\}$, $\{N_i:i\in I\}$ satisfying \eqref{eq:integer interval condition}, the tree $\psi_\infty(\cT(v;\pi))$ survives forever if $\cT(v;\pi)$ does so, and has the same distribution as $\cT(v;\tau_\infty(\pi))$, by a simple inductive argument similar to the one appearing in the proof of Proposition \ref{prop:dist_iden_looperase}.

We now specify the particular integer collections $\{M_i\}_{i \in I}$ and $\{N_i\}_{i \in I}$ for which we apply this result. Recall the indices $i_1<i_2<\dots$ of the occurrences of the loop $c=(w_0,\dots,w_m)$ in the path $\xi=\phi(\pi)$. Let $l_1:=\inf\{j \geq 1:i_j+m\neq i_{j+1}\}$. Set $M_1=i_{l_1}+m$ and $N_1=i_{l_1+1}$. Then having defined $l_{p-1}$ for $p \geq 2$, we inductively define $l_p=\inf\{j \geq i_{l_{p-1}+1}:i_j+m\neq i_{j+1}\}$, and set $M_p=i_{l_p}+m$ and $N_p=i_{l_p+1}$. In words, either the occurrences of $c$ on the path are consecutive, or they are separated by another (not necessarily simple) loop at $u$. Then $l_p$ is the $p$-th index (in increasing order) such that the $l_p$-th and the $(l_p+1)$-th occurrence of the loop $c$ on the path $\xi$ are non-consecutive. Then we set $M_p$ to be $i_{l_p+m}$ and $N_p=i_{l_p+1}$.

As before, we define the maps $\psi_i$ and $\tau_i$ corresponding to our specified $M_i$ and $N_i$ as in the last paragraph, for $i \geq 1$, and the limiting maps $\psi_\infty$ and $\tau_\infty$. Note that in this case we simply have $I=\N$. Taking $v=\bv$, as argued before, since $\T(\bv;\pi)$ survives forever, the tree $\psi_{\infty}(\T(\bv;\pi))$ survives forever, which also has the same distribution as $\T(v;\tau_\infty(\pi))$.  However, by the definition of the integers $M_i$ and $N_i$, this tree can also be written as $\T(\bv;\phi^{-1}(\xi'))$, for the path $\xi'=\xi\setminus(\cup_{j\geq 1, i_j+m\neq i_{j+1}}[i_j+m,i_{j+1}])$, recall the notation \eqref{eq:multiple_loop_erase}, which is essentially the path $(v_0,v_1,\dots,v_{i_1})$ together with the loop $c$ concatenated infinitely many times at its end.

 Since $\cT(\bv;\phi^{-1}(\xi'))$ is infinite, the expected number of vertices in generation $i_1+2$ of this tree is strictly positive, which by Wald's identity is equivalent to saying,
\begin{align*}
    \Exp{D_{v_0}}\Exp{D_{v_1}}\dots \Exp{D_{v_{i_1}}}>0,
\end{align*}
or, equivalently, $v_{i_1} \in \cC_D(v_0)$. Next, choose any $\by$ in this generation and note that $T_{\by}=s_{i_1+2}$. Define $\pi'=(s_{i_1+1}\rightarrow s_{i_1+1}s_{i_1+2}\rightarrow \dots \rightarrow s_{i_1+1}\dots s_{i_1+m+2})$. Note that if we inductively let $Z_0=\{\by\}$ and given $Z_{n-1}$, $Z_n=\cup_{u \in Z_{n-1}}C((p(u),u);\pi')$, then by Proposition \ref{prop:GW_subtrees}, $Z_n$ forms the $n$-th generation of a BGW tree $\BP_{\by,\pi'}$ with mean offspring number $\sM(c)$. Further, since $\cT(\bv;\phi^{-1}(\xi'))$ is infinite, this BGW tree survives, which means $\sM(c)>1$. Since $p_s>0,p_0^{(s,s')}<1$, $v_0\in V(s,s')$, $i_1\in \cC_D(v_0)$ and $\sM(c)>1$, this violates \eqref{eq:ext_BP_condition}, and we have arrived at our contradiction. \end{proof}

\section{Existence of giants: proof of Theorem \ref{thm:giants_MCMs}}\label{sec:giant_proof}
In this section we prove Theorem \ref{thm:giants_MCMs}. As usual, by our sampling result Corollary \ref{cor:sample_Gn}, it is enough to show this for $\G_n$. The subcritical part is a straightforward consequence of our local limit result Theorem \ref{thm:loc_lim}, and this argument is quite standard, see \cite{van2024random}. Roughly speaking, the local limit always \emph{dominates} the pre-limit in a certain sense, so that an almost sure extinction of the local limit immediately rules out the presence of a component with size comparable with the size of the vertex set in the pre-limit. For completeness, we give a short formal argument next, but let us mention this is essentially the same argument as appears in \cite[Corollary 2.27]{van2024random}.

\begin{proof}[Proof of Theorem \ref{thm:giants_MCMs}(2): The subcritical regime.]
    For a fixed $K \geq 1$, denoting by $\cC(v)$ the connected component of a vertex in $\G_n$, we have by Theorem \ref{thm:loc_lim}
    \begin{align*}
        Z_{>K}:=\frac{1}{n}\sum_{v \in V_n}\ind{\cC(v)>K}\to \Prob{\cC(\varnothing)>K}, \numberthis \label{eq:local_component_func_sub}
    \end{align*}
since an inspection of $B_{G_n}(K)$ completely reveals whether $\cC(v)>K$ or not, and
where $\cC(\varnothing)$ is the connected component of $\varnothing$ in $\cT_\infty$. Noting the bounds
\begin{align*}
    Z_{>K}\geq \frac{|\cC^{(1)}_n|\ind{|\cC^{(1)}_n>K|}}{n}\geq \frac{|\cC^{(1)}_n|}{n}-\frac{K}{n},
\end{align*}
we thus have for any $\delta>0$,
\begin{align*}
    \liminf_{n \to \infty}\Prob{\frac{|\cC^{(1)}_n|}{n}<\Prob{\cC(\varnothing)>K}+\delta/4+\delta/2}\geq 1-\delta,
\end{align*}
since $K/n<\delta/4$ for $n$ large enough and for any fixed $K$. By monotonicity, $\lim_{K \to \infty}\Prob{\cC(\varnothing)>K}=1-\eta=0$, thus, doing all this with a $K$ sufficiently large satisfying $\Prob{\cC(\varnothing)>K}<\delta/4$, we obtain
\begin{align*}
    \liminf_{n \to \infty}\Prob{\frac{|\cC^{(1)}_n|}{n}<\delta}\geq 1-\delta,
\end{align*}
which concludes the proof in the subcritical regime.
\end{proof}
The supercritical case is more involved, and uses a constructive approach. Namely, when $\eta<1$, using Theorem \ref{thm:extinct_1dep}, we have one of the terms defining the maximum in \eqref{eq:ext_BP_condition} is strictly larger than $1$. Using this term, we will do a clever exploration in $\G_n$, and show that in doing so we uncover a giant component. We describe this exploration next.

\subsection{Proof of the supercritical statement}\label{sec:proof_supcrit_giant}
Let $\eta<1$, so that the left-hand side of \eqref{eq:ext_BP_condition} is strictly larger than $1$. Recall the type digraph $G_S$ from Definition \ref{def:type graph}. Consider a $u \in V(G_S)$ such that $\cM(c)>1$. Let $c=(u=u_0,u_1,\dots,u_l=u_0=u)$, where let us write $u_i=(s_i,s_{i+1},s_{i+2})$ for each $0\leq i \leq l-3$ with $s_i \in [k]$ for each $i$, so that $u_{l-2}=(s_{l-2},s_{l-1},s_0),u_{l-1}=(s_{l-1},s_0,s_1)$ and $u_l=u_0=(s_0,s_1,s_2)$. Consider $o_1,o_2$ random vertices sampled without replacement from $\cN^{(s_0)}$. Denote by $$\{o_1 \stackrel{\G_n}{\longleftrightarrow} o_2\}$$ the event that $o_1$ and $o_2$ are connected in the graph $\G_n$. 
\begin{proposition}[Connection of random vertices]\label{prop:non_neg_conn_prob_randomvertices}
    It holds that
    \begin{align*}
        \liminf_{n \to \infty}\Prob{o_1 \stackrel{\G_n}{\longleftrightarrow} o_2}>0. \numberthis \label{eq:non_neg_conn_prob_randomvertices}
    \end{align*}
\end{proposition}
Let us first argue why this finishes the proof of the supercritical regime of Theorem \ref{thm:giants_MCMs}. 
    \begin{proof}[Proof of Theorem \ref{thm:giants_MCMs}(1.) subject to Proposition \eqref{prop:non_neg_conn_prob_randomvertices}]

    Let $\eta_*>0$ be such that $\Prob{o_1 \stackrel{\G_n}{\longleftrightarrow}o_2}>\eta_*$ for all large $n$, possible due to \eqref{eq:non_neg_conn_prob_randomvertices}, and note that
\begin{align*}
    \eta_*<\Prob{o_1\stackrel{\G_n}{\longleftrightarrow} o_2}=\Exp{\CProb{o_2\in \cC(o_1)}{G_n,o_1}}=\Exp{\frac{|\cC(o_1)|}{n-1}}\leq (1+o(1))\Exp{\frac{|\cC^{(1)}_n|}{n}},
\end{align*}
    true for all large $n$, implying \eqref{eq:supcrit_conclusion}.
\end{proof}

Next, let us go into the proof of Proposition \ref{prop:non_neg_conn_prob_randomvertices}.
We only prove it under the additional assumption that:
\begin{align*}
    \textrm{All degrees in $\G_n$ are uniformly bounded from above by a large constant $\bB>0$}.\numberthis\label{assump:unif_bdd_degs}
\end{align*}
There is a standard method to lift this assumption, and we refer the reader to \cite[p 160]{van2024random} where this is discussed for configuration models; the technique works verbatim for our case, and to avoid repetition, we do not mention the same steps.
 
Assumption \ref{assump:unif_bdd_degs} is unnecessary, for example, if we allow the variables $D_{(u,v,w)}$ to have finite $(1+\delta)$ moments, for some $\delta>0$. We comment more on this at the appropriate place.

\subsubsection{Connection of random vertices under uniformly bounded degrees}\label{sec:truncation_argument}
In this section, our goal is to prove the following proposition.
\begin{proposition}[Connection of random vertices under bounded degrees]\label{prop:connection_under_truncation}
    Assume all the conditions of Proposition \ref{prop:non_neg_conn_prob_randomvertices} and the uniform boundedness assumption \eqref{assump:unif_bdd_degs}. With $o_1$ and $o_2$ as before random vertices from $\cN^{(s_0)}$, we have
    \begin{align*}
        \liminf_{n \to \infty}\Prob{o_1\stackrel{\G_n}{\longleftrightarrow}o_2}>0.\numberthis \label{eq:connection_under_truncation}
    \end{align*}
\end{proposition}

Recall that $\sM(c)>1$ for the loop $c=(u=u_0,\dots,u_l=u_0)$ where we write $u_i=(s_{i},s_{i+1},s_{i+2})$, the subscripts of $s$ being $\bmod l$.\footnote{Throughout this section, a subscript $j$ of $s_j$ is always $\bmod l$.}

\begin{lemma}[GOOD pairs in $c$]\label{lem:good_pairs_c}
    For any $i=0,1,\dots,l-1$, the pair $(s_i,s_{i+1})$ is $\rm GOOD$.
\end{lemma}
\begin{proof}
Let us begin with a general observation. We claim that for any $s,s',s'' \in S$, 
\begin{align*}
    (s,s') \textrm{ is {\rm GOOD} and } \Exp{D_{(s,s',s'')}}>0 \implies (s',s'') \textrm{ is {\rm GOOD}.}\numberthis \label{claim:good_pairs}
\end{align*}
To check this, note that by definition
\begin{align}
    \Exp{D_{n,(s,s',s'')}}=\frac{\frac{1}{N^{(s')}}\sum_{x=1}^{N^{(s')}}d_x^{(s',s)}(n)d_x^{(s',s'')}(n)}{\frac{1}{N^{(s')}}\sum_{x=1}^{N^{(s')}}d_x^{(s',s)}(n)}&=O\left(\sqrt{\Exp{(d_{o}^{(s,s')}(n))^2} \Exp{(d_{o}^{(s',s'')}(n))^2}} \right)\\&=O\left(\sqrt{\Exp{(D_{n,(s,s')})^2} \Exp{(D_{n,(s',s'')})^2}} \right),
\end{align}
by Cauchy-Schwarz, and the fact that $(s',s)$ is GOOD (so that $(N^{(s')})^{-1}\sum_{x=1}^{N^{(s')}}d_x^{(s',s)}(n)=\Theta(1)$), and where $o$ above is a uniformly sampled random vertex of $\cN^{(s')}$. Now note that if $(s',s'')$ was bad, then the right-hand side in the above display would be $o(1)$ as $n \to \infty$, which contradicts $\Exp{D_{(s,s',s'')}}>0$.

To use \eqref{claim:good_pairs} effectively, recall that we assume the left-hand side of \eqref{eq:ext_BP_condition} is $>1$. This implies the existence of a pair $(s,s')$ that is GOOD, a $v=(s,s',s'')\in V(G_S)$ and a $u\in \cC_D(v)$ satisfying $\Exp{D_u}>0$, with $u$ being in the loop $c$ satisfying $\sM(c)>1$. 

Observe that for any directed path $(v_0=v\rightarrow v_1 \rightarrow \dots \rightarrow v_r=u)$ in $D(G_S)$, by inductively using \eqref{claim:good_pairs} (starting with $v_0=v$), for any vertex $(t,t',t'')$ on this path, we must have that both $(t,t')$ and $(t',t'')$ are GOOD. In particular, if we write $u=(s_j,s_{j+1},s_{j+2})$ for some $j=0,1,\dots,l-2$, we have both $(s_j,s_{j+1})$ and $(s_{j+1},s_{j+2})$ are GOOD. Finally, for any $u$ in the loop $c$, we have $\Exp{D_u}>0$, since $\sM(c)>1$. Thus, taking the directed path in $c$ that goes from $u$ to the vertex $(s_i,s_{i+1},s_{i+2})$, and arguing inductively again using \eqref{claim:good_pairs}, we have $(s_i,s_{i+1})$ is GOOD (as is $(s_{i+1},s_{i+2})$).
\end{proof}

To prove \eqref{eq:connection_under_truncation}, we essentially do a \emph{restricted exploration} in the graph $\G_n$. First observe that for any subgraph $H\subseteq \G_n$, it is enough to exhibit a path from $o_1$ to $o_2$ using only edges in $H$, to conclude \eqref{eq:connection_under_truncation}. We choose this subgraph cleverly. Namely, recall we write $c=(u=u_0,u_1,\dots,u_l=u_0)$, where $u_i=(s_i,s_{i+1},s_{i+2})$.
\begin{definition}[The subgraph $\G_n(c)$]\label{def:graph_gnc}
    Define $\G_n(c)$ to be the subgraph of $\G_n$ induced by the set of edges $\cup_{i=0}^{l-1} E_{s_i,s_{i+1}}$, where recall $E_{i,j}$ is the subset of edges $\{\{u,v\}\in E(\G_n):u \in \cN^{(i)},v\in \cN^{(j)}\}$.
\end{definition}
In particular, note that since all the pairs $(s_i,s_{i+1})$ are GOOD due to Lemma \ref{lem:good_pairs_c}, $\G_n(c)$ can be sampled by first sampling bipartite configuration models $CM_n^{s_i,s_{i+1}}$ (unipartite if $s_i=s_{i+1}$) for every $i=0,\dots,l-1$ on the vertex bipartitions $\cN^{(s_i)}\cup \cN^{s_{i+1}}$ (with appropriate degree sequences), and taking the union of all these edges. We will show that already in this subgraph, $o_1$ and $o_2$ are in the same connected component, which proves \eqref{eq:connection_under_truncation}. Let us next discuss the relevant component exploration processes for this purpose.

\paragraph{Restricted exploration from $o_1$.} To prove \eqref{eq:connection_under_truncation}, we run special exploration processes from both $o_1$ and $o_2$ in $\G_n(c)$ simultaneously. Let us first describe the exploration from $o_1$. Let $A_c(o_1,0)$ be $\{o_1\}$, where the subscript $c$ is to indicate that our exploration depends on the loop $c$. As we did in the proof of Theorem \ref{thm:loc_lim}, we will start exploring from $o_1$ in a breadth-first manner, but in a \emph{restricted} fashion as we describe in what follows.

$A_c(o_1,0)=\{o_1\}$ is the zeroth step of the exploration. To construct $A_c(o_1,1)$, we reveal a half edge $h\in H_{s_0\rightarrow s_1}$ incident to $o_1$ and pair it with a half-edge sampled uniformly at random from $H_{s_1\rightarrow s_0}$, where if there is no such half-edge we stop exploring at this point. Thus, assuming this step goes through, $A_c(o_1,1)$ consists of $o_1$ and a neighbor of $o_1$ in $\cN^{(s_1)}$. 

We want to now describe the construction of $A_c(o_1,t+1)$, given that we have already constructed $A_c(o_1,t)$, for $t\geq 0$. For this, for any $v \in A_c(o_1,t)$ with distance $r\geq 0$ from $o_1$, let us write \begin{align*}
    r=(l-1)p+r' \textrm{ with } 0\leq r' < l-1,\numberthis \label{eq:r_division}
\end{align*} and define $H_c(v,t)$ to be the set of unpaired half-edges incident to $v$ and in $H_{T(v)\rightarrow s_{r'+1}}$. Thus, for example, $H_{c}(o_1,t)$ is the set of unpaired half-edges incident to $o_1$ and in $H_{s_0 \rightarrow s_1}$ after constructing $A_c(o_1,t)$; if $A_c(o_1,1)$ consists of the vertices $o_1$ and $v$, then $H_c(v,t)$ is the set of unpaired half-edges incident to $v$ after having constructed $A_c(o_1,t)$ and also in $H_{s_1\rightarrow s_2}$, etc. 

Then, for $t\geq 0$, having constructed $A_c(o_1,t)$ and with the information $H_v(t)$ for $v\in A_c(o_1,t)$ at hand, to construct $A_c(o_1,t+1)$, we start by considering a vertex $v \in A(o_1,t)$ with minimal distance from $o_1$ among all such vertices satisfying $|H_v(t)|>0$. Let $v$ have distance $r$ from $o_1$, where we express $r$ as in \eqref{eq:r_division}. Let us then take a half-edge $h \in H_v(t)$, and randomly sample a uniform half-edge $h'$ not already paired in one of the previous steps from $H_{T_v \rightarrow r'+1}$, which is incident to some $u \in \cN^{(s_{r'+1})}$, and pair $h$ and $h'$. We include the vertex $u$ in $\cN^{(s_{r'+1})}$ in $A_c(o_1,t)$ to construct $A_c(o_1,t+1)$. 

Define the stopping time
\begin{align*}
    \tau(o_1):=\inf\{t \geq 0: A_c(o_1,t)\textrm{ is not a tree}\}.
\end{align*}
Let us observe the following:
\begin{itemize}
    \item Since we start the exploration from the vertex $o_1$ with type $s_0$, the set $H(o_1,0)$ consists only of half-edges from $H_{s_0\rightarrow s_1}$. Thus, the vertex we include in $A_c(o_1,0)$ to form $A_c(o_1,1)$ has type $s_1$. In general, by an inductive argument, it can be checked that as long as $t<\tau(o_1)$, i.e., $A_c(o_1,t)$ does not have a cycle, any vertex $v\in A_c(o_1,t)$ at distance $r$ from $o_1$ has type $s_{r'}$, where $r'$ is a function of $r$, expressed as in \eqref{eq:r_division}. 
    \item In particular, note that the different types of vertices that we see at different distances from $o_1$ in the exploration cluster $A_c(o_1,t)$ for $t<\tau(o_1)$ can be seen as a function of the loop $c$. We call this exploration as `exploring from $o_1$ as \emph{guided} by the loop $c'$.
\end{itemize}

Let us define $D_c(o_1,t)$ to be the set of unpaired half-edges incident to some vertex in $A_c(o_1,t)$. Thus $D_c(o_1,t)$ can be expressed as a disjoint union
\begin{align*}
    D_c(v,t)=\cup_{v \in A_c(o_1,t)}H_c(v,t).
\end{align*}

\begin{remark}\label{rem:expl_any_loop_c_stop_expl}
    Note that the (restricted) exploration from $o_1$ can be defined on $\G_n(c)$ for any loop $c$, irrespective of whether $\sM(c)>1$ or not. Further, it could be that at a particular stage of this process, the set of half-edges from which we are to sample, to pair to our current unpaired half-edge, is simply empty. In that case, we stop exploring at this point. If this time instant is say $t$, our convention is $A_c(o_1,s)=A_c(o_1,t)$ for all $s>t$.
\end{remark}

\paragraph{Simultaneous restricted exploration from $o_1$ and $o_2$.} We now describe how to do the restricted exploration from both $o_1$ and $o_2$ simultaneously. Let us begin with an easy but useful observation. Note from Remark \ref{rem:assump} that the expression
\begin{align*}
    \Exp{D_{n,(i,j,l)}}=\frac{\sum_{x=1}^{N^{(j)}}d_x^{(j,i)}d_x^{(j,l)}}{\sum_{x=1}^{N^{(j)}}d_x^{(j,i)}}
\end{align*}
is invariant under interchanging $i$ and $l$, so that $\Exp{D_{(i,j,l)}}=\Exp{D_{(l,j,i)}}.$ Thus, since $\cM(c)>1$, we also have $\cM(\overleftarrow{c})>1$, where $\overleftarrow{c}$ is the loop $v_0,v_1,\dots,v_l=v_0$, where $v_0=(s_0,s_{l-1},s_{l-2}),v_1=(s_{l-1},s_{l-2},s_{l-3}),\dots,v_{l-2}=(s_2,s_1,s_0),$ $v_{l-1}=(s_1,s_0,s_{l-1}),v_{l}=(s_0,s_{l-1},s_{l-2})=v_0$. We call $\overleftarrow{c}$ the \emph{reversal} of $c$. Note that
    \begin{align*}
        v_j=(s_{l-j},s_{l-(j+1)},s_{l-(j+2)})\textrm{ for }1\leq j \leq l-2.
    \end{align*}

In particular, it is possible to instead start exploring in a restricted fashion from $o_2$ in the graph $\G_n(c)$, guided now by the loop $\overleftarrow{c}$, and form the exploration clusters $(A_{\overleftarrow{c}}(o_2,t))_{t \geq 0}$ analogous to the clusters $A_c(o_1,t)$, as described in the last section.

However, we want to do these explorations simultaneously. We do this formally via a discrete time process defined as follows. Fix $\eps>0$. Later we will let $\eps\to 0$, so it is useful to think of it as a small constant. We define a Markov process $(\cA_t,\cB_t)_{t \geq 0}$ as follows. We let $\cA_0:=\{o_1\}$ and $\cB_0:=\{o_2\}$. Define the (possibly infinite) stopping time
\begin{align*}
    \tau'(o_1,\eps):=\inf\{t>\eps \sqrt{n}:\forall\;\;v\in A_c(o_1,t),H_c(v,t)\subset H_{s_0\rightarrow s_1}\}. \numberthis \label{eq:stop_time_from_o_1}
\end{align*}
On the event $\tau'(o_1,\eps)<\infty$, we define $(\cA_t,\cB_t):=(A_c(o_1,t),\cB_0)$ for all $0\leq t \leq \tau'(o_1,\eps)$. Conditionally on $(\cA_t,\cB_t)_{0 \leq t \leq \tau'(o_1,\eps)}$, we do a restricted exploration guided by $\overleftarrow{c}$ from $o_2$ to define the sets $\cB_t$ for $t>\tau'(o_1,\eps)$. Concretely, for each $0\leq i \leq l-1$, let $H'_{s_0\rightarrow s_1}$ be the set of all half-edges in $H_{s_0\rightarrow s_1}$ not already paired in the first $\tau'(o_1,\eps)$ steps of the process. Take a half-edge $h$ from $H'_{s_0\rightarrow s_{l-1}}$ incident to $o_2$, if any. Randomly sample a half-edge $h'$ from $H'_{s_{l-1},s_0}$ and pair it to $h$, and include the vertex $v$ to which $h'$ is incident into $\cB_{\tau'(o_1,\eps)}$ to create $\cB_{\tau'(o_1,\eps)+1}:=\cB_{\tau'(o_1,\eps)}\cup \{v\}=\{o_2,v\}$, while keeping $\cA_{\tau'(o_1,\eps)+1}=\cA_{\tau'(o_1,\eps)}=A_c(o_1,t)$.

Thus, beyond time $\tau'(o_1,\eps)$, we \emph{freeze} the evolution of the set $\cA_t$ at $\cA_{\tau'(o_1,\eps)}$, and only evolve the set $\cB_{t}$. The latter evolution is exactly the restricted exploration from $o_2$ guided by the loop $\overleftarrow{c}$, with the caveat that we are not allowed to sample any half-edge that has already been used to create $\cA_{\tau'(o_1,\eps)}=A_c(o_1,\tau'(o_1,\eps))$. 

More formally, we note that conditionally on having constructed $A_c(\tau'(o_1,\eps))$, the rest of the half-edge pairings in $\G_n(c)$, have the same distribution as the half-edge pairings of a smaller graph $\G_n'({c}
)$, also a union of bipartite configuration models just like $\G_n(c)$, with the same vertex set as of $\G_n(c)$, where for any vertex $v\in \cN^{(s_i)}$ its degree in $\cN^{(s_j)}$ for $j=i-1,i+1$ in the graph $\G_n'({c})$ is 
\begin{align*}
    d^{(s_j)}(v)':=d^{(s_j)}(v)-d^{(s_j)}(v,\tau'(o_1,\eps)),
\end{align*}
where $d^{(s_j)}(v,\tau'(o_1,\eps))$ is the number of half-edges incident to $v$ and in $H_{s_i \rightarrow s_j}$ that has already been paired in constructing $A_c(o_1,\tau'(o_1,\eps))$. Then, if we define $(A'_{\overleftarrow{c}}(o_2,t))_{t \geq 0}$ the restricted exploration process from $o_2$ as guided by the loop $\overleftarrow{c}$ in the graph $\G_n'({c})$, then for any $t>0$, we define 
\begin{align*}
    (\cA_{\tau'(o_1,\eps)+t},\cB_{\tau'(o_1,\eps)+t}):=(A_c(o_1,\tau'(o_1,\eps)),A'_{\overleftarrow{c}}(o_2,t)).
\end{align*}
Let $H'_{\overleftarrow{c}}(v,t)$ be the analogue of $H_{{c}}(v,t)$ in $\G_n'({c})$. Analogous to $\tau'(o_1,\eps)$, define the stopping time 
\begin{align*}
    \tau'(o_2,\eps):=\inf\{t>\eps \sqrt{n}:\forall\;\;v\in A'_{\overleftarrow{c}}(o_2,t),H'_{\overleftarrow{c}}(v,t)\subset H_{s_1\rightarrow s_0}\}. \numberthis \label{eq:stop_time_from_o_2}
\end{align*}
Let us also analogously define $D_{\overleftarrow{c}}(o_2,t)$ to be the set of all unpaired half-edges incident to some vertex in $A'_{\overleftarrow{c}}(o_2,t)$.

\begin{proposition}[Enough unpaired half-edges from the cluster of $o_1$]\label{prop:enough_unpaired_o_1_fresh}
   Recall the graph $\G_n$ from Corollary \ref{cor:sample_Gn}, and the graph $\G_n(c)$ corresponding to any loop $c$ with $\sM(c)>1$ as defined in Definition \ref{def:graph_gnc}. Under all the assumptions of Corollary \ref{cor:sample_Gn}, there exists a function $g(\eps)$ satisfying $g(\eps)>0$ for $\eps>0$ and $\lim_{\eps\to 0}g(\eps)=0$ such that
    \begin{align*}
       \lim_{\eps \to 0} \liminf_{n \to \infty}\Prob{\tau'(o_1,\eps)<\sqrt{\eps n},|D_c(o_1,\tau'(o_1,\eps))|>g(\eps)\sqrt{n}}=1. \numberthis \label{eq:enough_unpaired_o_1_fresh}
    \end{align*}
\end{proposition}
We will prove Proposition \ref{prop:enough_unpaired_o_1_fresh} later. But first, let us observe an immediate corollary of this result, pertaining to the restricted exploration from $o_2$. Recall the smaller graph $\G_n'({c})$ which is defined conditional on the first $\tau'(o_1,\eps)$ steps of the process $(\cA_t,\cB_t)_{t \geq 0}$, and the restricted exploration from $o_2$ in $\G_n'({c})$ guided by the loop $\overleftarrow{c}$. Let us also define the event
\begin{align*}
    \cF(o_1):=\{\tau'(o_1,\eps)<\sqrt{\eps n},|D_c(o_1,\tau'(o_1,\eps))|>g(\eps)\sqrt{n}\},\numberthis \label{eq:event_Fo1}
\end{align*}
where $g(\cdot)$ is the function appearing in Proposition \ref{prop:enough_unpaired_o_1_fresh}.

\begin{corollary}[Enough unpaired half-edges from the cluster of $o_2$]\label{cor:enough_unpaired_o_2_fresh}
    With the same function $g(\eps)$ as appearing in Proposition \ref{prop:enough_unpaired_o_1_fresh}, we have
    \begin{align*}
        \lim_{\eps \to 0}\liminf_{n \to \infty}\CProb{\tau'(o_2,\eps)<\sqrt{\eps n},|D_{\overleftarrow{c}}(o_2,\tau'(o_2,\eps))|>g(\eps)\sqrt{n}}{\cF(o_1)}=1.
    \end{align*}
\end{corollary}
    \begin{proof}
        Consider any sequence of deterministic sets $({\rm A}_t)_{t \geq 0}$ that the random sequence $(\cA_t)_{t \geq 0}$ can possibly be equal to with positive probability up to time $\tau'(o_1,\eps)$, i.e., $$\Prob{\cA_t={\rm A}_t\;\;\forall\;\;0\leq t\leq \tau'(o_1,\eps)}>0,$$ and such that the number of unpaired half-edges incident to ${\rm A}_{\tau'(o_1,\eps)}$ is at most $\sqrt{\eps n}$. By the law of total probability, it is enough to show that for any such sequence $({\rm A}_t)_{0 \leq t \leq \tau'(o_1,\eps)}$, we have 
\begin{align*}
    \lim_{\eps \to 0}\liminf_{n \to \infty}\CProb{\tau'(o_2,\eps)<\sqrt{\eps n},|D(o_2,\tau'(o_2,\eps))|>h(\eps)\sqrt{n}}{\cA_t={\rm A}_t\;\;\forall\;\;0\leq t\leq \tau'(o_1,\eps)}=1. \\\numberthis \label{eq:enough_unpaired_o2_fresh_toshow}
\end{align*}

As remarked before, conditionally on the event $\{\cA_t={\rm A}_t\;\;\forall\;\;0\leq t\leq \tau'(o_1,\eps)\}$, $\G_n'({c})$ is a union of configuration models on its own, and thus we can apply Proposition \ref{prop:enough_unpaired_o_1_fresh} directly to conclude \eqref{eq:enough_unpaired_o2_fresh_toshow}, if we can check that $\G_n'({c})$ satisfies the assumptions of it (some of which are the assumptions of Corollary \ref{cor:sample_Gn}, in fact), no matter what the sequence $({\rm A}_t)_{0 \leq t \leq \tau'(o_1,\eps)}$ is. 

Note that we need only check the assumptions pertaining to the parts $\cN^{(s_i)}$ for $0\leq i \leq l-1$. Assumption \ref{assump:part_i_size} for these parts is straightforward. For Assumptions \ref{assump:degree_regularity}, \ref{assump:offspring_degree_reg} and \ref{assump:one_good_deg_greater_2} we note that since in the first $\tau'(o_1,\eps)$ steps, we have exhausted at most $2\tau'(o_1,\eps)<2\sqrt{\eps n}=o(n)$ many half-edges, the conditions \eqref{assump:regularity_1_2} and \eqref{eq:assump_regularity_tensor} are satisfied identically, taking $i=s_p$ and $j=s_{p+1}$ in those conditions, for any $p=0,1,\dots,l-1$, with the \emph{same} constants appearing on the right-hand side of those conditions. We conclude by Proposition \ref{prop:enough_unpaired_o_1_fresh}, taking $\G_n(c)=\G_n'({c})$ and $c=\overleftarrow{c}$, and recalling $\sM(\overleftarrow{c})>1$, since $\sM(c)>1$.
    \end{proof}

\paragraph{Proof of connection of random vertices.} Let us now give the proof of Proposition \ref{prop:connection_under_truncation}. In the next section, we will develop the necessary machinery to prove Proposition \ref{prop:enough_unpaired_o_1_fresh} and provide its proof.
\begin{proof}[Proof of Proposition \ref{prop:connection_under_truncation}]
    Let us first write for any $\eps>0$,
    \begin{align*}
        \Prob{o_1\stackrel{\G_n}{\longleftrightarrow}o_2}\geq \CProb{o_1\stackrel{\G_n}{\longleftrightarrow}o_2}{\cF(o_1)\cap \cF(o_2)}\Prob{\cF(o_1)\cap \cF(o_2)}, 
    \end{align*}
    where recall the event $\cF(o_1)$ from \eqref{eq:event_Fo1}, and analogously we define the event 
    \begin{align*}
        \cF(o_2):=\{\tau'(o_2,\eps)<\sqrt{\eps n},|D_{\overleftarrow{c}}(o_2,\tau'(o_2,\eps))|>h(\eps)\sqrt{n}\}.\numberthis \label{eq:event_Fo2}
    \end{align*}
    Thanks to Proposition \ref{prop:enough_unpaired_o_1_fresh} and Corollary \ref{cor:enough_unpaired_o_2_fresh}, we have $\Prob{\cF(o_1)\cap \cF(o_2)}>0$ after letting $n\to \infty$ first followed by $\eps\to 0$.

    Thus it is enough to show that for any fixed $\eps>0$ no matter how small, we have
    \begin{align*}
        \liminf_{n \to \infty} \CProb{o_1\stackrel{\G_n}{\longleftrightarrow}o_2}{\cF(o_1)\cap \cF(o_2)}>0.
    \end{align*}
Consider the sets $(\cA_t,\cB_t)$ at time $t=\tau'(o_1,\eps)+\tau'(o_2,\eps)$. Under the conditioning event above, we have at least $g(\eps)\sqrt{n}$ many unpaired half-edges from $H_{s_0\rightarrow s_1}$ incident to some vertices in $\cA_t$ and at least $g(\eps)\sqrt{n}$ many half-edges from $H_{s_1\rightarrow s_0}$ incident to some vertices in $\cB_t$. It could have been that in some step $s$ of constructing the sets $\cB_{\tau'(o_1,\eps)},\cB_{\tau'(o_1,\eps)+1},\dots,\cB_{\tau'(o_1,\eps)+\tau'(o_2,\eps)}$, we pair a half-edge from $\cB_{\tau'(o_1,\eps)+s}=A'_{\overleftarrow{c}}(o_2,s)$ to one unpaired half-edge incident to $A_c(o_1,\tau'(o_1,\eps))=\cA_{\tau'(o_1,\eps)+s}$, in which case we have already discovered a path from $o_1$ to $o_2$ in the graph $\G_n$ (in fact, in the subgraph $\G_n(c)$). 

But in any case, if this is not true, consider continuing the process $(\cA_t,\cB_t)_{t \geq 0}$ beyond time $t=\tau'(o_1,\eps)+\tau'(o_2,\eps)$, where for the subsequent $ g(\eps)\sqrt{n}$ steps, we only pair the first $ g(\eps)\sqrt{n}$ many unpaired half-edges (in some arbitrary order) from $H_{s_1\rightarrow s_0}$ incident to $\cB_{\tau'(o_1,\eps)+\tau'(o_2,\eps)}=A'_{\overleftarrow{c}}(o_2,\tau'(o_2,\eps))$. The probability that none of these half-edges are paired to any of the unpaired half-edges from $H_{s_0\rightarrow s_1}$ incident to some vertices in $\cA_{\tau'(o_1,\eps)+\tau'(o_2,\eps)}=A_c(o_1,\tau'(o_1,\eps))$ is at most
\begin{align*}
    \left(1-\frac{g(\eps)\sqrt{n}}{|H_{s_0\rightarrow s_1}|}\right)^{g(\eps)\sqrt{n}}.
\end{align*}
This is because at every such subsequent step, with probability at least $g(\eps)\sqrt{n}/|H_{s_0\rightarrow s_1}|$ we sample a (not yet paired) half-edge from $D(o_1,\tau'(o_1,\eps))$ and the samplings at different steps are independent. Since the pair $(s_0,s_1)$ is good, we note that as $n \to \infty$, the last display tends to the constant 
\begin{align*}
    \exp{\left(-\frac{g(\eps)^2}{\Exp{D_{(s_0, s_1)}}} \right)}<1.
\end{align*}
It follows that 
\begin{align*}
    \liminf_{n \to \infty}\CProb{o_1\stackrel{\G_n}{\longleftrightarrow}o_2}{\cF(o_1)\cap \cF(o_2)}\geq 1-\exp{\left(-\frac{g(\eps)^2}{\Exp{D_{(s_0, s_1)}}} \right)}>0.
\end{align*}
\end{proof}
    
\paragraph{Proof of Proposition \ref{prop:enough_unpaired_o_1_fresh}.} It remains to prove Proposition \ref{prop:enough_unpaired_o_1_fresh}. To approach the proof, let us begin by defining a useful branching process, which will help us understand the exploration set $A_c(o_1,t)$. 

\begin{definition}\label{def:BP_cT_n}
  The $n$-dependent branching process $\cT_n$ is defined as follows.
  \begin{itemize}
      \item The root $\varphi$ of this process has type $s_0$, and has a random number of offspring of type $s_1$ given by the random variable $D_{n, (s_0, s_1)}$. We also say $\{\varphi\}$ is the zeroth generation $Z_0$ of $\cT_n$.
      \item  For any $r\geq 1$ an integer, let us write
      \begin{align*}
          r_l:=r\bmod l.
      \end{align*}
      Inductively, the set of all vertices at generation $k$ for $k \geq 1$ is denoted by $Z_k$, and has type $k_l$, and each of them has a random number of offspring $D_{n, s_{(k-1)_l},s_{k_l},s_{(k+1)_l}}$. Often we lighten notation, and simply write $s_k$ for $s_{k_l}$ (this should not create confusion; recall the subscripts of $s$ are always $\bmod l$).
  \end{itemize}
\end{definition}
We will couple the cluster $A_c(o_1,t)$ constructed by doing the restricted breadth-first exploration from $o_1\in \G_n(c)$ with the tree $\cT_n$ naturally, described step-by-step in the following points.
\begin{itemize}
    \item Recall $o_1$ is uniformly distributed in $\cN^{(s_0)}$. Thus, the number of neighbors it has in $\cN^{(s_1)}$ is precisely distributed as $D_{n,(s_0, s_1)}$. Thus, we can identify $o_1$ with $\varphi$, and let the number of offspring of $\varphi$ in $\cT_n$ be precisely the number of neighbors of $o_1$ in $\cN^{(s_1)}$. 
    \item Recall that to construct $A_c(o_1,1)$, we take a half-edge $h\in H_{s_0 \rightarrow s_1}$ incident to $o_1$. We then randomly sample a half-edge $h' \in H_{s_1 \rightarrow s_0}$ and pair it to $h$, and if $h'$ is incident to $v \in \cN^{s_1}$, we let $A_c(o_1,1)$ the edge $\{o_1,v\}$. Observe that since $v$ has exactly in distribution $D_{n,(s_0,s_1,s_2)}$ many offspring in $\cN^{(s_2)}$, we can identify $v$ with a vertex $\varphi(v)$ in $\cT_n$ discovered in a first step $\cT_n(1)$ of a breadth-first exploration of $\cT_n$ starting from $\varphi$, with the property that the number of unpaired half-edges incident to $v$ in $A_c(o_1,1)$ is precisely the number of offspring of $\varphi(v)$ in $\cT_n$ that we have \emph{discovered, but not yet explored.}
    
    \item In a general step, recall that to construct $A_c(o_1,t+1)$ from $A_c(o_1,t)$, we take a half-edge $h$ incident to some vertex $v \in A_c(o_1,t)$ (at minimal distance from $o_1$ with at least one unpaired half-edge incident to it) and uniformly sample a compatible half-edge $h'$ (among the ones that have not already been paired so far) and pair to $h$. Note that we can instead simply sample $h'$ uniformly from the set of \emph{all} half-edges compatible to $h$ and pair $h$ and $h'$, provided $h'$ does not equal any of the half-edges sampled so far. In general, since we aim to couple $A_c(o_1,t)$ with the breadth-first exploration $\cT_n(t)$ of the tree $\cT_n$, this method of sampling uniformly from the set of \emph{all} half-edges compatible to $h$ can break down, if any of the following two occurs:
    \begin{itemize}
        \item The half-edge $h'$ has already been sampled in a previous step, in which case we are trying to re-use a half-edge already paired, which is non-physical.
        \item The half-edge $h'$ is incident to a vertex $v \in A_c(o_1,t)$. In this case, we actually discover a cycle in $A_c(o_1,t)$ and thus we cannot hope for $A_c(o_1,t)$ to be coupled to $\cT_n(t)$.
    \end{itemize}
If any of the above two items occur, we say there is a \emph{restricted coupling conflict} at step $t$.

    \item Inductively, for $t\geq 0$, assume that $A_c(o_1,t)$ is perfectly coupled with $\cT_n(t)$ without seeing any restricted coupling conflict, as described above, with each vertex $v \in A_c(o_1,t)$ having a counterpart $\varphi(v)\in \cT_n(t)$, with the property that the number of unpaired half-edges incident to $v$ is the same as the number of discovered but yet unexplored offspring of $\varphi(v)$ in $\cT_n$. Then we can do this coupling procedure for one further step, i.e., construct $A_c(t+1)$ similarly, provided there is no {restricted coupling conflict} at step $t+1$.
\end{itemize}

Thus, defining the stopping time 
\begin{align*}
    \tau:=\inf\{t \geq 0:\textrm{ there is a restricted coupling conflict at step } t+1 \}, \numberthis \label{eq:stop_time_rest_coupconf}
\end{align*}
  we can couple $A_c(o_1,t)$ and $\cT_n(t)$ for any $t<\tau$ perfectly as described above. Let us gather all these observations in a lemma whose proof essentially follows from the discussion we just had and so we omit it.

\begin{lemma}[Perfect coupling of restricted exploration process]\label{lem:rest_coupling}
      The restricted exploration process $A_c(o_1,j)$ can be perfectly coupled with a breadth-first exploration process $\cT_n(j)$ for all $0\leq j \leq t$ where $t<\tau$, such that every $v \in A_c(o_1,j)$ has a counterpart $\varphi(v)$ in $\cT_n(j)$, with the property that the number of neighbors of $v$ in $A_c(o_1,j)$ equals the number of explored offspring of $\varphi(v)$ in $\cT_n$, while the number of unpaired half-edges incident to $v$ equals the number of discovered but yet unexplored offspring of $\varphi(v)$ in $\cT_n$. 
  \end{lemma}

What we really want to say is that $\tau$ is sufficiently large so that this coupling can be done up to $\eps\sqrt{n}$ steps. The next Proposition proves this.

\begin{proposition}[Restricted exploration coupling up to good number of steps]\label{prop:coupling_valid_epssqrtn}
   Assume all the conditions of Proposition \ref{prop:connection_under_truncation}. Consider the restricted exploration process $A_c(o_1,t)$ and its coupling with the breadth-first exploration process $\cT_n(t)$ in $\cT_n$ as in Lemma \ref{lem:rest_coupling}, and the stopping time $\tau$ as in \eqref{eq:stop_time_rest_coupconf}. Then 
    \begin{align*}
        \liminf_{n \to \infty}\Prob{\tau\leq \eps \sqrt{n}}=O(\eps^2),
    \end{align*}
    where the $O(\cdot)$ term on the right-hand side above is as $\eps\to 0$.
\end{proposition}

\begin{proof}
    Let us denote by $x_n:=\min_{0\leq i\leq l-1}|H_{s_i \rightarrow s_{i+1}}|$, and note that $x_n=\Theta(n)$ since the pairs $(s_i,s_{i+1})$ are GOOD. Note that 
    \begin{align*}
        &\Prob{\tau \leq \eps \sqrt{n}}\\&=\Prob{\textrm{one of the steps $1,2,\dots,\eps\sqrt{n}$ witnesses a restricted coupling conflict}}\\&\leq \sum_{i=1}^{\eps \sqrt{n}}\Prob{\textrm{step $i$ witnesses a restricted coupling conflict}}.
    \end{align*}
    Bounding the individual summands in the above sum is quite similar to \eqref{eq:expected_conflicts}. Concretely, the probability at step $i$ we sample a half edge $h'$ already sampled in a previous step is at most $$\frac{i-1}{x_n},$$ and at step $i$ we sample a half-edge $h'$ incident to a vertex $v\in A_c(o_1,i-1)$ (thus creating a cycle) is at most $$\frac{(i-1)\max_{0\leq j \leq l-1}\Delta_n(s_j,s_{j+1})}{x_n},$$ where recall $\Delta_n(p,q)$ is the maximum degree a vertex in $\cN^{(p)}$ can have in $\cN^{(q)}$, for any $p,q\in [k].$ Thus, recalling $x_n=\Theta(n)$ (so that $x_n-\eps \sqrt{n}>x_n/2$ for all large $n$),
    \begin{align*}
        \sum_{i=1}^{\eps \sqrt{n}}\Prob{\textrm{step $i$ witnesses a restricted coupling conflict}}\leq \frac{\eps^2 n}{2x_n}(1+\max_{0\leq j \leq l-1}\Delta_n(s_j,s_{j+1})).\\\numberthis \label{eq:expected_restricted_conf}
    \end{align*}By assumption \eqref{assump:unif_bdd_degs}, the right-hand side above is at most $\frac{\eps^2n(1+\bB)}{2x_n}=O(\eps^2)$ since $x_n=\Theta(n)$.
\end{proof}

As a consequence of the last proposition, we note that the coupling between $A_c(o_1,t)$ and $\cT_n(t)$ can be done up to $\eps \sqrt{n}$ steps with probability at least $1-o_{\eps,n}(1)$, having all the properties as in Lemma \ref{lem:rest_coupling}, where $o_{\eps,n}(1)$ denotes an error term that goes to zero when we first let $n \to \infty$ followed by $\eps\to 0$. With the help of this result, as we shall see, understanding the number $|D_c(o_1,\tau'(o_1,\eps))|$ of half-edges incident to vertices in $A_c(o_1,\tau'(o_1,\eps))$ for small $\eps>0$ boils down to understanding particular generation sizes in the tree $\cT_n$.  

Note that if the coupling of Lemma \ref{lem:rest_coupling} is valid up to the stopping time $\tau'(o_1,\eps)$, then by the properties of this coupling, $\tau'$ agrees with a particular stopping time $\tau''(\eps)$, defined in terms of the breadth-first exploration process $\cT_n(t)$ of the tree $\cT_n$. Namely, $\tau''(\eps)$ is precisely the first time \emph{after} $\eps \sqrt{n}$ steps of the breadth-first exploration, that we finish completely exploring a generation $Z_m$ in $\cT_n$, where $m$ has the form $m=pl$ for an integer $p\geq 1$ (i.e., $m$ is a multiple of $l$, $m=0\bmod l$), and in particular, all the discovered but yet unexplored vertices are those that have some vertex of $Z_m$ as their parent, thus themselves have type $s_1$. 

Concretely, let $\cT_n(D,t)$ and $\cT_n(E,t)$ respectively denote the set of discovered but not yet explored, and the set of all explored vertices in $\cT_n(t)$. Then
\begin{align*}
    \tau''(\eps):=\inf\{t>\eps \sqrt{n}:\textrm{ all vertices in }\cT_n(D,t)\textrm{ have type }s_1\}. \numberthis \label{eq:def_tau''}
\end{align*}

Our strategy to show that the coupling is valid up to $\tau'(o_1,\eps)$ is a roundabout way. Namely, we will first use properties of the branching process $\cT_n$ to check that the stopping time $\tau''(\eps)$ is only larger than $\eps \sqrt{n}$ by at most a universal constant factor $C$, and thus, since the coupling is valid up to time $C\eps \sqrt{n}$ (with probability $1-o_{\eps,n}(1)$ thanks to Proposition \ref{prop:coupling_valid_epssqrtn}), it must be valid up to the stopping time $\tau'(o_1,\eps)$.

To make this idea work, it is useful at this point to define another branching process $\cY_n$.
\begin{definition}[The process $\cY_n$]\label{def:Yn}
    The process $\cY_n$ is a subprocess of $\cT_n$, and satisfies the following. 
\begin{itemize}
    \item The root of $\cY_n$ is $\varphi$, the root of $\cT_n$. The set of children of $\varphi$ in $\cY_n$ is the set of all vertices in generation $l$, i.e., $Z_{l}$ of $\cT_n$. Thus, letting $Z'_1$ denote the first generation of $\cY_n$, we have $Z'_1=Z_l$. 
    \item In general, inductively, for any $p$, the $p$-th generation $Z'_p$ of $\cY_n$ is the set $Z_{pl}$, with the parent in $\cY_n$ of each vertex $v \in Z'_p$, is the ancestor $u$ of $v$ in $Z_{(p-1)l}$ in $\cT_n$. Note that $u \in Z'_{p-1}$.
\end{itemize}
\end{definition}

Observe that given the size $|Z_m|$ of the $m$-th generation in $\cT_n$, the distribution of the size $|Z_{m+1}|$ of the next generation is
\begin{align*}
    |Z_{m+1}|=\sum_{i=1}^{|Z_m|}D^{(i)}_{n,(s_{m-1},s_m,s_{m+1})}, \numberthis \label{eq:generation_size_cTn}
\end{align*}
where $\{D^{(i)}_{n,(s_{m-1},s_m,s_{m+1})}:i\geq 1\}$ constitutes an i.i.d.\ collection of variables with the same law as $D_{n,(s_{m-1},s_m,s_{m+1})}$ (always keeping in mind the subscripts of $s$ are $\bmod l$). As a consequence of this observation we note that:
\begin{itemize}
    \item The number of children of the root $\varphi$ in $\cY_n$ is distributed as the variable $D_{n,l}$, where we define $D_{n,l}$ inductively as follows. $D_{n,1}$ has the same distribution as $D_{n,(s_0, s_1)}$, and inductively, for any $2\leq j \leq l$, having defined $D_{n,j-1}$, 
    \begin{align*}
        D_{n,j}\stackrel{d}{=}\sum_{i=1}^{D_{n,j-1}} D^{(i)}_{n,(s_{j-2},s_{j-1},s_j)},
    \end{align*}
    the summands on the RHS above being independent of the upper limit of the sum.
    \item For any vertex $v\neq \varphi$ in $\cY_n$, the number of offspring of $v$ in $D$ has the distribution $d_{n,l}$, where $d_{n,l}$ is defined inductively as follows. $d_{n,1}\stackrel{d}{=}D_{n,(s_{l-1},s_{0},s_1)}$, and inductively for $2\leq j \leq l$, 
    \begin{align*}
        d_{n,j}\stackrel{d}{=}\sum_{i=1}^{d_{n,j-1}}D^{(i)}_{n,(s_{j-2},s_{j-1},s_j)},\numberthis \label{eq:offspring_law_Yn}
    \end{align*}
    the summands on the RHS above being independent of the upper limit of the sum.
\end{itemize}
It is useful at this point to define a natural limiting version $\cY_\infty$ of $\cY_n$, as follows. \begin{definition}[The process $\cY_\infty$]\label{def:Y-infty}
    \begin{itemize}
    \item The number of children of the root $\varphi$ in $\cY_\infty$ is distributed as the variable $D_{l}$, where we define $D_{l}$ inductively as follows. $D_{1}$ has the same distribution as $D_{(s_0,s_1)}$, and inductively, for any $2\leq j \leq l$, having defined $D_{j-1}$, 
    \begin{align*}
        D_{j}\stackrel{d}{=}\sum_{i=1}^{D_{j-1}} D^{(i)}_{(s_{j-2},s_{j-1},s_j)},
    \end{align*}
    the summands on the RHS above being independent of the upper limit of the sum.
    \item For any vertex $v\neq \varphi$ in $\cY_\infty$, the number of offspring of $v$ in $D$ has the distribution $d_{l}$, where $d_{l}$ is defined inductively as follows. $d_{1}\stackrel{d}{=}D_{(s_{l-1},s_{0},s_1)}$, and inductively for $2\leq j \leq l$, 
    \begin{align*}
        d_{j}\stackrel{d}{=}\sum_{i=1}^{d_{j-1}}D^{(i)}_{(s_{j-2},s_{j-1},s_j)},\numberthis \label{eq:offspring_law_Yn}
    \end{align*}
    the summands on the RHS above being independent of the upper limit of the sum.
    \end{itemize}
\end{definition}

The key result on $\cY_n$ and $\cY_\infty$ is the next lemma. \begin{lemma}\label{lem:on_Yn_Y-infty}
\begin{itemize}
    \item [(a.)] For all $n$ sufficiently large but fixed, the process $\cY_n$ survives forever with positive probability. 
    \item [(b.)] The process $\cY_\infty$ also survives forever with positive probability.
    \item [(c.)] Moreover, denoting by $\chi_n$ the survival probability of $\cY_n$, we have
    \begin{align*}
        \lim_{n \to \infty} \chi_n=\chi,    \end{align*} where $\chi$ is the survival probability of $\cY_\infty$.
\end{itemize}
     
\end{lemma}
\begin{proof}
For part (a.), note that by the definition of $\cY_n$, for any first-generation vertex $v \in Z'_1$, the subtree of it in $\cY_n$ is a BGW tree, with mean number of offspring
    \begin{align*}
        \sM_n(c):=\Exp{D_{n,(s_{l-1},s_0,s_1)}}\Exp{D_{n,(s_{l-1},s_0,s_1)}}\cdots \Exp{D_{n,(s_{l-2},s_{l-1},s_0)}}.
    \end{align*}
    Since $\sM(c)>1$, it follows that $\sM_n(c)>1$ for all large $n$, which follows from the fact that $\sM_n(c)\to \sM(c)$ as $n \to \infty$, thanks to \eqref{assump:regularity_1_2}. Finally, the first generation itself has expected size
    \begin{align*}
        \Exp{D_{n,s_0\rightarrow s_1}}\Exp{D_{n,(s_0,s_1,s_2)}}\cdots \Exp{D_{n,(s_{l-2},s_{l-1},s_0)}},
    \end{align*}
    converging as $n \to \infty$ to $\Exp{D_{s_0 \rightarrow s_1}}\Exp{D_{(s_0,s_1,s_2)}}\cdots \Exp{D_{(s_{l-2},s_{l-1},s_0)}}>0$, where the last inequality is true since $(s_0,s_1)$ is a $\rm GOOD$ pair, and since $$0<\Exp{D_{(s_0,s_1,s_2)}}\cdots \Exp{D_{(s_{l-2},s_{l-1},s_0)}}=\frac{\sM(c)}{\Exp{D_{(s_{l-1},s_0,s_1)}}}.$$  

    (b.) follows in exactly the same way as (a.), noting that the mean offspring number of the BGW subtree corresponding to any first generation vertex of $\cY_\infty$ is $\sM(c)>1$.

    (c.) follows from the fact that using \eqref{assump:regularity_1_2} and \eqref{eq:assump_regularity_tensor}, the offspring law of the root vertex of $\cY_n$ converges in distribution to the root offspring law of $\cY_\infty$, and the offspring law of any non-root vertex of $\cY_n$ converges in distribution to the offspring law of any non-root vertex of $\cY_\infty$.
\end{proof}

\begin{remark}\label{rem:Zn_survives}
 The process $\cY_n$ being a subprocess of the process $\cT_n$, the latter also survives forever with positive probability for all large $n$. Particularly, $\liminf_{n \to \infty}\chi'_n>0$ where we let $\chi'_n$ be the survival probability of $\cT_n$. Similarly, although not necessary for our purposes, let us point out that the natural limiting version $\cT_{(\infty)}$ of $\cT_n$ also survives with positive probability, since it contains $\cY_\infty$ as a subprocess.     
\end{remark}

Recall the stopping time $\tau''(\eps)$ from \ref{eq:def_tau''}, and recall that our goal is to show that the stopping time $\tau''(\eps)$ is at most a universal constant multiple of $\eps \sqrt{n}$. First note that in the breadth-first exploration process $\cT_n(t)$ of $\cT_n$, at every step, we explore exactly one vertex, and discover its offspring, to be explored later. Thus, letting $$\Xi_t=|Z_0|+|Z_1|+\dots+|Z_{tl}|$$ be the sum of the sizes of the first $tl$ generations of $\cT_n$, by the definition of $\tau''(\eps)$, we note that $\tau''(\eps)=\Xi_{t(\eps)+1}$, where $t(\eps)$ satisfies
\begin{align*}
    \Xi_{t(\eps)}\leq \eps \sqrt{n} < \Xi_{t(\eps)+1}.\numberthis \label{eq:ineq_t(eps)_epssqrtn}
\end{align*}

Our aim is to get good control over the sequence $(\Xi_t)_{t \geq 1}$ to show that $\Xi_{t+1}$ is at most a constant factor larger than $\Xi_t$, no matter what $t$ is, which, combined with the previous display, shows that $\tau''(\eps)$ is at most a constant factor of $\eps \sqrt{n}$.

We have to state a technical real-analytic result first. Consider positive reals $\alpha,\beta,m,m',\gamma>0$ with $m,m'>1$ and $\gamma\in (1/2,1)$. 

Set $\underline{a}_0=\Bar{a}_0=\alpha$ and $\underline{b}_0=\Bar{b}_0=\beta$, and given $\underline{a}_i,\Bar{a}_i,\underline{b}_i,\Bar{b}_i$ for $i \geq 0$, inductively define
\begin{align*}
    &\underline{a}_{i+1}:=m\underline{a}_i-(\Bar{a}_i)^{\gamma}, \Bar{a}_{i+1}:=m\Bar{a}_i+(\Bar{a}_i)^\gamma,\numberthis \label{eq:inductive_def_seq_a}\\& \underline{b}_{i+1}:=\underline{b}_i+m'\underline{a}_i-(\Bar{b}_i)^\gamma,\Bar{b}_{i+1}:=\Bar{b}_i+m'\Bar{a}_i+(\Bar{b}_i)^\gamma. \numberthis \label{eq:inductive_def_seq_b}
\end{align*}

\begin{lemma}[Sequence growth]\label{lem:seq_grow}
    Consider the sequences $(\underline{a}_i)_{i \geq 0}, (\Bar{a}_i)_{i \geq 0},$ as defined in \eqref{eq:inductive_def_seq_a} and $ (\underline{b}_i)_{i \geq 0}, (\Bar{b}_i)_{i \geq 0}$ as defined in \eqref{eq:inductive_def_seq_b}. Then for all $i \geq 1$, there are universal constants $A,B>1$ such that
    \begin{align*}
        &\frac{\alpha m^i}{A}\leq \underline{a}_i<\Bar{a}_i\leq \alpha Am^i,\textrm{ and  }\beta+\frac{m'\alpha m^i}{B}\leq \underline{b}_i<\Bar{b}_i\leq\beta+m'\alpha Bm^i.\numberthis \label{eq:seq_a,b_behave}
    \end{align*}
\end{lemma}
For the sake of completeness, we provide a proof of this result in Appendix \ref{app_sec:lem_seq_grow}; it is essentially an extension of \cite[Lemma 4.11]{van2024random}. To see the usefulness of the previous result, let us first for convenience introduce the shorthand notation
\begin{align*}
    \nu_{n,i}:=\Exp{D_{n,(s_{i-1},s_i,s_{i+1})}},\textrm{ for }0\leq i \leq l-1.\numberthis \label{eq:def_nu_i}
\end{align*}
Let us then define
\begin{align*}
    m_n:=\nu_{n,0}\nu_{n,1}\dots\nu_{n,l-1}, \textrm{ and }m'_n:=\nu_{n,0}+\nu_{n,0}\nu_{n,1}+\dots+\nu_{n,0}\nu_{n,1}\dots \nu_{n,l-1}.\numberthis \label{eq:def_mn_mn'}
\end{align*}
Observe the following facts from \eqref{eq:generation_size_cTn}:
\begin{itemize}
    \item $m_n$ equals $\sM_n(c)$, the mean number of offspring of any vertex in $\cY_n$ other than $\varphi$. In particular, recalling that we denote the $p$-th  generation of $\cY_n$ by $Z'_p$, note that for any $p\geq 2$
    \begin{align*}
        \CExp{|Z'_p|}{|Z'_{p-1}|}=|Z'_{p-1}|m_n.\numberthis \label{eq:cond_mean_id_xi}
    \end{align*}
    \item Arguing similarly, it can be checked that for a vertex $v\neq \varphi$ of type $s_0$ in the tree $\cT_n$, for any $0\leq i\leq l-1$, $\nu_{n,0}\dots\nu_{n,i}$ is the expected number of descendants of $v$ in $\cT_n$ at distance exactly $i+1$ from $v$. As a consequence, for any $t\geq 2$ we have the identity
    \begin{align*}
        \CExp{\Xi_{lt}}{Z'_{l(t-1)}}=|Z'_{l(t-1)}|m_n'.\numberthis \label{eq:cond_mean_id_Xi}
    \end{align*}
\end{itemize}
At this point let us for convenience of writing define for any $t\geq 0$
\begin{align*}
    \xi_{t}:=|Z_{t}|,
\end{align*}
so that $\xi_{lt}=|Z'_t|$, the size of the $t$-th generation of $\cY_n$. Let us also define a useful random variable $\kappa_l$ inductively. Here $\kappa_1\stackrel{d}{=}D_{n,(s_{l-1},s_0,s_1)}$, and inductively, for $2\leq j \leq l$, given that we have already defined $\kappa_{j-1}$, define
\begin{align*}
    \kappa_j:=\sum_{i=1}^{\kappa_{j-1}}D^{(i)}_{n,(s_{j-2},s_{j-1},s_j)},\numberthis \label{eq:de_kappa_j}
\end{align*}
the summands above being all independent of the upper limit of the sum. To interpret this random variable, observe that for any $v\in \cT_n$ such that $v\neq \varphi$ and $v$ has type $s_0$, $\kappa_j$ equals in distribution the number of descendants of $v$ in $\cT_n$ at distance $j$ away from $v$, $0\leq j \leq l$. In particular, $\Exp{\kappa_j}=\nu_{n,0}\dots\nu_{n,j-1}$ where recall $\nu_{n,i}$ from \eqref{eq:def_nu_i}. Let us further define the random variable
\begin{align*}
    \kappa:=\kappa_1+\kappa_2+\dots+\kappa_{l}. \numberthis \label{eq:def_kappa}
\end{align*}
Thus, $\kappa$ has the interpretation that it is the total number of descendants in $\cT_n$ of a vertex $v\neq \varphi$ with type $s_0$, at distance at most $l$ away from $v$. Again, observe that $\Exp{\kappa}=m'_n$, recalling $m'_n$ from \eqref{eq:def_mn_mn'}.

We then have the following useful result that bounds the generation sizes in $\cT_n$. This result is inspired by \cite[Lemma 4.12]{van2024random}.

\begin{proposition}[Bounding generation sizes]\label{prop:bound_gen_size_ind}
     Fix some $R\geq 1$. For any $\alpha,\beta>0$, take $\underline{a}_0=\Bar{a}_0=\alpha$ and $\underline{b}_0=\Bar{b}_0=\beta$, and define $\underline{a}_i,\Bar{a}_i,\underline{b}_i,\Bar{b}_i$ inductively for $i\geq 1$ as in \eqref{eq:inductive_def_seq_a} and \eqref{eq:inductive_def_seq_b}, with $m=m_n$ and $m'=m_n'$ as defined in \eqref{eq:def_mn_mn'}. Then
    \begin{align*}
        \liminf_{n \to \infty}\CProb{\underline{a}_t\leq \xi_{l(R+t)}\leq \Bar{a}_t,\underline{b}_t\leq \Xi_{R+t}\leq \Bar{b}_t\textrm{ for all }t\geq 1}{\xi_{lR}=\alpha,\Xi_R=\beta}\geq 1-O\left(\alpha^{1-2\gamma}\right),\\ \numberthis \label{eq:bounding_gen_sizes}
    \end{align*}
    where the constant hidden in the $O(\cdot)$ term above is universal, i.e., independent of $R, \alpha,\beta$.
\end{proposition}

\begin{proof}
For any $t\geq 1$ define the event
\begin{align*}
    \cE_t:=\{\underline{a}_t\leq \xi_{l(R+t)}\leq \Bar{a}_t,\underline{b}_t\leq \Xi_{R+t}\leq \Bar{b}_t,\xi_{lR}=\alpha,\Xi_{R}=\beta\}.\numberthis \label{eq:def_cEt}
\end{align*}
We aim to show $\Prob{\cap_{t \geq 1}\cE_t}\geq 1-O(\alpha^{1-2\gamma})$, for which, using the bound
\begin{align*}
    \Prob{\left(\cap_{t \geq 1}\cE_t \right)^c}\leq \sum_{t\geq 1}\CProb{\cE_{t}^c}{\cE_{t-1}},\numberthis \label{eq:union_bd_complement_seq_event}
\end{align*}
it is enough to show that
\begin{align*}
    \sum_{t\geq 1}\CProb{\cE_{t}^c}{\cE_{t-1}}=O(\alpha^{1-2\gamma}).\numberthis \label{eq:toshow_ind_gen_size}
\end{align*}
By a union bound, the summand above is at most
\begin{align*}
    \CProb{\cE_t^c}{\cE_{t-1}}\leq \CProb{\{\xi_{lt}<\underline{a}_t\}\cup\{\xi_{lt}>\Bar{a}_t\}}{\cE_{t-1}}+\CProb{\{\Xi_{t}<\underline{b}_t\}\cup\{\Xi_{t}>\Bar{b}_t\}}{\cE_{t-1}}.
\end{align*}

Note that conditionally on $\cE_{t-1}$, on the event $\{\xi_{lt}<\underline{a}_t\}\cup\{\xi_{lt}>\Bar{a}_t\}$ we have
\begin{align*}
    \xi_{lt}-m_n\xi_{l(t-1)}>\Bar{a}_t-m\Bar{a}_{t-1}=(\Bar{a}_{t-1})^\gamma \textrm{ and } \xi_{lt}-m_n\xi_{l(t-1)}<\underline{a}_t-m\underline{a}_{t-1}=(\Bar{a}_{t-1})^\gamma
\end{align*}
so that in particular
\begin{align*}
    |\xi_{lt}-m_n\xi_{l(t-1)}|=|\xi_{lt}-\CExp{\xi_{lt}}{\xi_{l(t-1)}}|>(\Bar{a}_{t-1})^\gamma,\numberthis \label{eq:cheby_error_xi}
\end{align*}
using \eqref{eq:cond_mean_id_xi} for the first equality above. Thus, by a conditional Chebyshev's inequality,
\begin{align*}
    \CProb{\{\xi_{lt}<\underline{a}_t\}\cup\{\xi_{lt}>\Bar{a}_t\}}{\cE_{t-1}}\leq \frac{\Exp{\indE{\cE_{t-1}}\CExp{(\xi_{lt}-\CExp{\xi_{lt}}{\xi_{l(t-1)}})^2}{\xi_{l(t-1)}}}}{(\Bar{a}_{t-1})^{2\gamma}}.\numberthis \label{eq:cond_cheby_xi}
\end{align*}
Recall that given $\xi_{l(t-1)}$, the distribution of $\xi_{lt}$ is a sum of $\xi_{l(t-1)}$ many i.i.d.\ random variables with distribution identical to $d_{n,l}$ corresponding to $j=l$ as in \eqref{eq:offspring_law_Yn}. Using this with the fact that the variance of an independent sum is a sum of the individual summand variances, and the assumption that the degrees are uniformly bounded from above by $\bB$, we deduce the RHS above is at most
\begin{align*}
    \CProb{\{\xi_{lt}<\underline{a}_t\}\cup\{\xi_{lt}>\Bar{a}_t\}}{\cE_{t-1}}\leq \frac{\Exp{\indE{\cE_{t-1}}\xi_{l(t-1)}F(\bB)}}{(\Bar{a}_{t-1})^{2\gamma}}\leq F(\bB)(\Bar{a}_{t-1})^{1-2\gamma},\numberthis \label{eq:cond_cheby_xi_actual_error}
\end{align*}
where $F(\bB)$ is a constant depending on $\bB$ and independent of $t$.

One can deduce a similar upper bound on the conditional probability of the event $\{\Xi_{t}<\underline{b}_t\}\cup\{\Xi_{t}>\Bar{b}_t\}$ with a few minor subtleties. Let us just point these out avoiding repetition. One ends up with an event as in \eqref{eq:cheby_error_xi}, but with $(\Bar{b}_{t-1})^\gamma$ on the RHS, using which and conditional Chebyshev one deduces a similar upper bound as in \eqref{eq:cond_cheby_xi}, but with the conditional expectation on the numerator replaced by $\CExp{(\Xi_{t}-\CExp{\Xi_{t}}{\xi_{l(t-1)}})^2}{\xi_{l(t-1)}}$ and the denominator replaced by $(\Bar{b}_{t-1})^{2\gamma}$. Now, given $\xi_{l(t-1)}$ note that $\Xi_t$ is an i.i.d.\ sum of $\xi_{l(t-1)}$ many random variables with the same law as the variable $\kappa$, recalling its definition and interpretation from \eqref{eq:def_kappa}. Thus, under the assumption that all degrees are at most $\bB$, $\kappa$ is a bounded random variable, which means $\CExp{(\Xi_{t}-\CExp{\Xi_{t}}{\xi_{l(t-1)}})^2}{\xi_{l(t-1)}}\leq H(\bB)\xi_{l(t-1)}$ for a constant $H(\bB)$ depending only on $B$ and not on $t$. We therefore obtain an upper bound similar to \eqref{eq:cond_cheby_xi_actual_error},
\begin{align*}
    \CProb{\{\Xi_{t}<\underline{b}_t\}\cup\{\Xi_{t}>\Bar{b}_t\}}{\cE_{t-1}} \leq H(\bB)\frac{\Bar{a}_{t-1}}{(\Bar{b}_{t-1})^{2\gamma}}.\numberthis \label{eq:cond_cheby_Xi}
\end{align*}
Combining Lemma \ref{lem:seq_grow} with \eqref{eq:cond_cheby_xi_actual_error} and \eqref{eq:cond_cheby_Xi}, we have
\begin{align*}
    \CProb{\cE_t^c}{\cE_{t-1}}&\leq \CProb{\{\xi_{lt}<\underline{a}_t\}\cup\{\xi_{lt}>\Bar{a}_t\}}{\cE_{t-1}}+\CProb{\{\Xi_{t}<\underline{b}_t\}\cup\{\Xi_{t}>\Bar{b}_t\}}{\cE_{t-1}} \\&\leq F(\bB)A'\alpha^{1-2\gamma}(m^{1-2\gamma})^{t-1}+H(\bB)B'(m')^{-2\gamma}\alpha^{1-2\gamma}(m^{1-2\gamma})^{t-1},
\end{align*}
where $A',B'>0$ are universal constants, independent of $\alpha$ and $\beta$. Thus, recalling the union bound \eqref{eq:union_bd_complement_seq_event}, we obtain
\begin{align*}
    \Prob{\left(\cap_{t \geq 1}\cE_t \right)^c}\leq \alpha^{1-2\gamma}\sum_{t\geq 1}\left(F(\bB)A'(m^{1-2\gamma})^{t-1}+H(\bB)B'(m')^{-2\gamma}(m^{1-2\gamma})^{t-1}\right),\numberthis \label{eq:final_UB_complement_seqevent}
\end{align*}
Recall in the above display we take $m=m_n,m'=m'_n$, where $m_n,m'_n$ are as in \eqref{eq:def_mn_mn'}, and note that thanks to assumptions \eqref{assump:regularity_1_2}, \eqref{eq:assump_regularity_tensor}, 
\begin{align*}
    \lim_{n \to \infty}m_n,\lim_{n \to \infty}m'_n\geq \sM(c)>1.
\end{align*}
Thus, we note that the sum on the RHS of \eqref{eq:final_UB_complement_seqevent} is firstly convergent, and secondly, is at most a constant $L=L(\bB,\sM(c))>0$, for all large $n$. The result follows.
\end{proof}

\begin{remark}[Limit of $m'_n$]\label{rem:limit_m'_n}
   Recalling the sequence $(m'_n)_{n \geq 1}$ from \eqref{eq:def_mn_mn'}, for later use we record the limit of $m'_n$ as 
    \begin{align*}
        \lim_{n \to \infty}m'_n=\nu_0+\nu_0\nu_1+\dots+\nu_0\nu_1\dots \nu_{l-1}=:\nu',
    \end{align*}
    where recalling $\nu_{n,i}$ from \eqref{eq:def_nu_i}, we define $\nu_i$ as its limit,
    \begin{align*}
        \lim_{n\to \infty}\nu_{n,i}=\Exp{D_{(s_{i-1},s_i,s_{i+1})}}=:\nu_i.
    \end{align*}
    In particular, note that $\nu'>\nu_0\dots\nu_{l-1}=\sM(c)>1$.
\end{remark}

\begin{remark}[Efficacy of the result]
Note that the error bound in Proposition \ref{prop:bound_gen_size_ind} is useful only when $\alpha$ is sufficiently large. \end{remark}

\begin{remark}[On the uniform boundedness of the degrees]
    Let us also take this moment to comment on the assumption of uniform boundedness from above by $\bB$ of the degrees, as we use this in the last proof, in particular at \eqref{eq:cond_cheby_xi_actual_error} and \eqref{eq:cond_cheby_Xi} to say that the conditional variance arising in those bounds are respectively at most $F(\bB)$ and $H(\bB)$ times $\xi_{l(t-1)}$. 
    
    This assumption is not required if the variables $D_{(u,v,w)}$ admit finite $(1+\delta)$ moments for some $\delta>0$ as remarked at the beginning of this section. To see this, assume the latter, and run the proof up to \eqref{eq:cheby_error_xi}, and then instead of applying a conditional Chebyshev to obtain \eqref{eq:cond_cheby_xi}, rather apply a conditional Markov with centered $(1+\delta)$ moments. As a result the exponent $1-2\gamma$ appearing on the RHS of \eqref{eq:cond_cheby_xi_actual_error} and \eqref{eq:cond_cheby_Xi} changes to simply $-\delta$, and the order of the probability on the LHS of \eqref{eq:final_UB_complement_seqevent} becomes at most $\alpha^{-\delta}$, which is good for our purposes. As will be evident shortly in the proof of Proposition \ref{prop:enough_unpaired_o_1_fresh}, we will simply need that this error can be made as small as possible by making $\alpha$ large, equivalently, the RHS of \eqref{eq:bounding_gen_sizes} as close to $1$ as possible, and this is possible whenever $\delta>0$.
\end{remark}

We can finally conclude this section by providing the proof of Proposition \ref{prop:enough_unpaired_o_1_fresh}.
\begin{proof}[Proof of Proposition \ref{prop:enough_unpaired_o_1_fresh}]
Recall the stopping time $\tau$ from \eqref{eq:stop_time_rest_coupconf} till which we do not see any restricted coupling conflict, and in particular the coupling of $A_c(o_1,t)$ with $\cT_n(o_1,t)$ is valid with all the properties as in Lemma \ref{lem:rest_coupling} for any $t<\tau$. We begin by bounding the probability of interest as
\begin{align*}
    &\Prob{\tau'(o_1,\eps)<\sqrt{\eps n},|D_c(o_1,\tau'(o_1,\eps))|>g(\eps)\sqrt{n}}\\&\geq \Prob{\tau'(o_1,\eps)<\sqrt{\eps n},|D_c(o_1,\tau'(o_1,\eps))|>g(\eps)\sqrt{n},\tau>\sqrt{\eps n}}\\& = \Prob{\tau''(\eps)<\sqrt{\eps n},\xi_{l(t(\eps)+1)}>g(\eps)\sqrt{n},\tau>\sqrt{\eps n}}\\&\geq \Prob{\tau''(\eps)<\sqrt{\eps n},\xi_{l(t(\eps)+1)}>g(\eps)\sqrt{n}}-\Prob{\tau\leq \sqrt{\eps n}}\\& =\Prob{\Xi_{t(\eps)+1}<\sqrt{\eps n},\xi_{l(t(\eps)+1)}>g(\eps)\sqrt{n}}-o_{\eps,n}(1),\numberthis \label{eq:decomp_first}
\end{align*}
thanks to Lemma \ref{prop:coupling_valid_epssqrtn}, where recall $t(\eps)$ from \eqref{eq:ineq_t(eps)_epssqrtn}, and note that the equality above in the third line follows from the definition of the coupling. So far, using our coupling, we have managed to reduce the matter completely to a probability concerning the branching process $\cT_n$, and the result will follow if we can show that $\Prob{\Xi_{t(\eps)+1}<\sqrt{\eps n},\xi_{l(t(\eps)+1)}>g(\eps)\sqrt{n}}$ is strictly positive for all large $n$ and for all small $\eps>0$.

For some fixed $R$, consider bounding as
\begin{align*}
    &\Prob{\Xi_{t(\eps)+1}<\sqrt{\eps n},\xi_{l(t(\eps)+1)}>g(\eps)\sqrt{n}}\geq \Prob{\Xi_{t(\eps)+1}<\sqrt{\eps n},\xi_{l(t(\eps)+1)}>g(\eps)\sqrt{n},t(\eps)>R}.
\end{align*}
Recall the event $\cE_t$ from \eqref{eq:def_cEt} where the sequences $\underline{a}_t,\Bar{a}_t,\underline{b}_t,\Bar{b}_t$ are defined as in the statement of Proposition \ref{prop:bound_gen_size_ind} for some $\alpha,\beta>0$. We further bound from below the RHS of the last display above by
\begin{align*}
    \sum_{\alpha,\beta>R}\CProb{\Xi_{t(\eps)+1}<\sqrt{\eps n},\xi_{l(t(\eps)+1)}>g(\eps)\sqrt{n},t(\eps)>R,\cap_{t \geq 1}\cE_t}{\cE_0(\alpha,\beta)}\Prob{\cE_0(\alpha,\beta)},\numberthis \label{eq:lotp_from gen R}
\end{align*}
where we define
\begin{align*}
    \cE_0(\alpha,\beta):=\{\xi_{lR}=\alpha,\Xi_R=\beta\}.
\end{align*}
Note that on the event $\cE_0(\alpha,\beta)\cap(\cap_{t\geq 1}\cE_t)\cap \{t(\eps)>R\}$, we have 
\begin{align*}
    \beta+\frac{m'\alpha m^{(f-R)}}{B}\leq \Xi_{f}\leq \beta+m'\alpha B m^{(f-R)},\textrm{ and }\frac{\alpha m^{(f-R)}}{A}\leq  \xi_{lf}\leq \alpha A m^{(f-R)},\numberthis \label{eq:recursive_bounds_incG}
\end{align*}
for any $f\in \{t(\eps),t(\eps)+1\}$, thanks to Lemma \ref{lem:seq_grow}.
Thus, defining 
\begin{align*}
    \cG:=\cE_0(\alpha,\beta)\cap(\cap_{t\geq 1}\cE_t)\cap \{t(\eps)>R\}\cap \{\Xi_{t(\eps)+1}<\sqrt{\eps n}\}\cap\{\xi_{l(t(\eps)+1)>g(\eps)\sqrt{n}}\},\numberthis \label{eq:big_good_event_cG}
\end{align*}
and recalling from \eqref{eq:cond_decomp_t(eps)_small} that $\Xi_{t(\eps)}<\eps\sqrt{n}$, on the event $\cG$, using \eqref{eq:recursive_bounds_incG} we have
\begin{align*}
  \beta+\frac{m'\alpha m^{(t(\eps)-R)}}{B}\leq \Xi_{t(\eps)}\leq \eps\sqrt{n}\implies  t(\eps)\leq \frac{\log\left(\frac{B(\eps \sqrt{n}-\beta)}{m'\alpha} \right)}{\log m}+R.\numberthis \label{eq:t_eps_UB}
\end{align*}
Since on $\cG$ we also have $\Xi_{t(\eps)+1}\leq \beta+m'\alpha B m^{(t(\eps)+1-R)}$ from \eqref{eq:recursive_bounds_incG}, this gives using the last upper bound on $t(\eps)$,
\begin{align*}
    \Xi_{t(\eps)+1}\leq \beta+B^2m(\eps\sqrt{n}-\beta).
\end{align*}
Recalling $m=m_n$ tends to $\sM(c)\in (1,\infty)$ as $n \to \infty$, and $B>1$, we note that the RHS above for all large $n$ is at most $B^2 (\sM(c)+\eps)\eps \sqrt{n}$, which is at most $\sqrt{\eps n}$ by choosing $\eps>0$ sufficiently small. Thus, for all large $n$ and small $\eps>0$
\begin{align*}
    \cE_0(\alpha,\beta)\cap(\cap_{t\geq 1}\cE_t)\cap \{t(\eps)>R\}\cap \{\Xi_{t(\eps)+1}<\sqrt{\eps n}\}=\cE_0(\alpha,\beta)\cap(\cap_{t\geq 1}\cE_t)\cap \{t(\eps)>R\}.
\end{align*}
We conclude that for all large $n$ and $\eps>0$ sufficiently small \eqref{eq:lotp_from gen R} equals
\begin{align*}
    \sum_{\alpha,\beta>R}\CProb{\xi_{l(t(\eps)+1)}>g(\eps)\sqrt{n},t(\eps)>R,\cap_{t \geq 1}\cE_t}{\cE_0(\alpha,\beta)}\Prob{\cE_0(\alpha,\beta)}.\numberthis \label{eq:lotp_from gen R_1}
\end{align*}
Next, again using \eqref{eq:ineq_t(eps)_epssqrtn} and \eqref{eq:recursive_bounds_incG}, we have on $\cG$
\begin{align*}
    \eps \sqrt{n}\leq \Xi_{t(\eps)+1}\leq \beta+m'\alpha B m^{(t(\eps)+1-R)}\implies t(\eps)\geq \frac{\log \left(\frac{\eps \sqrt{n}-\beta}{m'\alpha B} \right)}{\log m}+R-1,\numberthis \label{eq:t(eps)_LB}
\end{align*}
which implies using \eqref{eq:recursive_bounds_incG}
\begin{align*}
    \xi_{l(t(\eps)+1)}\geq \frac{\alpha}{A}\cdot \frac{\eps \sqrt{n}-\beta}{m'\alpha B}\geq g(\eps) \sqrt{n}
\end{align*}
for all large $n$, where we take 
\begin{align*}
    g(\eps)=\frac{\eps}{2AB(\nu'+\eps)}>0,\numberthis \label{eq:choice_g}
\end{align*}
where recall the limit $\nu'$ of the sequence $m'_n$ from Remark \ref{rem:limit_m'_n}. Working with this choice of $g$, \eqref{eq:lotp_from gen R_1} thus equals
\begin{align*}
    \sum_{\alpha,\beta>R}\CProb{t(\eps)>R,\cap_{t \geq 1}\cE_t}{\cE_0(\alpha,\beta)}\Prob{\cE_0(\alpha,\beta)}.\numberthis \label{eq:lotp_from gen R_2}
\end{align*}
We bound from below the last display as
\begin{align*}
     &\sum_{\alpha,\beta>R}\CProb{\cap_{t \geq 1}\cE_t}{\cE_0(\alpha,\beta)}\Prob{\cE_0(\alpha,\beta)}-\Prob{t(\eps)\leq R,\xi_{lR}>R,\Xi_R>R}\\&=\sum_{\alpha,\beta>R}\CProb{\cap_{t \geq 1}\cE_t}{\cE_0(\alpha,\beta)}\Prob{\cE_0(\alpha,\beta)}-\Prob{t(\eps)\leq R,\xi_{lR}>R}\numberthis \label{eq:cond_decomp_t(eps)_small}
\end{align*}
since trivially $\Xi_R>R$. Now, $\Prob{t(\eps)\leq R,\xi_{lR}>R}\leq \Prob{t(\eps)\leq R}$ and recalling \eqref{eq:ineq_t(eps)_epssqrtn}, we note that this event implies the event that by generation $R$ the process $\cT_n$ gives birth to $\eps\sqrt{n}$ many vertices. However, since $R$ is fixed, using the relation \eqref{eq:generation_size_cTn}, it can be checked using Wald's identity and conditions \eqref{assump:regularity_1_2}, \eqref{eq:assump_regularity_tensor} that the expected total size up to generation $R$ of the tree $\cT_n$ stays bounded in $n$ as $n \to \infty$. In particular, by a Markov inequality, we have
\begin{align*}
    \Prob{t(\eps)\leq R}\leq \frac{G(R)}{\eps\sqrt{n}},\numberthis \label{eq:t(eps)_diverge}
\end{align*}
where $G(R)$ is a function of $R$ (possibly blowing up as $R\to \infty$) but independent of $n$. Overall, using Proposition \ref{prop:bound_gen_size_ind}, and noting $\sum_{\alpha,\beta>R}\Prob{\cE_0(\alpha,\beta)}=\Prob{\xi_{lR},\Xi_{R}>R}=\Prob{\xi_{lR}>R}$ we obtain a lower bound on \eqref{eq:cond_decomp_t(eps)_small} as 
\begin{align*}
    (1-O(R^{1-2\gamma}))\Prob{\xi_{lR}>R}-\frac{G(R)}{\eps \sqrt{n}},\numberthis \label{eq:lb_after_Markov_t(eps)}
\end{align*}
where the constant hidden by the $O(\cdot)$ term above is independent of $R,\alpha,\beta$. We conclude that
\begin{align*}
   &\liminf_{n \to \infty}\left(\sum_{\alpha,\beta>R}\CProb{\cap_{t \geq 1}\cE_t}{\cE_0(\alpha,\beta)}\Prob{\cE_0(\alpha,\beta)}-\Prob{t(\eps)\leq R,\xi_{lR}>R}\right)\\& \geq (1-O(R^{1-2\gamma}))\liminf_{n \to \infty}\Prob{\xi_{lR}>R}.\numberthis \label{eq:liminf_n_LB_1}
\end{align*}
Now, recall that $\xi_{lR}$ is precisely the size of the $R$-th generation of the BGW tree $\cY_n$. It follows from Lemma \ref{lem:on_Yn_Y-infty} that 
\begin{align*}
    \liminf_{n \to \infty}\Prob{\xi_{lR}>R}=\Prob{\xi'_{R}>R},
\end{align*}
where for any $j\geq 0$ we denote by $\xi'_j$ the size of generation $j$ of the limiting process $\cY_\infty$. The last probability is at least $\Prob{\xi''_{R-1}>R}$, where we denote by $\xi''_j$ the size of the $j$-th generation of the BGW subtree corresponding to any arbitrary first generation vertex in $\cY_\infty$, which has mean offspring number $\sM(c)>1$. It follows from standard results on BGW trees (see e.g. \cite[Chapter 1]{athreya2012branching}) that $\xi''_j/\sM(c)^j$ almost surely approaches a non-negative random variable $W$, so that $\Prob{\xi''_{R-1}>R}\geq 1-o_R(1)$, where $o_R(1)$ denotes an error that goes to $0$ as $R\to \infty$. Overall, from \eqref{eq:liminf_n_LB_1}, \eqref{eq:cond_decomp_t(eps)_small} and \eqref{eq:lotp_from gen R}, with the choice of $g$ as specified in \eqref{eq:choice_g}, we note that as $n \to \infty$, the RHS of \eqref{eq:decomp_first} is
\begin{align*}
    \liminf_{n \to \infty} \Prob{\Xi_{t(\eps)+1}<\sqrt{\eps n},\xi_{l(t(\eps)+1)}>g(\eps)\sqrt{n}}-o_{\eps}(1)\geq (1-O(R^{1-2\gamma}))(1-o_R(1))-o_\eps(1),
\end{align*}
for any $R$ large but fixed such that the first term above is strictly positive, and where $o_R(1)$ and $o_\eps(1)$ denotes error terms that to go $0$ when one lets respectively $R\to \infty$ followed by $\eps\to 0$. Letting finally $\eps\to 0$ followed by $R\to \infty$, we obtain that
\begin{align*}
    \lim_{\eps \to 0}\liminf_{n \to \infty} \Prob{\tau'(o_1,\eps)<\sqrt{\eps n},|D(o_1,\tau'(o_1,\eps))|>g(\eps)\sqrt{n}}=1,
\end{align*}
concluding the proof. \end{proof}

\section{Typical distances: proof of Theorem \ref{thm:dist}}\label{sec:dist}

In this section we provide the proof of Theorem \ref{thm:dist}. As usual, recalling $\G_n$ from Corollary \ref{cor:sample_Gn}, it is enough to prove the theorem with $G_n$ replaced by $\G_n$.
\subsection{Subcritical regime}
Let us begin with the proof of the first part, in the subcritical ($\eta=1$) regime, which is straightforward.
\begin{proof}[Proof of Theorem \ref{thm:dist}(1.)]
    Since $\eta=1$, by Theorem \ref{thm:giants_MCMs}, $\frac{|\cC_n^{(1)}|}{n}\plim 0$. Now,
    \begin{align*}
        \Prob{d_{\G_n}(o_1,o_2)<\infty}&=\Exp{\CProb{d_{\G_n}(o_1,o_2)<\infty}{\G_n,o_2}}\\&=\Exp{\CProb{o_1\in \cC(o_2)}{\G_n,o_2}}\\&=\Exp{\frac{|\cC(o_2)|}{n-1}}\leq \Exp{\frac{|\cC_n^{(1)}|}{n-1}},
    \end{align*}
    and the final term above tends to $0$ by dominated convergence and the fact that the term inside the expectation converges to $0$ in probability.
\end{proof}
\subsection{Supercritical regime}
We next focus on the supercritical ($\eta<1$) regime. We begin with the upper bound on the typical distance. This is a by-product of the proof of Theorem \ref{thm:giants_MCMs}.
\begin{proof}[Proof of Theorem \ref{thm:dist}(2.)]
    Our goal is to show that there exists $j\in [k]$ and a constant $K^*>0$ with
    \begin{align*}
        \liminf_{n \to \infty}\CProb{d_{\G_n}(o_1,o_2)<K^*\log n}{o_1,o_2\in \cN^{(j)}}=1.\numberthis \label{eq:dist_supcrit_UB_TS}
    \end{align*}
    Since $\eta<1$, recalling Theorem \ref{thm:extinct_1dep}, there is a loop $c$ with $\sM(c)>1$. We continue using the same notations as in Section \ref{sec:proof_supcrit_giant}: we write $c=(u=u_0,u_1,\dots,u_{l-1},u_l=u_0)$, with $u_i=(s_i,s_{i+1},s_{i+2})$, the subscripts of $s$ always being $\bmod l$.

    We claim that we can find $K_2$ such that \eqref{eq:dist_supcrit_UB_TS} is true with $j=s_0$. Let us recall the simultaneous restricted exploration processes from both $o_1$ and $o_2$ as described after Remark \ref{rem:expl_any_loop_c_stop_expl}: we have a Markov process $(\cA_t,\cB_t)_{t \geq 0}$, with $(\cA_t,\cB_t)=(A_c(o_1,t),\{o_2\})$ for all $0\leq t \leq \tau'(o_1,\eps)$ and $(\cA_{\tau'(o_1,\eps)+s},\cB_{\tau'(o_1,\eps)+s})=(A_c(o_1,\tau'(o_1,\eps)),A'_{\overleftarrow{c}}(o_2,s))$ for any $s>0$, where recall the stopping times $\tau'(o_1,\eps)$ from \eqref{eq:stop_time_from_o_1} and $\tau'(o_2,\eps)$ from \eqref{eq:stop_time_from_o_2}, and where $A_c(o_1,t)$ and $A_{\overleftarrow{c}}(o_2,t)$ are respectively the restricted exploration clusters from $o_1$ and $o_2$, as guided by respectively the loops $c$ and its reversal $\overleftarrow{c}$. Recall also the perfect coupling events $\cF(o_1)$ from \eqref{eq:event_Fo1} and $\cF(o_2)$ from \eqref{eq:event_Fo2}.

Analogous to the proof of Proposition \ref{prop:connection_under_truncation}, it is enough to show \eqref{eq:dist_supcrit_UB_TS} with the probability replaced by a conditional probability, given $\cF(o_1)\cap \cF(o_2)$. Recall the branching process tree $\cT_n$ from Definition \ref{def:BP_cT_n} and the stopping time $\tau''(\eps)$ from \eqref{eq:def_tau''}, which is the analogue of the stopping time $\tau'(o_1,\eps)$ on the tree $\cT_n$. 

One can define analogously another branching process tree $\cT_n'$ as in Definition \ref{def:BP_cT_n}, but corresponding to the loop $\overleftarrow{c}$ and the exploration cluster $A_{\overleftarrow{c}}(o_2,t)$, and a stopping time $\tau^*(\eps)$ (similar to $\tau''(\eps)$) that corresponds to the stopping time $\tau'(o_2,\eps)$ on the tree $\cT'_n$. These details are direct analogues, and are left to the reader.

Observe that by the proof of Proposition \ref{prop:connection_under_truncation}, with conditional probability (given $\cF(o_1)\cap \cF(o_2)$) bounded away from zero, we discover a path from $o_1$ to $o_2$, with length at most
\begin{align*}
    H(o_1)+H(o_2)+1,\numberthis \label{eq:hopcount_UB}
\end{align*}
    where $H(o_1)$ is the \emph{height}\footnote{The \emph{height} of any rooted tree is the maximal graph distance of any vertex from the root in the tree.} of the tree $\cT_n(\tau''(\eps))$ and $H(o_2)$ is the height of the tree $\cT'_n(\tau^*(\eps))$. 
    
    But recall from \eqref{eq:ineq_t(eps)_epssqrtn} (and the discussion just above it) that at time $t=\tau''(\eps)$ we precisely finish exploring $(t(\eps)+1)l$ generations of $\cT_n$, in the breadth-first exploration process $\cT_n(t)$ of $\cT_n$, so that, by definition, $H(o_1)=(t(\eps)+1)l$. This is at most $(K_2 \log n)/2$ for some constant $K_2$, by \eqref{eq:t_eps_UB}. An exactly similar argument, but now working with $\cT'_n$ and $\tau^*(\eps)$, establishes $H(o_2)$ is also at most $(K_2 \log n)/2$ (possibly increasing the value of the constant $K_2$, if needed). This concludes that \eqref{eq:hopcount_UB} is at most $K_2 \log n$ (again, by slightly increasing the value of the constant $K_2$).
\end{proof}

 Let us finally focus on the lower bound in part (3.) of Theorem \ref{thm:dist}. The argument for this is standard, and together with a monotonicity observation, is very similar to the typical distance lower bound argument of the proof of \cite[Theorem 3]{van2023local}. It is a bit long, but quite straightforward. We choose to be a little informal here, and provide a sketch only, leaving the details to the reader. 

 \begin{proof}[Proof sketch of Theorem \ref{thm:dist}(3.)]
     Recall we assume all pairs $(i,j)$ are GOOD, so that by Corollary \ref{cor:sample_Gn}, $G_n$ can be sampled by sampling unipartite configuration models  $CM_n^{i,i}$ (conditioned on simplicity) within part $\cN^{(i)}$ and bipartite configuration models $CM_n^{i,j}$ (conditioned on simplicity) across parts $\cN^{(i)}$ and $\cN^{(j)}$ for $i\neq j$, for all $i,j\in [k]$, and taking the union of all these edges. The resulting graph is $\G_n$, and it is enough to establish the lower bound statement with $\G_n$ replacing $G_n$.

     Working with configuration models enables us to harness a particularly useful monotonicity property, which we use to reduce the problem further. Namely, for any (unipartite or bipartite) configuration model, it is straightforward to check that if we have two degree sequences $\bd=(\bd_v:v\in V)$ and $\bd'=(\bd'_v:v\in V)$ on the same vertex set $V$, such that $\bd'$ dominates $\bd$ pointwise (i.e., $\bd_v\leq \bd'_v$ for each $v$), then there is a coupling such that $CM(\bd')$ contains $CM(\bd)$ as a subgraph. For example, for each vertex, one can first pair $\bd_v$ many half-edges to construct $CM(\bd)$; now there are some unpaired half-edges left (exactly $\bd'_v-\bd_v$ many incident to $v$), pairing which we obtain $CM(\bd')$.

How to use this? Well, note that the distance between two vertices can only decrease by the inclusion of more edges, so that it is enough to prove the lower bound for a dominating graph $\G^*_n$ on the same vertex set that contains $\G_n$ as a subgraph. We construct this dominating graph as follows. We give every vertex some extra half-edges, such that the following holds:
\begin{itemize}
    \item For every $i,j\in [k]$, \eqref{assump:regularity_1_2} holds with $D_{(i,j)}$ replaced on the right-hand side by a common random variable $D^*$ that stochastically dominates $D_{(i,j)}$ for all $i,j\in [k]$, and in particular satisfying
    \begin{align*}
        \Exp{D^*}\geq \max_{i,j\in [k]}\Exp{D_{(i,j)}}.\numberthis\label{eq:D^*_assump}
    \end{align*}
    \item For every $i,j,l\in [k]$ with $i\neq l$, \eqref{eq:assump_regularity_tensor} holds with a common random variable $D^{\dagger}$ replacing $D_{(i,j,l)}$ on the right-hand side, that stochastically dominates $D_{(i,j,l)}$ for all $i,j,l\in [k]$ with $i\neq l$, and in particular satisfying
    \begin{align*}
        \Exp{D^\dagger}\geq \max_{i,j,l\in [k]:i\neq l}\Exp{D_{(i,j,l)}}.\numberthis\label{eq:D^dagger_assump}
    \end{align*}
\end{itemize}
In particular, by the coupling discussed in the previous paragraph, the dominating graph $\G_n^*$ contains $\G_n$ as a subgraph, and it is enough to show that with positive probability $d_{\G_n^*}(o_1,o_2)$ is at least $K_* \log n$. 

We rely on the local limit of $\G_n^*$ for this. Applying Theorem \ref{thm:loc_lim}, it can be checked that the local limit $\cT^*_\infty$ is a simple $2$-type branching process tree, described as follows. For any vertex $v \in \cT^*_\infty$, we let $D^{(1)}_v$ and $D^{(2)}_v$ be respectively the number of type-$1$ and type-$2$ offspring of $v$. We also let $\varnothing$ denote the root of $\cT^*_\infty$.
\begin{itemize}
    \item For each vertex $v$ in $\cT^*_\infty$, $D^{(1)}_v$ and $D^{(2)}_v$ are independent, and for different vertices $u\neq v$ in $\cT^*_\infty$, $(D^{(1)}_u,D^{(2)}_u)$ is independent of $(D^{(1)}_v,D^{(2)}_v)$.
    \item We have $D^{(1)}_\varnothing\stackrel{d}{=}\sum_{i=1}^k D^*_i$, where $\{D^*_i:1\leq i\leq k\}$ form an i.i.d.\ collection of copies of $D^*$. Further, $D^{(2)}_\varnothing=0$.
    \item Recalling from \eqref{eq:shifted_sb} the shifted size-biased copy $\underline{X}$ corresponding to a variable $X$, for any $v\in \cT^*_\infty$ such that $v \neq \varnothing$, we have $D^{(1)}_v\stackrel{d}{=}\underline{D^*}$ and $D^{(2)}_v\stackrel{d}{=}\sum_{i=1}^{k-1}D^\dagger_i$, where $\{D^\dagger_i:1\leq i\leq k-1\}$ form an i.i.d.\ collection of copies of $D^\dagger$.
\end{itemize} 
Intuitively, it is useful to think of $D^{(1)}_v$ is the number of offspring of $v$ that has the same type as the parent of $v$, and $D^{(2)}_v$ is the total number of offspring of $v$ with a type different than that of its parent. Clearly, $\cT^*_\infty$ is a branching process tree, which we claim survives forever, and in particular, its mean offspring number satisfies \begin{align*}
   \mu:=(k-1)\E[\underline{D^*}]\Exp{D^\dagger}>1.\numberthis \label{eq:dominating_supercriticality}
\end{align*} The above claim is a simple consequence of the fact that $\eta<1$, so that the original local limit $\cT_\infty$ survives, and the fact that the number of offspring of any vertex $v$ in $\cT_\infty$ is stochastically dominated by the number of offspring of any vertex in $\cT^*_\infty$ (so that the latter tree also survives).

The coupling of breadth-first exploration on $\G^*_n$ from the random vertex $o_1$ together with the breadth-first exploration on $\cT^*_\infty$ from the root $\varnothing$ (with $\varnothing$ in $\cT^*_\infty$ corresponding to $o_1$ in $\G^*_n$), is very similar to the coupling as discussed in the proof of Theorem \ref{thm:loc_lim} (also see the `\textbf{Restricted exploration from $o_1$}' paragraph in Section \ref{sec:truncation_argument}), and can be formulated analogously. 

Let us denote these exploration clusters at some step $t>0$ respectively by $\G^*_n(t)$ and $\cT^*_\infty(t)$. Recall the concept of \emph{conflict} at some step $t+1$ from the proof of Theorem \ref{thm:loc_lim}; in the present context, this is the first time we either resample a half-edge from $\G^*_n(t)$, or pair a half-edge incident to a vertex already discovered in $\G^*_n(t)$ (thus creating a cycle). 

Then, as in \eqref{eq:stop_time_tau_n}, let us define the stopping time
\begin{align*}
    \tau_n:=\inf\{t\geq 0:\textrm{there is a conflict at step $t$}\}.\numberthis \label{eq:stop_time_B_bound}
\end{align*}
Clearly, $(\G^*_n(t))_{0 \leq t\leq \tau_n}$ can be coupled with $(\cT^*_\infty(t))_{0 \leq t \leq \tau_n}$ such that they are equal. 
We claim that 
\begin{align*}
   \liminf_{n\to \infty} \Prob{\tau_n<\eps n^{1/4}}=O(\eps^2),\numberthis\label{eq:claim_stop_time_tau_n_B}
\end{align*}
as $\eps\to 0$. To see this, on the event $\{\tau_n<\eps n^{1/4}\}$, by a union bound over the different possibilities of $\tau_n=t$ for $t=0,1,\dots,\eps n^{1/4}$,
\begin{align*}
    &\Prob{\tau_n<\eps n^{1/4}}\leq \sum_{t=0}^{\eps n^{1/4}}\frac{t}{\min_{i,j\in [k]}|H_{i \rightarrow j}|-\eps \sqrt{n}}+\sum_{t=0}^{\eps  n^{1/4}}\frac{t\max_{i,j\in [k]}\Delta_n(i,j)}{\min_{i,j\in [k]}|H_{i \rightarrow j}|-\eps \sqrt{n}},\numberthis\label{eq:stop_tau_n_analogous}
\end{align*}
where recall $\Delta_n(i,j)=\max_{v\in \cN^{(i)}}d_v^{(i,j)}$. Clearly the first term above is $o(1)$ as $n \to \infty$ since $|H_{i \rightarrow j}|=\Theta(n)$ for all $i,j\in [k]$ (recall for this proof we assume $(i,j)$ is GOOD for all $i,j\in [k]$). Finally, for the second term, note that by the $r=2$ condition of \eqref{assump:regularity_1_2}, which we assume to be true with $D_{(i,j)}$ replaced by $D^*$ on the right-hand side (recall \eqref{eq:D^*_assump}), we must have 
\begin{align*}
    \max_{i,j\in [k]}\Delta_n(i,j)\leq C\sqrt{n}
\end{align*}
for all large $n$, for a universal constant $C>0$ (e.g., $C$ can be taken as $2\Exp{(D^*)^2}$). Thus, after letting $n\to \infty$, the second term above is $O(\eps^2)$, which proves \eqref{eq:claim_stop_time_tau_n_B}.

Thus, we conclude that on the event $\{\tau_n\geq \eps n^{1/4}\}$, $(\G^*_n(t):0\leq t \leq \eps n^{1/4})$ can be coupled with $(\cT^*_\infty(t):0\leq t \leq \eps n^{1/4})$ such that they are equal. On the other hand, $\cT^*_\infty$ is a supercritical branching process, with mean offspring number as in \eqref{eq:dominating_supercriticality}. In particular, it grows \emph{exponentially fast}: the expected size of generation $l$ in $\cT^*_\infty$ is of order $\mu^l$, for any $l\geq 1$.

What this means is, since in the breadth-first exploration process, we explore the tree $\cT^*_\infty$ generation-by-generation, to explore $\eps n^{1/4}$ many vertices, we must have explored in order $$\log (\eps n^{1/4})=\Theta(\log n)$$ many generations of $\cT^*_n$. In other words, letting $\mathcal{Z}_l$ denote the size of the $l$-th generation of $\cT^*_\infty$, if we define $l(\eps)$ as
\begin{align*}
    Z_0+\dots+Z_{l(\eps)}\leq \eps n^{1/4}\leq Z_0+\dots+Z_{l(\eps)+1},
\end{align*}
then the random variable $l(\eps)$ scales like $\log n$. This can be established, for example, by drawing an analogy of the last display with \eqref{eq:ineq_t(eps)_epssqrtn}, with $t(\eps)$ in \eqref{eq:ineq_t(eps)_epssqrtn} corresponding to $l(\eps)$ above, and mimicking the argument that gives the upper bound on $t(\eps)$ as in \eqref{eq:t_eps_UB} and the lower bound on $t(\eps)$ as in \eqref{eq:t(eps)_LB}. 

We conclude that on the positive probability event $\cE_1=\{\tau_n>\eps n^{1/4}\}$, we can couple breadth-first exploration up to $\eps n^{1/4}$ many vertices starting from $o_1$ in $\G^*_n$ with breadth-first exploration up to these many vertices on $\cT^*_\infty$, such that they are equal, and further, given $\cE_1$, the event $\cE_2$ that the cluster $\G^*_n(\eps n^{1/4})$ contains the ball $B_{\G^*_n}(o_1, K_*\log n)$ for some constant $K_*>0$, has positive probability. Thus, given $\cE_1\cap\cE_2$, since $|B_{\G^*_n}(o_1, K_*\log n)|\leq \eps n^{1/4}=o(n)$, with (conditional) probability at most 
\begin{align*}
    \frac{\eps n^{1/4}}{n-1}=o(1)
\end{align*}
does $o_2$ fall in $B_{\G^*_n}(o_1, K_*\log n)$, so that with (conditional) probability at least $1-o(1)$ one has $d_{\G^*_n}(o_1,o_2)\geq K_* \log n$. This implies the required positive lower bound on the probability of the event $\{d_{\G^*_n}(o_1,o_2)\geq K_* \log n\}$. 
\end{proof}

\section{Discussion}\label{sec:disc}
\subsection{Extracting giants via simple loops}\label{sec:disc_giant}
One of the key technical contributions of our paper is Section \ref{sec:survival}, where we developed a technique that helps us in extracting giant components via breadth- first exploration processes, without the assumption of irreducibility of the limiting branching process. We take a moment to discuss this principle a bit more generally now.

Consider a sequence of (possibly random) graphs $(G_n)_{n \geq 1}$, where $G_n=(V(G_n),E(G_n))$, converging locally (in the sense of Theorem \ref{thm:loc_lim}) to a multi-type branching process tree $\sT_\infty$, with finitely many types. Here let us work with the \emph{standard} definition of such processes, where given a vertex $u$ has type $p$, it gives birth to $J_{(p,q)}$ many vertices of type $q$, in expectation. That is, contrary to our local limit $\cT_\infty$ in Theorem \ref{thm:loc_lim}, we do not assume any dependence on the type of the \emph{parent} of $u$; our setting is more general in the sense that declaring simply $D_{(i,j,l)}=J_{(j,l)}$ for all $i$ in the type space $S$ gives back the more standard setting.

One can (analogous to Definition \ref{def:type graph}) define a graph structure on $S$, by letting a directed edge from $i$ to $j$ to be present whenever $J_{(i,j)}>0$, and in this setting, translating our Theorem \ref{thm:extinct_1dep} from $D_{(i,j,l)}$ to $J_{(j,l)}$, we obtain that $\sT_\infty$ goes extinct if and only if 
\begin{align*}
    \max_{s\in S: p_s>0}\max_{v\in V(s)}\max_{c\in \sL(v)}\sM(c)\leq 1,\numberthis \label{eq:extinct_MTBP}
\end{align*}
which is to be understood as the analogue of \eqref{eq:ext_BP_condition}, where $p_s$ is simply the probability that the root $\varnothing$ of $\sT_\infty$ has type $s$, $V(s)$ is the set of all vertices reachable via a directed path from $s$, and $\sL(v)$ is the set of all directed simple loops $c=(v=v_0,v_1,\dots,v_l=v)$ starting and ending at $v$, where for any such $c$, $\sM(c)=J_{v_0,v_1} J_{v_1,v_2}\dots J_{v_{l-1},v_l}$.

In particular, analogous to our approach in Section \ref{sec:giant_proof}, when $\sT_\infty$ survives, and thus there exists a loop $c$ with $\sM(c)>1$, one can extract out a giant component in $G_n$ via the following strategy:
\begin{itemize}
    \item Keep only those edges $e\in E(G_n)$, where one endpoint of $e$ has type $v_i$ and the other has type $v_{i+1}$ for some $i=0,\dots,l-1$, and call the obtained subgraph $G_n(c)$.
    \item Start exploring in a breadth-first manner from a random vertex $o$ in $G_n(c)$, and show that this exploration process can be coupled to breadth-first exploration of a standard Bienaymé-Galton-Watson branching process tree with mean offspring number $\sM(c)>1$, so that it survives, enabling the component of $o$ in $G_n(c)$ to grow up to a linear size.
\end{itemize}

We expect this general strategy of extracting large components via `simple loops' to find applicability in the analysis of random graphs with a predefined community structure, much like the model we work with in the current paper, giving naturally rise to reducible local limits. Here we work with the case when these graphs within and across communities are uniform graphs with a specified degree sequence. However, even in the case when these (within and across community graphs) are inhomogeneous random graphs \cite{bollobas2007phase} or preferential attachment graphs \cite{barabasi1999emergence}, or even a mixture of different types of random graphs, this strategy is a natural approach to extract giant connected components. 

\subsection{Simple survival criterion for multi-type branching processes}\label{sec:disc_ext} In addition to applicability in extracting giants as discussed in the previous section, we expect the simple criterion \eqref{eq:extinct_MTBP} guaranteeing extinction of a multi-type branching process to be generally useful. Recall that a multi-type branching process with mean offspring matrix $J=(J_{(p,q)})_{p,q\in S}$ (following notation from last section) goes extinct if and only if $\|J\|\leq 1$, where $\|J\|$ denotes the spectral radius (maximum eigenvalue) of $J$, see \cite[Chapter V]{athreya2012branching}. 

There can be situations where computing the spectral radius may be challenging, and for such cases, to guarantee non-extinction, we expect the negation of \eqref{eq:extinct_MTBP} to be a little simpler to verify, than computing $\|J\|$ and checking whether $\|J\|>1$. Namely, one needs only to produce a simple loop $c=(v=v_0,v_1\dots,v_l=v)$ in the type graph with $\sM(c)>1$, such that there is a type $s\in S$ that equals the root type with positive probability, and such that there is a directed path from $s$ to some vertex of the loop $c$.

\subsection{Typical distance lower bound}\label{sec:disc_TD}
Embarrassingly, we cannot prove Theorem \ref{thm:dist} part (3.) without the extra assumption that all pairs $(i,j)$ are $\rm GOOD$. We do not expect this to be at all necessary; intuitively, the present of some BAD pairs should imply the presence of fewer edges as compared to the case when all pairs are GOOD, and thus distances should be \emph{larger} in the former than in the latter case. On the basis of this intuition, we conjecture:
\begin{conjecture}\label{conj:dist}
    Under the assumption of Theorem \ref{thm:loc_lim} we can find $K>0$ such that
    \begin{align*}
        \liminf_{n \to \infty}\Prob{d_{G_n}(o_1,o_2)>K \log n}>0.
    \end{align*}
\end{conjecture}
Consider the following problem.
\begin{problem}\label{prob:dist}
    Consider two graphic degree sequences (or bigraphic bi-degree sequences) $\bd'$ and $\bd$ on the same vertex set such that $\bd'$ dominates $\bd$ pointwise. If $G(\bd')$ and $G(\bd)$ be respectively uniformly sampled random graphs with degree sequences $\bd'$ and $\bd$, is there a coupling such that $G(\bd')$ contains $G(\bd)$ as a subgraph?
\end{problem}
An affirmative answer to this problem leads to a solution of Conjecture \ref{conj:dist} essentially following the approach we outlined in the proof sketch of Theorem \ref{thm:dist} part (3.): one can use the coupling to contain the (within or across partition) subgraphs corresponding to BAD pairs, inside dominating subgraphs corresponding to GOOD pairs, and it is enough to prove the logarithmic distance lower bound for this dominating subgraph. But now since all pairs are GOOD, one can use Corollary \ref{cor:sample_Gn} to sample these graphs using configuration models, and essentially follow the same argument as we have outlined in the proof sketch.

\paragraph{Acknowledgements.} Thanks to Partha S. Dey for looking at a first draft and giving me numerous helpful suggestions.

\bibliographystyle{abbrvnat}
\bibliography{ref}

\appendix

\section{Proof of Lemma \ref{lem:seq_grow}}\label{app_sec:lem_seq_grow}
\begin{proof}
    The proof of the first set of bounds
    \begin{align*}
        \frac{\alpha m^i}{A}\leq \underline{a}_i<\Bar{a}_i\leq \alpha Am^i \numberthis\label{eq:seq_grow_first_set}
    \end{align*}follows the same argument as the proof of \cite[Lemma 4.11]{van2024random}. For the second set of bounds, we only prove the upper bound, as the argument for the lower bound is similar. We claim that uniformly over $i\geq 1$, we can find $B$ sufficiently large such that
    \begin{align*}
        \left(1+\frac{i\beta^\gamma+(m'\alpha B)^\gamma \left(\frac{m^{(i+1)\gamma}-1}{m-1}\right)}{m'\alpha A\left(\frac{m^{i+1}-1}{m-1} \right)} \right)\leq \frac{(m-1)B}{Am},\numberthis \label{eq:appndx_choice_B}
    \end{align*}
    where $A$ is as in \eqref{eq:seq_grow_first_set}. To see this, first note that since $\gamma<1$ and $m>1$, we have an upper bound
    \begin{align*}
        \frac{i\beta^\gamma+(m'\alpha B)^\gamma \left(\frac{m^{(i+1)\gamma}-1}{m-1}\right)}{m'\alpha A\left(\frac{m^{i+1}-1}{m-1} \right)}\leq KB^\gamma
    \end{align*}
    uniformly over $i\geq 1$, for some constant $K>0$ independent of $B$, since the $i\to \infty$ limit of the left-hand side above is $0$. Finally, again since $\gamma<1$, the right-hand side above is at most $\frac{(m-1)B}{Am}$ for $B$ sufficiently large.

    Now, we are in a position to argue for the required upper bound. For an inductive argument, assume the upper bound $$\Bar{b}_l\leq \beta+m'\alpha B m^l$$ holds for all $l=0,1,\dots,i$, where $B$ is as in \eqref{eq:appndx_choice_B} and we want to establish this for $l=i+1$. By definition
    \begin{align*}
        \Bar{b}_{i+1}=\Bar{b}_i+m'\Bar{a}_i+(\Bar{b}_i)^\gamma,
    \end{align*}
    iterating which, we obtain
    \begin{align*}
        \Bar{b}_{i+1}=\beta+m'(\Bar{a}_0+\dots+\Bar{a}_i)+\left((\Bar{b}_0)^\gamma+\dots+(\Bar{b}_i)^\gamma \right).
    \end{align*}
By \eqref{eq:seq_grow_first_set}, we have the upper bound
\begin{align*}
    \Bar{b}_{i+1}&\leq \beta+m'\alpha A\left(\frac{m^{i+1}-1}{m-1} \right)+\left((\Bar{b}_0)^\gamma+\dots+(\Bar{b}_i)^\gamma \right)\\&=\beta+m'\alpha A\left(\frac{m^{i+1}-1}{m-1} \right)\left(1+\frac{(\Bar{b}_0)^\gamma+\dots+(\Bar{b}_i)^\gamma}{m'\alpha A\left(\frac{m^{i+1}-1}{m-1} \right)} \right)\\&\leq \beta+m'\alpha A\left(\frac{m^{i+1}-1}{m-1} \right)\left(1+\frac{i\beta^\gamma+(m'\alpha B)^\gamma \left(\frac{m^{(i+1)\gamma}-1}{m-1}\right)}{m'\alpha A\left(\frac{m^{i+1}-1}{m-1} \right)} \right)\\&\leq \beta+m'\alpha Bm^{i+1},
\end{align*}
concluding the proof of the upper bound, where the last but one step we used the induction hypothesis and in the last step we used \eqref{eq:appndx_choice_B}. \end{proof}

\end{document}